\documentclass[a4paper,11pt,reqno,noindent]{amsart}
\usepackage[centertags]{amsmath}
\usepackage{amsfonts,amssymb,amsthm,dsfont,cases,amscd,esint,enumerate, stmaryrd}
\usepackage[T1]{fontenc}
\usepackage[english]{babel}
\usepackage[applemac]{inputenc}
\usepackage{newlfont}
\usepackage{color}
\usepackage[body={17cm,21.5cm}, top = 1in, bottom = 1.5in, centering]{geometry} 
\usepackage{fancyhdr}
\usepackage{cancel}
\usepackage{enumerate}

\usepackage{mathtools}

\usepackage{pgfplots}
\pgfplotsset{compat=1.15}
\usepackage{mathrsfs}
\usetikzlibrary{arrows}
\usetikzlibrary{patterns}
\usetikzlibrary{intersections, pgfplots.fillbetween}

\usepackage{graphicx}
\usepackage{caption}
\usepackage{subcaption}

\theoremstyle{plain}
\newtheorem{theor}{Theorem}[section]
\newtheorem{lem}[theor]{Lemma}

\theoremstyle{definition}
\newtheorem{examples}[theor]{Examples}
\newtheorem{rems}[theor]{Remarks}
\newtheorem{rem}[theor]{Remark}

\theoremstyle{plain}
\newtheorem*{asn*}{\assumptionnumber}
  \providecommand{\assumptionnumber}{}
  \makeatletter
  \newenvironment{asn}[2]
   {\renewcommand{\assumptionnumber}{Assumption~#1 {\normalfont--- #2}}
    \begin{asn*}
    \protected@edef\@currentlabel{{\normalfont#1}}}
   {\end{asn*}}
  \makeatother

\newtheorem*{assn*}{\assumptionnumber}
  \providecommand{\assumptionnumber}{}
  \makeatletter
  \newenvironment{assn}[1]
   {\renewcommand{\assumptionnumber}{Assumption~#1}
    \begin{assn*}
    \protected@edef\@currentlabel{{\normalfont#1}}}
   {\end{assn*}}
  \makeatother

\numberwithin{equation}{section}

\newcommand{\e}{\varepsilon}

\newcommand{\calF}{\mathcal{F}}

\newcommand{\dist}{\operatorname{dist}}

\newcommand{\R}{\mathbb R}

\newcommand{\calP}{\mathcal P}

\newcommand{\cvf}{\rightharpoonup}

\newcommand{\loc}{{\operatorname{loc}}}

\newcommand{\E}{\mathbb{E}}

\newcommand{\Ld}{{L}}

\newcommand{\diam}{\operatorname{diam}}
\newcommand{\Div}{{\operatorname{div}}}

\newcommand{\step}[1]{\noindent \textit{Step} #1.}
\newcommand{\substep}[1]{\noindent \textit{Substep} #1.}
\newcommand{\Pm}{\mathbb{P}}

\newcommand{\expec}[1]{\mathbb{E}\left[ #1 \right]}

\usepackage[colorlinks,citecolor=black,urlcolor=black]{hyperref}

\title{Homogenization of the stochastic double-porosity model}

\author[E. Bonhomme]{Elise Bonhomme}
\address[Elise Bonhomme]{Universit\'e Libre de Bruxelles, D\'epartement de Math\'ematique, 1050~Brussels, Belgium}
\email{elise.bonhomme@ulb.be}
\author[M. Duerinckx]{Mitia Duerinckx}
\address[Mitia Duerinckx]{Universit\'e Libre de Bruxelles, D\'epartement de Math\'ematique, 1050~Brussels, Belgium}
\email{mitia.duerinckx@ulb.be}
\author[A. Gloria]{Antoine Gloria}
\address[Antoine Gloria]{Sorbonne Universit\'e, Universit\'e Paris Cit\'e, CNRS, Laboratoire Jacques-Louis Lions, LJLL, F-75005 Paris, France  \& Universit\'e Libre de Bruxelles, D\'epartement de Math\'ematique, 1050~Brussels, Belgium}
\email{antoine.gloria@sorbonne-universite.fr}

\begin{document}

\begin{abstract}
This work is devoted to the homogenization of elliptic equations in high-contrast media in the so-called `double-porosity' resonant regime,
for which we solve two open problems of the literature. First, we prove qualitative stochastic homogenization under very weak conditions, which cover the case of inclusions that are not uniformly bounded or separated. Second, under stronger assumptions, we provide sharp error estimates for the two-scale expansion. The main difficulty is related to the loss of integrability of the control in the resonant zones.
\end{abstract}

\maketitle
\setcounter{tocdepth}{1}
\tableofcontents

\section{Introduction}
This paper is concerned with the homogenization of the so-called ``double-porosity'' problem, which is a standard averaged mesoscopic model used in engineering to describe flows in fractured porous media (see e.g.~\cite{Hornung-97}).
The model takes the form of a parabolic problem in a medium described by a connected matrix punctured by a dense array of small soft inclusions --- in the specific resonant regime when the conductivity inside inclusions scales like the square of the typical size of the inclusions.
This specific high-contrast regime leads to resonant phenomena with unusual micro-macro scale interactions and memory effects, which have served as a basis in the design of various metamaterials (see e.g.~\cite{Bouchitte-15} and references therein).
Upon Laplace transform, the parabolic equation reduces to the corresponding {\it massive} elliptic problem: in a given domain $D\subset\R^d$, given a forcing $f\in\Ld^2(\R^d)$, denoting by $F_\e(D)$ the set of inclusions of size $\e$ in $D$, we consider in this article the solution~$u_\e$ of the problem
\begin{equation}\label{eq:double-por-00}
\left\{\begin{array}{ll}
u_\e -\nabla\cdot (\mathds1_{D\setminus F_\e(D)}+\e^2\mathds1_{F_\e(D)})\nabla u_\e = f,&\qquad\text{in $D$},\\[1mm]
u_\e=0.&\qquad\text{on $\partial D$}.\end{array}\right.
\end{equation}
Let us recall the state of the art on this problem.
Qualitative homogenization of this high-contrast model was first established by Arbogast, Douglas, and Hornung~\cite{Arbogast} in the periodic setting using an approach that is the precursor of the periodic unfolding method (see also~\cite{Smyshlyaev0} for finer assumptions).
It was soon after reproved by Allaire~\cite{Allaire-92} using two-scale convergence, and the corresponding stochastic setting was first treated by Bourgeat,~Mikeli\'{c}, and~Piatnitski~\cite{MR1993376} using stochastic two-scale convergence (see also some refinements in~\cite{MR3902123}).
An alternative variational approach by $\Gamma$-convergence in the periodic setting was developed by Braides, Chiad\`o Piat, and Piatnitski~\cite{Braides-Piatnitski-04}.
There is also a very large literature on the spectral behavior of high-contrast operators, see for instance~\cite{Zhikov-0,Smyshlyaev1,Smyshlyaev,CCV-23,CCV-21}, and the references therein. In the periodic setting, error bounds in $\Ld^2$ are implicitly contained in the work of Zhikov~\cite{Zhikov-00}, which were more recently improved in the form of sharp resolvent estimates in~\cite{Cherednichenko-Cooper-16,Cherednichenko-Ershova-Kiselev-20,Smyshlyaev}.

Two main questions have remained open in the field. First, whereas homogenization for the nonresonant version of the problem (the so-called soft-inclusion problem\footnote{that is, \eqref{eq:double-por-00} with conductivity $\e^2$ replaced by $0$ inside the inclusions.}) holds under weak geometric assumptions on the inclusions, see e.g.~\cite[Chapter~8]{JKO94}, all works on double porosity assume inclusions to be both uniformly bounded and uniformly separated from one another.
Second, the only quantitative convergence rates available are proved using Floquet theory, and therefore only hold in the periodic setting on the whole space $\R^d$ --- and not on bounded domains or for random inclusions.
In the present article, we give a definite answer to both questions. 
\begin{enumerate}[---]
\item \emph{Qualitative homogenization under weak geometric assumptions:}\\
We prove stochastic homogenization of the double-porosity problem under (almost) the same assumptions as those needed for the homogenization of soft inclusions, see Theorem~\ref{th:main}, and we also prove a general qualitative 
corrector result, see Theorem~\ref{th:corr-qual}.
\smallskip\item \emph{Quantitative error estimates:}\\
We establish new and optimal error estimates for the two-scale expansion of the double-porosity problem (both for the periodic and random settings, both on bounded domains and in the whole space), see Theorem~\ref{th:quant}.
\end{enumerate} 
Although we focus on the scalar case for notational simplicity,  all our results hold for systems (truncation arguments that we use in the proof then need to be performed componentwise).

Last, we also provide a rather extensive survey of extension results for gradient fields (with detailed proofs, see Lemma~\ref{lem:ass1} and Appendix~\ref{sec:app1}). Such results are a key tool for the homogenization of problems with inclusions, but they are scattered in the literature and part of the homogenization folklore. We believe this extensive survey, also including some new results, can serve as a good reference for future work in the field.

\subsection{Qualitative theory}
Let $F=\cup_nI_n\subset \R^d$ be a random ensemble of inclusions, which we assume throughout this work to satisfy the following general assumption.

\begin{assn}{H$_0$}\label{ass0}
The random set $F=\cup_nI_n\subset \R^d$ is a stationary ergodic random inclusion process,\footnote{More precisely, on a given a probability space $(\Omega,\Pm)$, we consider a collection $\{I_n\}_n$ of maps $I_n:\Omega\to\mathcal O(\R^d)$, where $\mathcal O(\R^d)$ stands for the collection of open subsets of $\R^d$, and we require that for all $n$ the indicator function $\mathds1_{I_n}$ is measurable on the product space $\Omega\times\R^d$. Stationarity of $F=\cup_nI_n$ then means that the finite-dimensional laws of the random field $\mathds1_{x+F}$ do not depend on the shift $x\in\R^d$. Ergodicity means that, if a function is measurable with respect to $\mathds1_F$ and is almost surely unchanged when $F$ is replaced by $x+F$ for any $x\in\R^d$, then the function is almost surely constant.}
and it satisfies the following:
\begin{enumerate}[$\bullet$]
\item The inclusions $I_n$'s are almost surely disjoint, open, connected, bounded sets, with Lipschitz boundary, and the complement~\mbox{$\R^d\setminus F$} is almost surely connected.
\smallskip\item The random set $F$ is nontrivial in the sense that $\E[\mathds1_F]<1$.
\end{enumerate}
\end{assn}

Establishing homogenization in the case of soft inclusions requires to extend functions defined on $\R^d\setminus F$ into functions defined on $\R^d$. Except under very restrictive assumptions on the inclusion process, the extension operator cannot be bounded in~$H^1$ in general, and we need to cope with a possible loss of integrability, see e.g.~\cite{Zhikov-86,Zhikov-90} or~\cite[Lemma~8.22]{JKO94}.
For the double-porosity problem, on top of this extension property, one further needs to solve resonant cell problems inside the inclusions. If there is a loss in the integrability exponent in the extension property, energy estimates are not enough to control the solutions of such cell problems: instead, suitable $\Ld^q$ regularity estimates are needed, which are only known to hold under suitable regularity assumptions on the inclusions. In brief, our analysis requires the following two properties:
\begin{enumerate}[{H$_2$:}]
\item[{H$_1$:}] extension property for functions defined outside the inclusions, with a possible loss in the integrability exponent;
\smallskip\item[{H$_2$:}] suitable $L^q$ elliptic regularity estimates in each inclusion, for some exponent $q$ depending on the loss of integrability in~{H$_1$}.
\end{enumerate}
Instead of making specific geometric assumptions that ensure the validity of both properties,
we take a more abstract point of view and formulate the exact properties that are needed for our homogenization result to hold. Their validities are separate probability and PDE questions that are discussed in detail in Section~\ref{sec:validity-ass} below.
Given some $1<p\le2\le q<\infty$, we consider the following two assumptions:
\begin{asn}{H$_1(p)$}{Extension property}\label{ass1}
For all balls $B\subset\R^d$ and all $\e>0$ small enough (only depending on~$B$),
there exists an extension operator $P_{B,\e}:H^1_0(B)\to W^{1,p}_0(2B)$ such that for all $u\in H^1_0(B)$ the extension $P_{B,\e} u\in W^{1,p}_0(2B)$ satisfies $P_{B,\e} u=u$ in $B\setminus \e F$ and
\begin{equation}\label{eq:ass1-1}
\|\nabla P_{B,\e} u\|_{\Ld^{p}(2B)}\,\le\,C_B\|\nabla u\|_{\Ld^2(B\setminus\e F)},
\end{equation}
for some constant $C_B$ only depending on $d,B$. In addition, there exists an extension operator $P:H^1_\loc(\R^d)\to W^{1,p}_\loc(\R^d)$ such that for all $u\in H^1_\loc(\R^d)$ the extension $Pu\in W^{1,p}_\loc(\R^d)$ satisfies $Pu=u$ in~$\R^d\setminus F$ and such that, for any random field $u\in L^2(\Omega;H^1_\loc(\R^d))$ with $(u,\{I_n\}_n)$ jointly stationary, the extension $Pu$ is also stationary and satisfies
\begin{equation}\label{eq:ass1-2}
\|\nabla Pu\|_{\Ld^p(\Omega)}\,\le\,C\|\mathds1_{\R^d\setminus F}\nabla u\|_{\Ld^2(\Omega)},
\end{equation}
for some constant $C$ only depending on $d$.
\end{asn}

\begin{asn}{H$_2(q)$}{Elliptic regularity}\label{ass2}
There exists a constant $C_0$ such that the following holds:\footnote{Note that this property is required to hold for any {\it unit} ball $B$ and is thus not scale-invariant: this role of scale~$O(1)$ originates in the~$O(1)$ massive term in the double-porosity model~\eqref{eq:double-por-00} that we consider.}
for all $n$, for all unit balls $B$, if $g\in L^q(I_n\cap 2B)^d$ and $u\in H^1_0(I_n)$ satisfy the following relation in the weak sense,
\[u-\triangle u=\Div(g),\qquad\text{in $I_n\cap 2B$},\]
then we have the local estimate
\begin{equation*}
\|\nabla u\|_{\Ld^{q}(I_n\cap B)}\le C_0\Big(\|g\|_{\Ld^{q}(I_n\cap 2B)}+\|\nabla u\|_{\Ld^2(I_n \cap 2B)}\Big).
\end{equation*}
\end{asn}

As elliptic regularity inside inclusions is only needed to compensate for the loss in the integrability exponent in the extension property, we need Assumption~\ref{ass2} with $q=p'$ if Assumption~\ref{ass1} holds. Our first qualitative result now reads as follows.

\begin{theor}[Qualitative homogenization]\label{th:main}
Let the random inclusion process $F=\cup_nI_n\subset \R^d$ satisfy Assumptions~\ref{ass0}, \ref{ass1}, and~\ref{ass2} for some $1<p \le 2$ and $q=p'$.
Given a bounded Lipschitz domain $D\subset\R^d$, for all~$\e>0$, let $F_\e(D):=\cup\{\e I_n:\e I_n\subset D\}$, let $\chi_\e:=\mathds1_{F_\e(D)}$, let $f\in \Ld^2(D)$, and consider the solution $u_\e\in H^1_0(D)$ of the double-porosity problem
\begin{equation}\label{eq:double-por-eqn}
u_\e-\nabla\cdot(1-\chi_\e+\e^2\chi_\e)\nabla u_\e=f,\qquad\text{in $D$}.
\end{equation}
Then we have almost surely
\begin{equation}\label{eq:res-weak-conv}
u_\e\cvf{}\bar u+\E[v](f-\bar u),\qquad (1-\chi_\e)\nabla u_\e\cvf{}\bar a\nabla\bar u,\qquad\text{weakly in $\Ld^2(D)$},
\end{equation}
where:
\begin{enumerate}[---]
\item $v\in H^1_0(F)$ is the unique almost sure weak solution of
\begin{equation}\label{e.999}
v-\triangle v=1,\qquad\text{in $F$},
\end{equation}
which is a stationary random field and satisfies $0<\E[v]<1$;
\smallskip\item$\bar u\in H^1_0(D)$ is the unique solution of the homogenized problem
\begin{equation}\label{eq:homog-prob}
(1-\E[v])\bar u-\nabla\cdot\bar a\nabla\bar u=(1-\E[v])f,\qquad\text{in $D$},
\end{equation}
where $\bar a$ is the homogenized matrix for the soft-inclusion problem, that is, the symmetric positive-definite matrix given by the following cell formula, for all $\xi\in \R^d$,
\begin{equation}\label{eq:def-bara}
\xi\cdot\bar a \xi \,:=\,\inf\Big\{\E[\mathds1_{\R^d\setminus F}|\xi+\nabla\gamma|^2]\,:\,\gamma\in H^1_\loc(\R^d;\Ld^2(\Omega)),\,\text{$\nabla\gamma$ stationary},\,\E[\nabla\gamma]=0\Big\}.
\end{equation} 
\end{enumerate}
\end{theor}

We face two main difficulties to prove Theorem~\ref{th:main}: a compactness issue due to the loss of integrability in the extension operators (preventing us from using two-scale compactness, which was the only strategy available so far in the random setting) and a control of correctors that is too weak for a direct use of Tartar's method of oscillating test functions. Compactness will be obtained the hard way by elliptic regularity, whereas correctors will be replaced by suitable approximations (the analysis of which turns out to be quite challenging) and combined with truncation arguments to run the oscillating test-functions approach.

\medskip
As usual in homogenization theory, in order to capture leading-order oscillations around the weak convergences in~\eqref{eq:res-weak-conv}, we need to introduce suitable correctors. Here, the relevant correctors are those associated with the corresponding soft-inclusion problem, that is, the minimizers of the cell problem~\eqref{eq:def-bara}. We refer to~\cite[Chapter~8]{JKO94} for their existence and uniqueness under the extension assumption~\ref{ass1}.

\begin{lem}[Correctors; see~\cite{JKO94}]\label{lem:corr0}
Let the random inclusion process $F=\cup_nI_n\subset \R^d$ satisfy Assumptions~\ref{ass0} and~\ref{ass1} with $1<p\le2$.
For $1\le i\le d$, there exists a random field $\varphi_i\in \Ld^p(\Omega;W^{1,p}_\loc(\R^d))\linebreak\cap\Ld^2(\Omega;H^1_\loc(\R^d\setminus F))$ that solves the variational problem~\eqref{eq:def-bara} in the direction $\xi= e_i$, that is, an almost sure weak solution of the corrector problem
\begin{equation}\label{eq:corr}
-\nabla\cdot(\mathds1_{\R^d\setminus F}(e_i+\nabla\varphi_i))=0,\qquad\text{in $\R^d$},
\end{equation}
such that $\nabla\varphi_i$ is stationary and $\E[\nabla\varphi_i]=0$. The restriction $\mathds1_{\R^d\setminus F}\nabla\varphi_i$ is defined uniquely and for all $\xi\in\R^d$ the homogenized matrix defined in~\eqref{eq:def-bara} coincides with
\begin{equation}\label{eq:def-bara-re}
\xi\cdot\bar a \xi\,=\,\E[\mathds1_{\R^d\setminus F}|\xi+\nabla\varphi_\xi|^2]\,=\,\xi\cdot\E[\mathds1_{\R^d\setminus F}(\xi+\nabla\varphi_\xi)],
\end{equation}
where we have set $\varphi_\xi:=\sum_{i=1}^d\xi_i\varphi_i$.
\end{lem}

Next to the above qualitative homogenization result, we show that oscillations around the weak convergences in~\eqref{eq:res-weak-conv} are precisely captured in terms of these correctors.
Henceforth, we use Einstein's summation convention on repeated indices.

\begin{theor}[Corrector result]\label{th:corr-qual}
With the same assumptions and notations as in Theorem~\ref{th:main}, further assuming that the homogenized solution $\bar u$ satisfies $\nabla\bar u\in \Ld^\infty(D)$, we have almost surely
\begin{equation}\label{eq:res-st-conv}
\begin{array}{rll}
u_\e-\bar u -v(\tfrac\cdot\e)(f-\bar u)&\to&0,\\[1mm]
(1-\chi_\e)\big(\nabla u_\e-(e_i+\nabla\varphi_i)(\tfrac\cdot\e)\nabla_i \bar u_\e \big)&\to&0,\qquad\text{strongly in $L^2(D)$},
\end{array}
\end{equation}
where $\varphi_i$ is the corrector defined in Lemma~\ref{lem:corr0} above and where $\bar u_\e$ is any approximation\footnote{One possible choice is $\bar u_\e = \zeta_\e \ast (\bar u \mathds 1_D)$ with $\zeta_\e:=\e^{-d} \zeta(\frac \cdot \e)$ for some $\zeta \in C^\infty_c(\R^d)$ with $\int_{\R^d}\zeta=1$. If $\nabla \bar u \in C^0_b(D)$, then one can actually take $\bar u_\e=\bar u$ in~\eqref{eq:res-st-conv}.} of~$\bar u$ in~$H^1(D)$ such that $\sup_\e (\|\nabla \bar u_\e\|_{\Ld^\infty(D)}+\e \|\nabla^2 \bar u_\e\|_{\Ld^\infty(D)})<\infty$.
\end{theor}

\begin{rem}[Extension to oscillatory data]\label{rem:oscf}
The proof of the above results is easily adapted to the more general case of a right-hand side of the form $f_\e=f(x,\frac x\e,\omega)$ with $f(x,\cdot)$ stationary for all $x$ and with $\|f_\e\|_{L^2(D)}\lesssim1$. In that setting, our results are modified as follows: almost surely,
\begin{equation*}
\begin{array}{rlll}
u_\e-(1-\E[v])\bar u-\E[v_f]&\cvf&0,\qquad&\text{weakly in $L^2(D)$},\\[1mm]
u_\e-(1-v(\tfrac\cdot\e))\bar u-v_f(\tfrac\cdot\e)&\to&0,\qquad&\text{strongly in $L^2(D)$},\\[1mm]
(1-\chi_\e)\big(\nabla u_\e-(e_i+\nabla\varphi_i)(\tfrac\cdot\e)\nabla_i\bar u\big)&\to&0,\qquad&\text{strongly in $L^2(D)$},
\end{array}
\end{equation*}
where:
\begin{enumerate}[---]
\item on top of $v$ defined as in Theorem~\ref{th:main}, for all $x\in\R^d$, we define $v_f(x,\cdot)$ as the unique stationary random field that is the almost sure weak solution of
\[v_f(x,\cdot)-\triangle v_f(x,\cdot)=f(x,\cdot),\qquad\text{in $F$};\]
\item $\bar u\in H^1_0(D)$ is now the unique solution of the modified homogenized problem
\[(1-\E[v])\bar u-\nabla\cdot\bar a\nabla\bar u=\E[f-v_f],\qquad\text{in $D$}.\]
\end{enumerate}
This extension of our results for oscillatory data provides in particular a qualitative two-scale analysis of resolvents, which can be used as a starting point for questions related to the convergence of the spectrum as studied e.g.~in~\cite{Smyshlyaev1,CCV-23,CCV-21}. We do not pursue in that direction here.
\end{rem}

\subsection{Quantitative error estimates}
We now turn to the validity of quantitative error estimates for homogenization.
For that purpose, we shall need to strengthen our assumptions a little: we assume that the extension property~\ref{ass1} holds without loss of integrability, that is, with $p=2$, and we add to this a uniform boundedness requirement.

\begin{assn}{H$_3$}\label{ass3}
The random inclusion process $F=\cup_nI_n\subset \R^d$ satisfies~\ref{ass0} as well as the following:
\begin{enumerate}[$\bullet$]
\item\emph{Extension property:} \ref{ass1} holds with $p=2$, and for all balls $B$ and all $\e>0$ the extension operator~$P_{B,\e}$ further satisfies the $L^2$ bound $\|P_{B,\e}u\|_{L^2(2B)}\le C_B\|u\|_{L^2(B)}$.\footnote{Note that this $L^2$ control is obtained for free in all the situations covered in Lemma~\ref{lem:ass1} for which there is no loss of integrability $p=2$.}
\smallskip\item\emph{Uniform Poincar\'e constant:} $\sup_n\calP_2(I_n)<\infty$ almost surely, where $\calP_2(I_n)$ denotes the Poincar\'e--Wirtinger constant in $H^1(I_n)/\R$.
\end{enumerate}
\end{assn}

In order to get accurate error estimates,
next to the correctors defined in Lemma~\ref{lem:corr0} above,
we further introduce the associated flux correctors, as well as new so-called `inclusion correctors' similar to~\cite{bernou2023homogenization}, for which existence and uniqueness are standard, see~e.g.~\cite{MR4103433,bernou2023homogenization}. For our application, we focus on the case $p=2$ in the statement.

\begin{lem}[Flux and inclusion correctors; e.g.~\cite{MR4103433,bernou2023homogenization}]\label{lem:corr}
Let the random inclusion process $F=\cup_nI_n\subset \R^d$ satisfy Assumptions~\ref{ass0} and~\ref{ass1} with $p=2$ (say).
Next to the correctors $\varphi=\{\varphi_i\}_{1\le i\le d}$ with $\varphi_i\in H^1_\loc(\R^d;\Ld^2(\Omega))$ as defined in Lemma~\ref{lem:corr0}, we may define the flux correctors $\sigma=\{\sigma_{ijk}\}_{1\le i,j,k\le d}$ and the inclusion correctors $\theta=\{\theta_i\}_{1\le i\le d}$ as follows:
\begin{enumerate}[---]
\item For $1\le i,j,k\le d$, we define the flux corrector $\sigma_{ijk}\in H^1_\loc(\R^d;\Ld^2(\Omega))$ as the unique almost sure weak solution of
\[-\triangle \sigma_{ijk} = \nabla_j (q_{i})_k-\nabla_k (q_{i})_j,\]
in terms of $q_i:=\mathds1_{\R^d\setminus F}(e_i+\nabla \varphi_i)-\bar a e_i$, such that $\nabla\sigma_{ijk}$ is stationary, $\E[\nabla\sigma_{ijk}]=0$, $\E[|\nabla\sigma_{ijk}|^2] <\infty$, and $\int_{B(0,1)}\sigma_{ijk}=0$ (say). Note that the definition of $\bar a$ ensures $\E[q_i]=0$ and that the flux corrector~$\sigma_{ijk}$ satisfies
\[\sigma_{ijk}=-\sigma_{ikj},\qquad\nabla_k\sigma_{ijk}=(q_i)_j.\]
\item For $1\le i\le d$, we define the inclusion corrector $\theta_i\in H^1_\loc(\R^d;\Ld^2(\Omega))$ as the unique almost sure weak solution of
\[\triangle \theta_i \,=\, \nabla_iv,\] 
where we recall that $v$ is defined in~\eqref{e.999},
such that $\nabla\theta_i$ is stationary, $\E[\nabla\theta_i]=0$, $\E[|\nabla\theta_i|^2] <\infty$, and $\int_{B(0,1)}\theta_i=0$ (say). Note that it satisfies
\[\nabla\cdot\theta\,=\,v-\E[v].\]
\end{enumerate}
\end{lem}

We show that the quantitative homogenization of the double-porosity problem depends on the quantitative sublinearity of these correctors.
For shortness, in the statement below, we focus on the typical setting when those correctors have bounded moments: this is always the case in the periodic setting by Poincar\'e's inequality, whereas in the random setting we refer e.g.~to~\cite{GNO-quant,GO15,AKM-book,DG-22pol,bernou2023homogenization} for similar corrector bounds under suitable mixing assumptions (see~\cite{bernou2023homogenization} for the inclusion corrector $\theta$); we also refer to the recent work~\cite{bella2025} where the strategy of~\cite{GNO-quant} is adapted to soft inclusions.
The following quantitative homogenization result is new even in the periodic setting (for which we can simply drop the expectations).

\begin{theor}\label{th:quant}
Let the random inclusion process $F=\cup_nI_n\subset \R^d$ satisfy Assumption~\ref{ass3}, and further assume that the correctors~$\varphi,\sigma,\theta$ satisfy
\begin{equation}\label{eq:moment bounds}
    \textstyle\sup_x\E\big[|\varphi(x)|^2+|\sigma(x)|^2+ |\theta(x)|^2\big]\,<\,\infty.
\end{equation}
Given a bounded Lipschitz domain $D\subset\R^d$ and $f\in H^1(D)$, for all $\e>0$, 
let $u_\e\in H^1_0(D)$ be the solution of the double-porosity problem~\eqref{eq:double-por-eqn},
and let $\bar u\in H^1_0(D)$ be the solution of the corresponding homogenized problem~\eqref{eq:homog-prob}. If the latter satisfies
\begin{equation}\label{eq:hyp bar u}
\nabla^2\bar u\in L^2(D)\quad\text{and}\quad\nabla\bar u\in L^\infty(D),
\end{equation}
then we have
\begin{equation}\label{e:quant2s}
\E\Big[\|u_\e-\bar u-\e\varphi_i(\tfrac\cdot\e)\nabla_i\bar u\|_{H^1(D\setminus F_\e(D))}^2\Big]^\frac12+\E\Big[\|u_\e-\bar u -v(\tfrac\cdot\e)(f-\bar u)\|_{\Ld^2(D)}^2\Big]^\frac12\,\lesssim_{f,\bar u}\,\sqrt\e,
\end{equation}
where we use the same notation for $F_\e(D)$ as in the statement of Theorem~\ref{th:main} and where we recall that~$v$ is defined in~\eqref{e.999}.
\end{theor}

\begin{rems}$ $
\begin{enumerate}[(a)]
\item The regularity condition~\eqref{eq:hyp bar u} is satisfied if $f$ and the domain $D$ are smooth enough. In particular, in dimension $1 \leq d \leq 3$, it is enough to consider $f\in H^1(D)$ and a domain $D$ of class~$C^1$: indeed, we then find $\nabla \bar u \in H^2(D)$, hence $\nabla \bar u \in L^\infty(D)$ by Sobolev embedding.
\smallskip\item In~\eqref{e:quant2s}, the error bound is $O(\sqrt\e)$ instead of $O(\e)$ due to boundary layers in the bounded domain~$D$. If $D$ is replaced by~$\R^d$ (or if $\bar u$ is compactly supported in $D$), then there is no boundary layer and the error bound can be strengthened to
\begin{equation*} 
\E\Big[\|u_\e-\bar u-\e \varphi_i(\tfrac\cdot\e)\nabla_i\bar u\|_{H^1(\R^d\setminus \e F)}^2\Big]+\E\Big[\|u_\e-\bar u - v(\tfrac\cdot\e)(f-\bar u)\|_{\Ld^2(\R^d)}^2\Big]^\frac12\,\lesssim \, \e\|f\|_{H^1(\R^d)}. 
\end{equation*}
Variants of two-scale expansions based on Floquet--Bloch theory (and thus restricted to the periodic setting on $\R^d$) allow to weaken the regularity condition on $f$ in this estimate; see the recent independent work~\cite[Section~7]{Smyshlyaev}. In the present setting, this would amount to stepping back a little and replacing the homogenized equation for $\bar u$ by the following effective, yet coupled, two-scale system, for $(\bar u_\e,w_\e)\in H^1(\R^d)\times H^1_0(\e F)$,
\begin{equation}\label{e.hom-coupled}
\left\{
\begin{array}{rcll}
\bar u_\e+w_\e-\nabla \cdot \bar a \nabla \bar u_\e &=&f,& \text{in $\R^d$},
\\
\bar u_\e+w_\e-\e^2 \triangle w_\e&=&f,\quad& \text{in $\e F$}.
\end{array}
\right.
\end{equation}
Note that the solution $(\bar u_\e,w_\e)$ is uniformly bounded in $H^1(\R^d)\times \Ld^2(\e F)$. In these terms, the proof of Theorem~\ref{th:quant} simplifies and yields
\begin{equation*} 
\E\Big[\|u_\e-\bar u_\e -\e\varphi_i(\tfrac\cdot\e)\nabla_i\bar u_\e \|_{H^1(\R^d\setminus \e F)}^2\Big]^\frac12+\E\Big[\|u_\e-\bar u_\e -w_\e\|_{\Ld^2(\R^d)}^2\Big]^\frac12\,\lesssim\,\e \|f\|_{\Ld^2(\R^d)},
\end{equation*}
where only the $\Ld^2$ norm of $f$ appears (instead of $H^1$). This extends~\cite[Theorem~7.8]{Smyshlyaev} to the random setting.
Note that the two-scale operator in~\eqref{e.hom-coupled} is the natural candidate to describe the spectrum of $1-\nabla \cdot (1-\chi_\e+\e^2\chi_\e)\nabla$.
\end{enumerate}
\end{rems}

\subsection{Discussion of the assumptions}\label{sec:validity-ass}
We turn to a detailed discussion of the validity of our two main assumptions~\ref{ass1} and~\ref{ass2}.
These have been largely discussed in the literature, in particular in the context of homogenization of the soft-inclusion problem. Here, we review previous results, adapt them to the present setting, and include some extensions and new cases. They illustrate the much wider applicability of Theorem~\ref{th:main} with respect to previous results of the literature.

We start with the extension assumption~\ref{ass1}, for which we mainly build upon previous work of Zhikov~\cite{Zhikov-86,Zhikov-90,Zhikov-90b}
(see also more recently~\cite{CCV-21,Heida-23}).
Several sufficient conditions are listed in the following lemma, for which
proofs and detailed references
are postponed to Appendix~\ref{sec:app}.
As this lemma shows, the validity of~\ref{ass1} depends in a subtle way on the geometric properties of the inclusions.
First, as stated in item~(i), we note that the separation distance between the inclusions needs in general to be large enough depending on their diameters: large empty space is needed around large inclusions. Up to losing some integrability,
this can be relaxed into a moment condition, cf.~(ii).
Measuring the size of the inclusions in terms of their diameters, however, is not optimal: in case of strongly anisotropic inclusions such as cylinders, we show in~(iii) that the suitable notion of size is given by the radii instead of the diameters --- thus strongly weakening the separation condition. We stick to cylindrical inclusions for shortness and leave easy generalizations of this result to the reader.
In~(iv), we state that no separation is actually required at all in the special case of uniformly convex inclusions (provided they are of comparable sizes, say). Finally, in a different perspective, we include as is in~(v) a result due to Zhikov~\cite{Zhikov-90} on percolation clusters in random chess structure.

\begin{lem}[Validity of~\ref{ass1}]\label{lem:ass1}
Let the random inclusion process $F=\cup_nI_n\subset \R^d$ satisfy~\ref{ass0}.

\begin{enumerate}[(i)]
\item \emph{General inclusions with uniform separation:}\\
Assume that there is a constant $C_0$ such that for all $n$ we have
\begin{equation}\label{eq:nun-infty}
\min_{m:m\ne n}\dist(I_n,I_m)\ge \tfrac1{C_0} \diam (I_n),\qquad\text{almost surely}.
\end{equation}
In addition, assume that the rescaled inclusions $I_n'=\diam(I_n)^{-1}I_n$ are uniformly Lipschitz in the following sense:\footnote{This regularity assumption coincides with (a suitable uniform version of) Stein's ``minimal smoothness'' assumption in~\cite[Chapter~VI, Section~3.3]{Stein-70} (see also~\cite[Theorem~3.8]{CCV-21}).} for all $n$ there is almost surely a collection of balls~$\{D_i^n\}_i$ covering $\partial I_n'$ such that
\begin{enumerate}[\quad---]
\item for all $i$, in some orthonormal frame, $D_i^n\cap\partial I_n'$ is the graph of a Lipschitz function with Lipschitz constant bounded by $C_0$;
\item for all $x \in \partial I_n'$ we have $B(x,\frac1{C_0}) \subset D_i^n$ for some $i$;
\item $\sup_i\sharp\{j:D_i^n\cap D_j^n\ne\varnothing\}\le C_0$.
\end{enumerate}\smallskip
Then~\ref{ass1} holds for all~$1 \leq p \le2$.

\medskip
\item \emph{General inclusions with moment condition on separation:}\\
For all $n$, consider the characteristic separation length
\begin{equation}\label{eq:def-nun}
\nu_n \,:=\, \tfrac{\rho_n}{D_n}\wedge1,\qquad\text{in terms of}\quad
\rho_n\,:=\,\min_{m:m\ne n}\dist(I_n,I_m),\quad
D_n\,:=\,\diam(I_n),
\end{equation}
and assume that the following moment bound holds, for some $\alpha\ge1$,
\begin{equation}\label{eq:cond-dist}
\limsup_{R\uparrow\infty}\frac1{|B(0,R)|}\sum_{n:I_n\cap B(0,R)\ne\varnothing} \nu_n^{-\alpha}\,<\,\infty,\qquad\text{almost surely}.
\end{equation}
In addition, assume for convenience the following strengthened form of uniform Lipschitz condition for the rescaled inclusions $I_n'=\diam(I_n)^{-1}I_n$:
there is a constant $C_0$ and for all $n$ there is a Lipschitz homeomorphism $\phi_n':B(0,2)\to\R^d$ (onto its image) such that $\phi_n'(B(0,1))=I_n'$ and $\tfrac1{C_0}\le\|\nabla\phi_n'\|_{L^\infty}\le C_0$.
Then~\ref{ass1} holds for all $1\le p\le\frac{2\alpha}{1+\alpha}$.

\medskip
\item \emph{Anisotropic inclusions with weaker moment condition on separation:}\\
Assume that each inclusion $I_n$ is the isometric image of a cylinder $B'(0,\delta_n)\times(-L_n,L_n)$, with random width $\delta_n$ and random length $L_n$, where $B'(0,\delta_n):=\{x'\in\R^{d-1}:|x'|<\delta_n\}$.
In this anisotropic setting, for all $n$, consider the modified characteristic separation length\footnote{Note that a slight modification of the proof would actually even allow to replace the ratio $\rho_n/\delta_n$ by $\rho_n/(\delta_n\wedge L_n)$, which is interesting in case of flat cylinders with $L_n\ll\delta_n$.}
\[\mu_n\,:=\,\tfrac{\rho_n}{\delta_n}\wedge1,\qquad\text{in terms of $\rho_n:=\min_{m:m\ne n}\dist(I_n,I_m)$,}\]
and, instead of~\eqref{eq:cond-dist}, assume that the following moment bound holds, for some $\beta \ge1$,
\begin{equation}\label{eq:cond-dist cylinders}
\limsup_{R\uparrow\infty}\frac1{|B(0,R)|}\sum_{n:I_n\cap B(0,R)\ne\varnothing} \mu_n^{-\beta}\,<\,\infty,\qquad\text{almost surely}.
\end{equation}
Then~\ref{ass1} holds for all~$1 \le p \le \frac{2 \beta}{1+\beta}$ (without any assumption on the random lengths!).

\medskip
\item \emph{Strictly convex inclusions without separation condition:}\\
Assume that each inclusion $I_n$ is strictly convex and of class~$C^2$, and assume that they satisfy the following uniform condition on the ratio of principal curvatures: there is a constant $C_0$ and for all $n$ there are radii $0<r_1^n<r_2^n<\infty$ such that ${r_2^n}/{r_1^n}\le C_0$ and such that for any boundary point $x \in \partial I_n$ there exist two balls~$B_{1}^n$ and~$B_{2}^n$, with radii~$r_1^n$ and~$r_2^n$ respectively, so that
\[B_1^n\subset I_n\subset B_2^n,\qquad \partial B_1^n\cap\partial I_n=\{x\}=\partial B_2^n\cap\partial I_n.\]
Further assume that no inclusion is surrounded by much smaller ones: for all $n$,
\begin{equation}\label{eq:closesmallincl}
\dist(I_n,I_m)\le\tfrac1{C_0}r^n_1~~~\Longrightarrow~~~ r^m_1\ge\tfrac1{C_0} r^n_1.
\end{equation}
Then~\ref{ass1} holds for all $1 \leq p<2\frac{d+1}{d+3}$ (without any separation condition!).\footnote{In fact, the upper bound $2\frac{d+1}{d+3}$ on the integrability exponent is not optimal: according to~\cite[Remark~3.13]{JKO94}, in dimension $d=3$, the optimal upper bound is $\frac32$ instead of $2\frac{d+1}{d+3}=\frac43<\frac32$.}

\medskip
\item \emph{Subcritical percolation clusters in random chess structure:}\\
Let the plane $\R^2$ be splitted into unit squares painted independently in black or white, with probability $\mu \in (0,1)$ and $1-\mu$, respectively,
and consider clusters of black squares having an edge in common. Assume that the probability $\mu$ is subcritical, that is, $0<\mu<\mu_c \sim 0.41$, so that all black clusters are bounded almost surely. Now define $F=\cup_nI_n$ as the complement of the infinite white cluster.
Then~\ref{ass1} holds for all $1\leq p<2$.
\end{enumerate}
\end{lem}

Next, we turn to the question of the validity of the elliptic regularity assumption~\ref{ass2}, which takes the form of a localized Calder\'on--Zygmund estimate in the inclusions. 
As it is well known, such estimates require suitable regularity of the inclusions' boundaries, and $C^1$ regularity is critical~\cite{Jerison-Kenig-95}. Sufficient conditions are listed in the following lemma, for which further explanations and detailed references are postponed to Appendix~\ref{sec:app}. 

\begin{lem}[Validity of~\ref{ass2}]\label{lem:ass2}
Let $\{I_n\}_n$ be a collection of disjoint, open, connected, bounded subsets of $\R^d$ with Lipschitz boundary.
\begin{enumerate}[(i)]
\item \emph{Uniformly $C^{1}$ inclusions:}\\
Assume that rescaled inclusions $I_n'=\diam(I_n)^{-1}I_n$ are uniformly of class $C^1$ in the following sense (which is the $C^1$ version of the uniform Lipschitz condition in Lemma~\ref{lem:ass1}(i)): there is a constant $C_0$ and a continuous map $\omega:[0,\infty]\to[0,\infty]$ with $\omega(0)=0$, and for all $n$ there is a collection of balls~$\{D_i^n\}_i$ covering~$\partial I_n'$ such that
\begin{enumerate}[\quad---]
\item for all $i$, in some orthonormal frame, $D_i^n\cap\partial I_n'$ is the graph of a $C^1$ function the gradient of which is bounded pointwise by $C_0$ and admits $\omega$ as a modulus of continuity;
\item for all $x \in \partial I_n'$ we have $B(x,\tfrac1{C_0}) \subset D_i^n$ for some $i$;
\item $\sup_i\sharp\{j:D_i^n\cap D_j^n\ne\varnothing\}\le C_0$.
\end{enumerate}\smallskip
Then~\ref{ass2} holds for all $2\le q<\infty$.
\medskip\item \emph{Uniformly Lipschitz inclusions:}\\
Assume that the rescaled inclusions $I_n'=\diam(I_n)^{-1}I_n$ satisfy the uniform Lipschitz regularity condition in Lemma~\ref{lem:ass1}(i), for some constant $C_0$.
Then there exists $q_0>3$ (or $q_0>4$ if $d=2$), only depending on $d,C_0$, such that~\ref{ass2} holds for all $2\le q\le q_0$.
\medskip\item \emph{Uniformly $C^1$ deformations of convex polygonal domains:}\\
Assume that there is a finite collection of convex polytopes $\{J_i\}_i$ such that rescaled inclusions $I_n'=\diam(I_n)^{-1}I_n$ are uniformly $C^1$ deformations of the $J_i$'s in the following sense:
there is a constant~$C_0$ and a continuous map $\omega:[0,\infty]\to[0,\infty]$ with $\omega(0)=0$, such that for all $n$ there is a $C^1$ diffeomorphism $\psi_n:J_{i}\to I_n'$, for some $i$, with $\frac1{C_0}\le\|\nabla\psi_n\|_{L^\infty}\le C_0$ and with $\omega$ being a modulus of continuity for both $\nabla\psi_n$ and $\nabla\psi_n^{-1}$.
Then~\ref{ass2} holds for all $2\le q<\infty$.
\end{enumerate}
\end{lem}

\begin{examples}\label{ex:ass}
The combination of the above two lemmas shows to what extent Theorem~\ref{th:main} extends the previous literature on homogenization of the double-porosity model. In~\cite{MR1993376,MR3902123}, inclusions are indeed assumed to be both uniformly bounded and uniformly separated from one another.
Many new models can now be covered, such as the following examples illustrated in Figure~\ref{fig:Examples1.8}.
\begin{enumerate}[(a)]
\item \emph{Random spherical inclusions with random radii:}\\
Consider a stationary ergodic random inclusion process $F_1=\cup_nI_n$ where the inclusions $I_n$'s are pairwise disjoint and where each $I_n$ is a ball with a random radius denoted by $r_n$. Consider the associated characteristic separation lengths
\[\nu_n\,:=\,\tfrac{\rho_n}{r_n},\qquad\text{in terms of}\quad\rho_n\,:=\,\min_{m:m\ne n}\dist(I_n,I_m),\]
and assume that they satisfy the following moment bound, for some $\alpha>1$,
\begin{equation}\label{eq:mombound-nun+}
\limsup_{R\uparrow\infty}\frac1{|B_R|}\sum_{n:I_n\cap B_R\ne\varnothing}\nu_n^{-\alpha}<\infty,\qquad\text{almost surely}.
\end{equation}
Qualitatively, this requires having on average more empty space around larger inclusions.
By Lemmas~\ref{lem:ass1}(ii) and~\ref{lem:ass2}(i), we find that~\ref{ass1} and~\ref{ass2} hold for all $p=q'\in(1,\frac{2\alpha}{1+\alpha}]$, so Theorem~\ref{th:main} applies.
For $\alpha=\infty$, the above moment bound~\eqref{eq:mombound-nun+} reduces to the uniform separation condition $\rho_n\ge\frac1{C_0}r_n$, and we then find that~\ref{ass1} holds for~$p=2$ without loss of integrability.

\smallskip\noindent
In the special case when random radii are bounded from above and below, that is, $\inf_nr_n>0$ and $\sup_nr_n<\infty$ almost surely, then we can rather appeal to Lemma~\ref{lem:ass1}(iv): by convexity, no condition is then required on separation and we find that~\ref{ass1} and~\ref{ass2} hold automatically for all $p=q'\in(1,2\frac{d+1}{d+3})$.

\smallskip\item \emph{Random cylindrical inclusions with random lengths:}\\
Consider a stationary ergodic random inclusion process $F_2=\cup_nI_n$ where the inclusions $I_n$'s are pairwise disjoint and where each $I_n$ is the isometric image of a cylinder $B'(0,1)\times(-L_n,L_n)$ with unit width and random length $L_n\ge1$.
In this anisotropic setting, assume that the separation distances between inclusions satisfy the following moment bound, for some $\beta>1$,
\begin{equation}\label{eq:cond-dist cylinders+re}
\limsup_{R\uparrow\infty}\frac1{|B_R|}\sum_{n:I_n\cap B_R\ne\varnothing} \Big(\min_{m:m\ne n}\dist(I_n,I_m)\Big)^{-\beta}\,<\,\infty,\qquad\text{almost surely}.
\end{equation}
Note that this condition is independent of the random lengths $\{L_n\}_n$, hence is much weaker than~\eqref{eq:mombound-nun+}.
By Lemmas~\ref{lem:ass1}(ii) and~\ref{lem:ass2}(i), we find that~\ref{ass1} and~\ref{ass2} hold for all $p=q'\in(1,\frac{2\beta}{1+\beta}]$, so Theorem~\ref{th:main} applies.

\smallskip\noindent
Notably, for $\beta=\infty$, the condition~\eqref{eq:cond-dist cylinders+re} reduces to $\inf_{n\ne m}\dist(I_n,I_m)>0$ almost surely, and we then find that~\ref{ass1} holds for $p=2$ without loss of integrability --- which is a first to our knowledge for an inclusion process that does not satisfy the uniform separation condition~\eqref{eq:nun-infty}.

\smallskip\item \emph{Poisson inclusion processes:}\\
Various interesting inclusion processes can be constructed from a Poisson process $\{x_n\}_n$ on~$\R^d$ and can be checked to satisfy our assumptions:
\smallskip\begin{enumerate}[(c.1)]
\item Given an arbitrary stationary ergodic point process $\{x_n\}_n$ on $\R^d$, we can consider the associated inclusion process $F_3:=\cup_nI_n$ given by
\[I_n\,:=\,B(x_n,r_n),\qquad \text{with}\quad r_n:=\tfrac12\min_{m:m\ne n}\dist(x_n,x_m).\]
By Lemmas~\ref{lem:ass1}(iv) and~\ref{lem:ass2}(i),  we find that~\ref{ass1} and~\ref{ass2} hold at least for all $p=q'\in(1,2\frac{d+1}{d+3})$.
Note that by definition we have $\inf_{n\ne m}\dist(I_n,I_m)=0$ almost surely and that the inclusions can be arbitrarily large or small with positive probability if $\{x_n\}_n$ is for instance a Poisson process.

\smallskip
\item Given a Poisson process $\{x_n\}_n$ on $\R^d$ with intensity $\lambda$, consider the regularized clusters of unit balls around points of the process, as given e.g.\@ by
\begin{gather*}
F_4\,:=\,\mathcal O\big(\big\{x\in C+2B^\circ\,:\,\dist\big(x,C+2B^\circ\big)>1\big\}\big),\\
\text{in terms of}\quad C\,:=\,\bigcup_{n, m\atop0<|x_n-x_m|\le4}[x_n,x_m],
\end{gather*}
where $B^\circ:=B(0,1)$ stands for the unit ball at the origin and where for a set $S$ we define~$\mathcal O(S)$ to be the complement of the infinite connected component of the complement of $S$.
This is well-defined and nontrivial almost surely in the subcritical Poisson percolation regime, that is, provided that the intensity of the Poisson process is small enough, $0<\lambda<\lambda_c$.
By definition, the random set $F_4$ is uniformly of class $C^2$ and the distance between connected components is always $\ge2$.
In addition, diameters of connected components are known to have some finite exponential moments (see e.g.~\cite[Theorem~2]{DCRT-20}).
By Lemmas~\ref{lem:ass1}(ii) and~\ref{lem:ass2}(i), we find that~\ref{ass1} and~\ref{ass2} hold for all $p=q'\in(1,2)$.
\end{enumerate}

\smallskip\item \emph{Subcritical percolation clusters:}\\
For the inclusion process $F_5$ in the plane $\R^2$ given by subcritical percolation clusters in random chess structure as described in Lemma~\ref{lem:ass1}(v), further applying Lemma~\ref{lem:ass2}(ii), we find that~\ref{ass1} and~\ref{ass2} hold at least for all $p=q'\in[\frac43,2)$.
Note that in this setting the set $F_5$ is almost surely not Lipschitz due to black squares possibly touching only at a vertex.
This model could be extended to higher dimensions, but we skip the details.
\end{enumerate}
\end{examples}


\begin{figure}[htbp]
    \centering
    \begin{minipage}{0.45\textwidth}
        \centering
        \includegraphics[width=\textwidth, trim=0.3cm 0.5cm 0.5cm 1cm, clip]{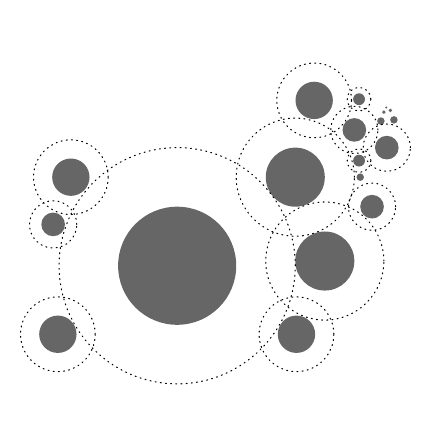}
        \caption*{(a) model $F_1$}
    \end{minipage}
    \hfill
    \begin{minipage}{0.45\textwidth}
        \flushright
        \includegraphics[width=\textwidth, angle=180, trim=0.cm 0cm 0.cm 0cm, clip]{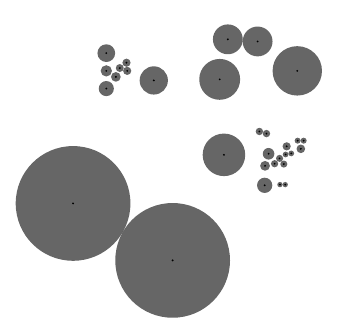}
        \caption*{(c.1) model $F_3$}
    \end{minipage}
    \begin{minipage}{1\textwidth}
        \centering
        \includegraphics[width=\textwidth, trim=0cm 0.2cm 0cm -0.5cm, clip]{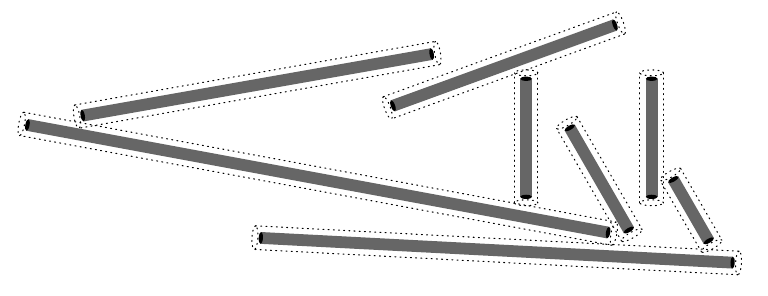}
        \caption*{(b) model $F_2$}
    \end{minipage}
    
    \vspace{1cm}
    \begin{minipage}{0.5\textwidth}
        \centering
        \includegraphics[width=\textwidth, trim=2cm 2cm 2cm 2cm, clip]{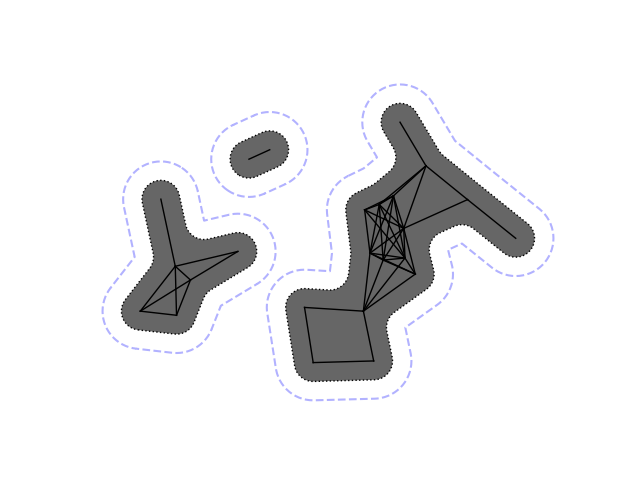}
        \caption*{(c.2) model $F_4$}
    \end{minipage}
    \hfill
    \begin{minipage}{0.45\textwidth}
        \centering
        \includegraphics[width=\textwidth, trim=0cm 0cm 0cm -0.5cm, clip]{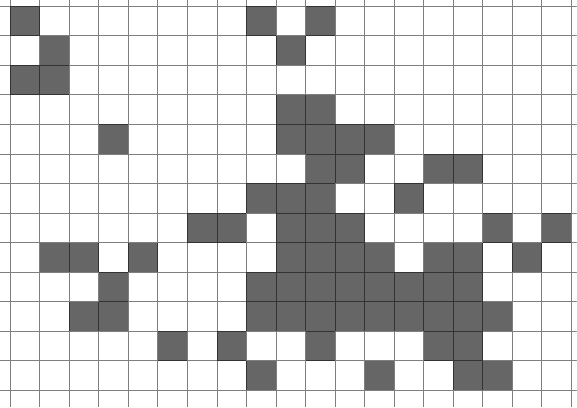}
        \caption*{(d) model $F_5$}
    \end{minipage}
\caption{This illustrates typical realizations for the different models constructed in Examples~\ref{ex:ass}, for which the assumptions of the qualitative homogenization result hold.}
\label{fig:Examples1.8}
\end{figure}


\section{Qualitative theory}\label{sec:qual}
This section is devoted to the proof of Theorems~\ref{th:main} and~\ref{th:corr-qual}. We start in Section~\ref{sec:compact} by proving a compactness result capturing the `resonant' behavior inside the inclusions for the solution $u_\e$ of the double-porosity problem~\eqref{eq:double-por-eqn}, in terms of the weak limit of $u_\e$ outside. Next, in Section~\ref{sec:erg-res}, we prove an ergodic theorem for this resonant description inside inclusions. Based on these two main ingredients, the proofs of Theorems~\ref{th:main} and~\ref{th:corr-qual} are postponed to Sections~\ref{sec:thmain} and~\ref{sec:thcorrqual}, respectively.

\subsection{Compactness result}\label{sec:compact}
The main technical ingredient is the following compactness result for the solution~$u_\e$ of the double-porosity problem~\eqref{eq:double-por-eqn}, which follows by combining the extension operator of Assumption~\ref{ass1} together with suitable `resonant' auxiliary problems inside the inclusions.
\begin{lem}\label{prop:sol-inside}
Let the random inclusion process $F=\cup_nI_n\subset \R^d$ satisfy Assumptions~\ref{ass0},
\ref{ass1}, and~\ref{ass2} for some $1<p \le 2$ and $q=p'$.
Given a bounded Lipschitz domain $D\subset\R^d$ and $f\in \Ld^2(D)$, for all $\e>0$,
let $u_\e\in H^1_0(D)$ be the solution of the double-porosity problem~\eqref{eq:double-por-eqn}.
Then, almost surely, there exists $\bar u\in W^{1,p}_0\cap L^2(D)$ such that, along a subsequence, we have
\begin{equation}\label{eq:estim-inside-Lp}
\lim_{\e \downarrow 0}\|u_\e-\bar u\|_{\Ld^p(D\setminus F_\e(D))}\,=\,0,\qquad\lim_{\e \downarrow 0}\|u_\e-\bar u-v_\e\|_{\Ld^p(F_\e(D))}\,=\,0,
\end{equation}
that is, in a more compact form, $\lim_{\e}\|u_\e-\bar u-\chi_\e v_\e\|_{\Ld^p(D)}=0$,
where $v_\e\in W^{1,p}_0(F_\e(D))$ is the solution of the auxiliary equation
\begin{equation}\label{eq:def-veps}
v_\e-\e^2\triangle v_\e=f-\bar u,\qquad\text{in $F_\e(D)$},
\end{equation}
and where we use the same notation for $F_\e(D)$ and $\chi_\e$ as in the statement of Theorem~\ref{th:main}.
\end{lem}

Before moving on to the proof of this result, let us first note the following consequence of Assumption~\ref{ass2}.

\begin{lem}\label{lem:regp}
Let the random inclusion process $F=\cup_nI_n\subset\R^d$ satisfy Assumptions~\ref{ass0} and~\ref{ass2} with $q=p'$ for some $1<p\le 2$. Then, for all $n$ and all $f_1\in \Ld^p(I_n)$, $f_2\in \Ld^p(I_n)^d$, there is a unique weak solution $u\in W^{1,p}_0(I_n)$ of
\begin{equation}\label{e.regp0}
u-\triangle u=f_1+\Div(f_2),\qquad\text{in $I_n$},
\end{equation}
and it satisfies
\begin{equation}\label{e.regp}
\|u\|_{\Ld^p(I_n)}\,\lesssim\,\|f_1\|_{\Ld^p(I_n)}+\|f_2\|_{\Ld^p(I_n)},
\end{equation}
where the multiplicative constant does not depend on $n$.
\end{lem}

\begin{proof}[Proof of Lemma~\ref{lem:regp}]
We focus on the proof of the estimate~\eqref{e.regp} in case $f_1\in L^2(I_n)$, $f_2\in L^2(I_n)^d$, while the existence part of the statement and the corresponding general estimate follow by an approximation argument.
By linearity, it is enough to consider separately the cases $f_2=0$ and  $f_1=0$.
We split the proof into two steps accordingly.

\medskip
\noindent
\step1 Case $f_2=0$.\\
Let $f_1\in L^2(I_n)$. For any $M,\epsilon>0$, testing equation~\eqref{e.regp0} with 
\[
u_{\e,M} \,:=\, \Big((|u|\wedge M)^2+\e^2\Big)^{\frac{p-2}{2}}\Big((u\wedge M)\vee (-M)\Big) ~\in~ H^1_0(I_n),
\]
we find 
\begin{multline*}
    \int_{I_n} \Big(\big((|u|\wedge M)^2+\e^2\big)^{\frac{p-2}{2}}+(p-2) (|u|\wedge M)^2\big((|u|\wedge M)^2+\e^2\big)^{\frac{p-4}{2}} \Big)|\nabla u|^2 \mathds1_{|u|<M}
    \\+ \int_{I_n}\big((|u|\wedge M)^2+\e^2\big)^{\frac{p-2}{2}}(|u|\wedge M)|u|
    \,
    =\,\int_{I_n}  f_1  u_{\e,M}.
\end{multline*}
Taking the limits $\e \downarrow 0$ then $M \uparrow\infty$, this entails by H\"older's inequality 
$$
(p-1)\int_{I_n}  |\nabla u |^2 |u|^{p-2}  +  \int_{I_n} |u|^p   = \int_{I_n} f_1 u|u|^{p-2}   \leq \| f_1  \|_{\Ld^p(I_n)} \|{u}\|_{\Ld^p(I_n)}^{p-1},
$$
and therefore 
\[\|{u} \|_{\Ld^p(I_n)} \leq\|{f_1 }\|_{\Ld^p(I_n)},\]
that is, \eqref{e.regp}.

\medskip\noindent
\step2 Case $f_1=0$.\\
We further distinguish between two cases, depending on the size of the diameter of $I_n$.

\medskip\noindent
\substep{2.1} Case $\diam(I_n)\le 1$.\\
Given $h\in C^\infty_c(I_n)^d$, let us consider the unique solution $z_h\in H^1_0(I_n)$ of
\begin{equation}\label{e.regp66}
z_h-\triangle z_h=\Div(h),\qquad\text{in $I_n$}.
\end{equation}
As $\diam(I_n)\le 1$, we can choose a unit ball $B$ such that $I_n\subset B$.
By Assumption~\ref{ass2} with $q=p'\ge2$, we then get
\[\|\nabla z_h\|_{\Ld^{p'}(I_n)}\,\lesssim\,\|h\|_{\Ld^{p'}(I_n)}+\|\nabla z_h\|_{\Ld^2(I_n)}.\]
By the energy estimate for $z_h$ and Jensen's inequality, with $\diam(I_n)\le1$, we can deduce
\begin{equation}\label{eq:estim-zh-lala}
\|\nabla z_h\|_{\Ld^{p'}(I_n)}\,\lesssim\,\|h\|_{\Ld^{p'}(I_n)}+\|h\|_{\Ld^2(I_n)}\,\lesssim\,\|h\|_{\Ld^{p'}(I_n)}.
\end{equation}
Now for the solution $u$ of~\eqref{e.regp0} with $f_1=0$, we have the duality identity
\[\int_{I_n}h\cdot\nabla u\,\stackrel{\eqref{e.regp66}}=\,
-\int_{I_n}uz_h-\int_{I_n}\nabla u\cdot\nabla z_h\,\stackrel{\eqref{e.regp0}}=\,
\int_{I_n}f_2\cdot\nabla z_h,\]
which entails
\begin{equation*}
\|\nabla u\|_{\Ld^p(I_n)}
\,=\,\sup\Big\{\int_{I_n}h\cdot\nabla u\,:\,\|h\|_{\Ld^{p'}(I_n)}=1\Big\}
\,\le\,\|f_2\|_{\Ld^p(I_n)}\sup\Big\{\|\nabla z_h\|_{\Ld^{p'}(I_n)}\,:\,\|h\|_{\Ld^{p'}(I_n)}=1\Big\},
\end{equation*}
and thus, by~\eqref{eq:estim-zh-lala},
\begin{equation*}
\|\nabla u\|_{\Ld^p(I_n)}
\,\lesssim\,\|f_2\|_{\Ld^p(I_n)}.
\end{equation*}
By Poincar\'e's inequality with $\diam(I_n)\le 1$, we may then conclude
\[
\|u\|_{\Ld^{p}(I_n)} \,\lesssim\,\|\nabla u\|_{\Ld^{p}(I_n)}\,\lesssim\,\|f_2\|_{\Ld^p(I_n)},
\]
that is, \eqref{e.regp}.

\medskip\noindent
\substep{2.2}  Case $\diam(I_n)\ge 1$.\\
In this case, we need to further rely on the zeroth-order term in equation~\eqref{e.regp0}, which sets a scale in the problem.
Arguing again by duality, we first claim that~\eqref{e.regp} is a consequence of the  following: For all $f \in \Ld^{p'}(I_n)$ with $2 \leq p' < \infty$, the solution $v_f \in H^1_0(I_n)$ of
\begin{equation}\label{e.regp4}
v_f - \triangle v_f = f, \qquad \text{in $I_n$},
\end{equation}
satisfies
\begin{equation}\label{e.regp2}
\| \nabla v_f \|_{\Ld^{p'}(I_n)} \,\lesssim\, \| f \|_{\Ld^{p'}(I_n)}.
\end{equation}
This implication indeed follows from the duality identity
\[\int_{I_n}fu\,\stackrel{\eqref{e.regp4}}=\,\int_{I_n} uv_f+\nabla u \cdot \nabla v_f \,\stackrel{\eqref{e.regp0}}=\,-\int_{I_n}f_2\cdot\nabla v_f,\]
which yields
\[\|u\|_{\Ld^p(I_n)}\,=\,\sup\Big\{\int_{I_n}fu\,:\,\|f\|_{\Ld^{p'}(I_n)}=1\Big\}\,\le\,\|f_2\|_{\Ld^p(I_n)}\sup\Big\{\|\nabla v_f\|_{\Ld^{p'}(I_n)}\,:\,\|f\|_{\Ld^{p'}(I_n)}=1\Big\}.\]

We now turn to the proof of \eqref{e.regp2}, for which we proceed locally.
Since $\Ld^{p'}(I_n)\subset \Ld^2(I_n)$, for any unit ball $B$ with $I_n \cap B \ne \varnothing$, we may consider the solution $\phi_f \in H^1_0(2B)$  of
\begin{equation}\label{e.regp3}
\triangle \phi_f= \mathds{1}_{I_n \cap 2B} f,\qquad\text{in $\R^d$}.
\end{equation}
In these terms, we find $v_f-\triangle v_f = \Div(\nabla \phi_f)$ in $I_n\cap 2B$, and Assumption~\ref{ass2} for $q=p'$ entails 
\begin{equation*}
\| \nabla v_f \|_{\Ld^{p'}(I_n \cap B)}  \,\lesssim\,  \| \nabla \phi_f\|_{\Ld^{p'}(I_n \cap 2B)} + \| \nabla v_f \|_{\Ld^2(I_n \cap 2B)} .
\end{equation*}
By Poincar\'e's inequality in $2B$ for $\nabla \phi_f$ with $\int_{2B}\nabla \phi_f=0$, followed 
by Calder\'on-Zygmund estimates in $2B$,
we get
\[
\|\nabla \phi_f\|_{\Ld^{p'}(2B)}\,\lesssim\,
\|\nabla^2 \phi_f\|_{\Ld^{p'}(2B)} \,\lesssim\, 
\|f\|_{\Ld^{p'}(I_n \cap 2B)} ,
\]
hence
\begin{equation}\label{e.regp5}
\| \nabla v_f \|_{\Ld^{p'}(I_n \cap B)} \lesssim  \| f \|_{\Ld^{p'}(I_n \cap 2B)} + \| \nabla v_f \|_{\Ld^2(I_n \cap 2B)}.
\end{equation}
It remains to control the second right-hand side term, and we claim that for some $M\lesssim 1$ large enough (not depending on $I_n$) we have the Caccioppoli-type estimate
\begin{equation}\label{e.regp6}
 \| \nabla v_f \|_{\Ld^2(I_n \cap 2B)} \lesssim \Big\| e^{-\frac{|\cdot-x_B|}{M}}f \Big\|_{\Ld^{p'}(I_n)},
\end{equation}
where $x_B$ stands for the center of $B$. Before concluding the argument, let us prove this estimate.
Setting $\eta:=e^{-\frac{|\cdot- x_B|}{M}}$, and testing~\eqref{e.regp4} with  $\eta^2 v_f \in H^1_0(I_n)$, we find
\[
\int_{I_n} \eta^2 v_f^2 +\int_{I_n}\eta^2 |\nabla v_f|^2 \lesssim \int_{I_n} \eta^2 |v_f|| f|+\int_{I_n} \eta |v_f| |\nabla \eta | |\nabla v_f| ,
\]
which entails, for $M\lesssim1$ large enough, by Young's inequality and by the property $|\nabla \eta| \le \frac 1M \eta$ for the exponential,
\[
\int_{I_n}\eta^2 v_f^2 +\int_{I_n}\eta^2 |\nabla v_f|^2 \lesssim \int_{I_n}\eta^2 f^2.
\]
Since $\eta \gtrsim 1$ on $I_n \cap 2B$, the left-hand side controls in particular $\int_{I_n \cap 2B} |\nabla v_f|^2$. Using H\"older's inequality to control the right-hand side, we then get
\[\int_{I_n\cap2B}|\nabla v_f|^2\, \lesssim\, \int_{I_n}\eta^2 f^2
\lesssim\, \Big(\int_{I_n}\eta^2\Big)^{1-\frac{2}{p'}}\Big(\int_{I_n} \eta^2|f|^{p'}\Big)^\frac2{p'},\]
which proves~\eqref{e.regp6}.

With~\eqref{e.regp5} and~\eqref{e.regp6} at hand, we are now in position to conclude. For that purpose, let us cover~$I_n$ with a finite union of unit balls $B_i := B(x_i,1)$ such that
$$
\sup_i \sharp \left\{ j \, : \, B_i \cap B_j \neq \varnothing \right\} \,\lesssim\,1.
$$
Summing the contributions of the norm of $v_f$ in each $B_i$, and combining~\eqref{e.regp5} and~\eqref{e.regp6}, we find
\begin{equation*}
\| \nabla v_f \|^{p'}_{\Ld^{p'}(I_n)}
\,\lesssim\,\sum_{i}   \Big(  \| f \|^{p'}_{\Ld^{p'}(I_n \cap 2B_i)} +  \left\| e^{-\frac{|\cdot-x_i|}{M}}  f \right\|^{p'}_{\Ld^{p'}(I_n)} \Big)  
\,\lesssim\, \|f\|^{p'}_{\Ld^{p'}(I_n)},
\end{equation*}
where we used that $\sum_i (e^{-\frac{\lvert x_i - \cdot \rvert}{M}}+\mathds1_{I_n \cap 2B_i}) \lesssim 1$. This proves~\eqref{e.regp2}, hence~\eqref{e.regp}.
\end{proof}

With the above elliptic regularity result at hand, we are now in position to prove the key compactness result of Lemma~\ref{prop:sol-inside}.

\begin{proof}[Proof of Lemma~\ref{prop:sol-inside}]
Given a ball $B$ that contains the domain $D$, and extending $u_\e$ by $0$ on $B\setminus D$, the extension property of Assumption~\ref{ass1} allows to consider\footnote{Since $F_\e(D)\subset (\e F)\cap D$ does not necessarily coincide with $(\e F)\cap D$ due to inclusions possibly intersecting the boundary~$\partial D$, we have to replace the extension by the identity in $((\e F)\cap D) \setminus F_\e(D)$, which does not change the bounds.}
\[\hat u_\e\,:=\,P_{B,\e} (u_\e \mathds 1_{D \setminus F_\e(D)})|_D\,\in \,W^{1,p}_0(D),\]
which satisfies $\hat u_\e|_{D\setminus F_\e(D)}=u_\e|_{D\setminus F_\e(D)}$ and
\begin{equation}\label{eq:estim-hatueps-ext}
\| \nabla\hat u_\e\|_{L^{p}(D)} \lesssim \|\nabla u_\e\|_{L^2(D \setminus F_\e(D))}.
\end{equation}
By the a priori estimate for $u_\e$ and by Poincar\'e's inequality, this entails that $\hat u_\e$ is bounded in $W^{1,p}_0(D)$, uniformly on the probability space. Almost surely, by Rellich's theorem, there exists $\bar u\in W^{1,p}_0(D)$ such that $\hat u_\e\to\bar u$ in $\Ld^p(D)$ along a subsequence (not relabeled). Hence, by the extension property, along this subsequence,
\begin{equation}\label{eq:plop1}
\|u_\e-\bar u\|_{\Ld^p(D\setminus F_\e(D))}\,\le\,\|\hat u_\e-\bar u\|_{\Ld^p(D)}\,\xrightarrow{\e\downarrow0}\,0.
\end{equation}
It remains to estimate $u_\e-\bar u-v_\e\in W^{1,p}(F_\e(D))$, where $v_\e$ is the solution of the auxiliary problem~\eqref{eq:def-veps}. By the triangle inequality, 
\begin{equation}\label{eq:plop2}
\|u_\e-\bar u-v_\e\|_{\Ld^p(F_\e(D))} \le \|u_\e-\hat u_\e-v_\e\|_{\Ld^p(F_\e(D))}+\|\hat u_\e-\bar u\|_{\Ld^p(D)}.
\end{equation}
Since $r_\e:=u_\e-\hat u_\e-v_\e$ belongs to $W^{1,p}_0(F_\e(D))$ and satisfies 
\begin{equation*}
r_\e-\e^2\triangle r_\e =\bar u-\hat u_\e+\e^2\triangle \hat u_\e,\qquad\text{in $F_\e(D)$},
\end{equation*}
Lemma~\ref{lem:regp} together with a rescaling argument yields
\[
\|r_\e\|_{\Ld^p(F_\e(D))} \,\lesssim\, \|\hat u_\e-\bar u\|_{\Ld^p(D)}
+\e \|\nabla \hat u_\e\|_{\Ld^p(D)}. 
\]
Combined with~\eqref{eq:estim-hatueps-ext}, \eqref{eq:plop1}, and~\eqref{eq:plop2}, this concludes the proof of the claimed convergences~\eqref{eq:estim-inside-Lp}.

It remains to check that almost surely the extracted limit $\bar u\in W^{1,p}_0(D)$ actually belongs to $L^2(D)$ (note that this does not follow from the Sobolev embedding for small $p$). For $\psi\in C^\infty_c(D)$, recalling that $u_\e=\hat u_\e$ in $D\setminus F_\e(D)$ and using that $\mathds1_{F_\e(D)}=\mathds1_{\e F}$ in the support of $\psi$ for $\e$ small enough, we can estimate
\begin{eqnarray*}
\bigg|\int_D\psi\Big(\mathds1_{D\setminus F_\e(D)}u_\e-(1-\E[\mathds1_F])\bar u\Big)\bigg|
&\le&\bigg|\int_D\mathds1_{D\setminus F_\e(D)}\psi(\hat u_\e-\bar u)\bigg|+\bigg|\int_D(\mathds1_{F_\e(D)}-\E[\mathds1_F])\psi\bar u\bigg|\\
&\le&\|\psi\|_{L^{p'}(D)}\|\hat u_\e-\bar u\|_{L^p(D)}+\bigg|\int_D(\mathds1_F(\tfrac\cdot\e)-\E[\mathds1_F])\psi\bar u\bigg|.
\end{eqnarray*}
By definition of $\bar u$, the first right-hand side term tends to $0$ as $\e\downarrow0$ along the chosen subsequence, while the ergodic theorem ensures that for a typical realization the second right-hand side term also tends to $0$. We conclude that $\mathds1_{D\setminus F_\e(D)}u_\e\to(1-\E[\mathds1_F])\bar u$ in the distributional sense on $D$. As the energy estimate for $u_\e$ ensures that $\|u_\e\|_{L^2(D)}\lesssim1$, and as $\E[\mathds1_F]<1$ by Assumption~\ref{ass0}, this entails that~$\bar u$ belongs to~$L^2(D)$.
\end{proof}

\subsection{Ergodic theorem for resonant equations}\label{sec:erg-res}
We show how the ergodic theorem can be applied to compute the limit of the solution $v_\e$ of the auxiliary problem~\eqref{eq:def-veps}, which is viewed as a locally stationary random field. This would be a straightforward consequence of the ergodic theorem together with Poincar\'e's inequality on each inclusion, but some careful approximation argument is needed to avoid additional requirements on the geometry and on the diameters of the particles.

\begin{lem}\label{lem:control Poincare constants}
With the same assumptions and notation as in Lemma~\ref{prop:sol-inside}, the auxiliary solution $v_\e$ satisfies almost surely
\begin{equation}\label{eq:conv-veps-lem}
v_\e - v(\tfrac\cdot\e) (f-\bar u) \to 0, \quad \text{strongly in $\Ld^2(D)$},
\end{equation}
where $v\in H^1_0(F)$ is the unique almost sure weak solution of~\eqref{e.999}, which is a stationary random field. In particular, almost surely,
\[v_\e \cvf \E[v] (f-\bar u), \quad \text{weakly in $\Ld^2(D)$},\]
\end{lem}
\begin{proof}
We focus on the proof of~\eqref{eq:conv-veps-lem}, while the other part of the statement is an immediate consequence of the ergodic theorem for the stationary random field $v$.
We proceed in three steps.

\medskip
\step1 Regularization of the source term.\\
For $\delta >0$, let $\eta_\delta:=\delta^{-d}\eta(\tfrac\cdot\delta)$ be the rescaling of some mollifier $\eta\in C^\infty_c(\R^d)$ with $\eta\ge0$ and $\int_{\R^d}\eta=1$,
and consider the almost sure solution $v_\e^\delta \in H^1_0(F_\e(D))$ of
\[v_\e^\delta - \e^2 \Delta v_\e^\delta = \eta_\delta\ast(f-\bar u ),\qquad\text{in $F_\e(D)$}.\]
The energy estimate for the difference $v_\e-v_\e^\delta$ yields almost surely
\[\|v_\e-v_\e^\delta\|_{L^2(D)}\,\le\,\|(f-\bar u)-\eta_\delta\ast(f-\bar u )\|_{L^2(D)}.\]
Further using the fact that $0 \leq v \leq 1$ almost everywhere in $F$, we deduce for all $\delta >0$,
\begin{multline*}
\limsup_{\e\downarrow0} \| v_\e - v(\tfrac\cdot\e)(f-\bar u) \|_{L^2(D)}\\
\,\le\, 2\| (f - \bar u) - \eta_\delta\ast(f-\bar u) \|_{\Ld^2(D)} + \limsup_{\e\downarrow0} \big\| v_\e^\delta - v(\tfrac\cdot\e)\big(\eta_\delta\ast(f-\bar u)\big) \big\|_{\Ld^2(D)}.
\end{multline*}
As $f\in L^2(D)$ by assumption and as $\bar u\in L^2(D)$ by Lemma~\ref{prop:sol-inside}, the first right-hand side term tends to~$0$ as $\delta\downarrow0$. Hence, it remains to prove that the second limit also vanishes for any $\delta>0$.
In particular, this means that we can assume without loss of generality $f - \bar u \in W^{1,\infty}(D)$.

\medskip
\step2 Reduction to uniformly bounded diameters.\\
By Step~1, we can assume $f - \bar u \in W^{1,\infty}(D)$.
For fixed $L>0$, we introduce the decimated inclusion process
\[F^L \,:=\, \bigcup_{n: \diam (I_n) \leq L}I_n\,\subset\, F,\]
which remains a stationary ergodic random inclusion process like~$F$, and we also define its rescaling restricted to~$D$,
\[
F^L_\e(D) \,:=\, \bigcup_{n:\diam(I_n)\le L\atop \e I_n \subset D}\e I_n\,=\, F_\e(D)\cap(\e F^L).
\]
In these terms, consider the unique almost sure solutions $v^L \in H^1_0(F^L)$ and $v_\e^L \in H^1_0(F_\e^L(D))$ (both extended to zero as $H^1_0(D)$ functions) of
\[v^L - \Delta v^L = 1,\quad\text{in $F^L$},\qquad \text{and}\qquad v_\e^L - \e^2 \Delta v_\e^L = f-\bar u,\quad\text{in $F_\e^L(D)$}.\]
By the ergodic theorem together with the facts that $v_\e=v_\e^L$ in $F^L_\e(D)$ and that $|v_\e|\le\|f-\bar u\|_{L^\infty(D)}$ almost everywhere in $F_\e(D)$ by the maximum principle, we find almost surely
\begin{eqnarray*}
\limsup_{\e\downarrow0}\|v_\e-v_\e^L\|_{L^2(D)}
&=&\limsup_{\e\downarrow0}\bigg(\sum_{\e I_n\subset F_\e(D)\setminus F^L_\e(D)}\|v_\e\|_{L^2(\e I_n)}^2\bigg)^\frac12\\
&\le&\|f-\bar u\|_{L^\infty(D)}\limsup_{\e\downarrow0}\bigg(\sum_{\e I_n\subset F_\e(D)\setminus F^L_\e(D)}|\e I_n|\bigg)^\frac12\\
&\le&\|f-\bar u\|_{L^\infty(D)}\limsup_{\e\downarrow0}\bigg(\int_D\mathds1_{F\setminus F^L}(\tfrac\cdot\e)\bigg)^\frac12\\
&=&|D|^\frac12\|f-\bar u\|_{L^\infty(D)}\E[\mathds1_{F\setminus F^L}]^\frac12.
\end{eqnarray*}
Arguing similarly for $v(\tfrac\cdot\e)-v^L(\tfrac\cdot\e)$, now using the facts that $v=v^L$ in $F^L$ and that $|v|\le1$ almost everywhere in $F$, we are led to the following, for all $L >0$, almost surely,
\begin{multline}\label{eq:reduc-veps-L}
\limsup_{\e\downarrow0} \| v_\e - v(\tfrac\cdot\e)(f-\bar u) \|_{L^2(D)}\\
\,\le\, 2 |D|^{1/2} \| f - \bar u \|_{\Ld^\infty(D)}  \E[\mathds{1}_{F \setminus F^L}]^\frac12 +\limsup_{\e\downarrow0} \| v_\e^L - v^L(\tfrac\cdot\e)(f-\bar u) \|_{\Ld^2(D)}.
\end{multline}
Now note that the first right-hand side term tends to $0$ as $L\uparrow\infty$: indeed, by definition of $F^L$ and by dominated convergence, only using that $\diam(I_n)<\infty$ almost surely for all $n$ (see Assumption~\ref{ass0}), we find
\[\E [\mathds{1}_{F \setminus F^L}]\,=\,\E\bigg[\sum_n\mathds1_{I_n}\mathds1_{\diam(I_n)\ge L}\bigg] \to 0,\qquad\text{as $L \uparrow \infty$}.\]
Hence, it only remains to show that the second limit in~\eqref{eq:reduc-veps-L} vanishes for any $L>0$. In particular, this means that we can further assume without loss of generality
\begin{equation}\label{eq:sup-rad}
r_0\,:=\,\sup_n \diam(I_n) < \infty, \quad \text{almost surely}.
\end{equation}

\medskip
\step3 Conclusion.\\
By Step~1 we can assume $f-\bar u\in W^{1,\infty}(D)$ and by Step~2 we can assume that the inclusions are uniformly bounded in the sense of~\eqref{eq:sup-rad}.
For fixed $\delta>0$, we can cover the set $D_{\delta}:= \{ x \in D:\dist(x,\partial D)>\delta \}$ with a finite collection of $N_\delta$ balls $B(x_i,\frac12\delta)$ of radius $\frac12\delta$ centered at some points $x_i \in D_{\delta}$, for $1\le i\le N_\delta$, with $N_\delta$ only depending on $d,D,\delta$, in such a way that
\[\sup_{1\le i\le N_\delta} \sharp \{ j :B(x_i,\delta) \cap B(x_j,\delta) \neq \varnothing \}\,\le\,M\]
is bounded by a finite number $M$ only depending on $d$ (and not on $D,\delta$).
By the covering property, for each inclusion $\e I_n \subset D_\delta$, there exists $1\le i\le N_\delta$ such that $\e I_n \cap B(x_i,\frac12\delta) \neq \varnothing$, and thus, using the assumed uniform boundedness of the inclusions, cf.~\eqref{eq:sup-rad}, we have for $0<\e<\delta/(2r_0)$,
\[
\e I_n \,\subset\, B(x_i,\tfrac12\delta) + B(0,r_0\e) \,\subset\, B(x_i,\delta)\, \Subset\, D.
\]
Hence, with the notation $F_\e(B(x_i,\delta)) := \bigcup_{n :\e I_n \subset B(x_i,\delta)} \e I_n $, we infer that 
\[
D \cap \e F \subset \bigg( \bigcup_{i=1}^{N_\delta} F_\e(B(x_i,\delta)) \bigg) \cup (D \setminus D_{2\delta}),
\]
and we can then bound
\begin{multline*}
\| v_\e - v(\tfrac\cdot\e)(f-\bar u)\|_{L^2(D)}^2
\, \lesssim\,|D \setminus D_{2\delta}|\| f - \bar u \|_{L^\infty(D)}^2\\
+ \sum_{i=1}^{N_\delta} \int_{F_\e(B(x_i,\delta))} \bigg( \Big|f - \bar u - \fint_{B(x_i,\delta)}(f-\bar u) \Big|^2 + \Big| v_\e - v(\tfrac\cdot\e) \fint_{B(x_i,\delta)} (f-\bar u) \Big|^2 \bigg).
\end{multline*}
Using an energy estimate for $v_\e - v(\tfrac\cdot\e) \fint_{B(x_i,\delta)} (f-\bar u)$ on each inclusion $\e I_n\subset B(x_i,\delta)$, using again the fact that $|v|\le1$ almost everywhere, and using Poincar\'e's inequality in $B_\delta(x_i,\delta)$, we are led to
\begin{eqnarray*}
\limsup_{\e\downarrow0}\| v_\e - v(\tfrac\cdot\e)(f-\bar u)\|_{L^2(D)}^2
&\lesssim&|D \setminus D_{\delta}|\|f-\bar u \|_{L^\infty(D)}^2 +\sum_{i=1}^{N_\delta} \int_{B(x_i,\delta)} \Big|f - \bar u - \fint_{B(x_i,\delta)}(f-\bar u) \Big|^2\\
& \lesssim& |D \setminus D_{\delta}|\|f-\bar u \|_{L^\infty(D)}^2 +M\delta^2|D| \|\nabla (f-\bar u)\|^2_{L^\infty(D)}.
\end{eqnarray*}
As the right-hand side tends to $0$ as $\delta\downarrow0$, the conclusion~\eqref{eq:conv-veps-lem} follows.
\end{proof}

\subsection{Proof of Theorem~\ref{th:main}}\label{sec:thmain}
Combining the extension property of Assumption~\ref{ass1} together with the compactness result of Lemma~\ref{prop:sol-inside}, we can control the solution inside inclusions by its value outside, and we may then proceed to the proof of Theorem~\ref{th:main} with Tartar's method of oscillating test functions. In order to accommodate the mere $\Ld^p$ control inside inclusions given by Lemma~\ref{prop:sol-inside}, we use approximate correctors in Tartar's argument, as well as a truncation argument. The use of approximate correctors raises some interesting subtleties, and we have to unravel some hidden monotonicity to conclude. To prove that $\bar u$ is the solution of the homogenized problem, we also rely on Lemma~\ref{lem:control Poincare constants}.

Let us first recall that the energy estimates for the solution $u_\e$ of the double-porosity problem~\eqref{eq:double-por-eqn} yield the following uniform a priori bounds,
\[\|u_\e\|_{\Ld^2(D)} \lesssim1,\qquad\|(1-\chi_\e+\e \chi_\e)\nabla u_\e\|_{\Ld^2(D)} \lesssim 1.\]
Appealing to Lemma~\ref{prop:sol-inside}, for a given typical realization, there exists $\bar u\in W^{1,p}_0(D)$ such that, along a subsequence (not relabeled), we have
\begin{equation}\label{eq:compact-baru}
\|u_\e-\bar u\|_{\Ld^p(D\setminus F_\e(D))}\to0,\qquad\|u_\e-\bar u-v_\e\|_{\Ld^p(F_\e(D))}\to0,
\end{equation}
where $v_\e$ is defined in Lemma~\ref{prop:sol-inside}. Note that at this stage we do not know yet that $\bar u$ satisfies an equation, nor even that $ \bar u \in H^1_0(D)$ since $1<p \le2$. This will only be deduced at the end of the proof.
Another preliminary comment: when using the corrector $\varphi$ defined in Lemma~\ref{lem:corr0} to prove homogenization by the classical method of oscillating test functions, we face delicate issues related to the lack of integrability of correctors; to circumvent this issue, it is convenient to replace correctors by $\e$-dependent approximations that are adapted to the structure of the double-porosity problem.
We split the proof into four main steps, first defining suitable approximate correctors and establishing their convergence properties, before turning to the proof of homogenization by means of oscillating test functions combined with the a priori resonant description~\eqref{eq:compact-baru} inside inclusions.

\medskip
\step1 Approximate stationary correctors.\\
For $\e>0$ and $1\le i\le d$, we define $\varphi_{\e i}\in H^1_\loc(\R^d;\Ld^2(\Omega))$ as the unique stationary weak solution of 
\begin{equation}\label{eq:approx-cor-eps}
\e^2 \varphi_{\e i}-\nabla \cdot (\mathds1_{\R^d\setminus F}+\e^2 \mathds1_F) (e_i+\nabla \varphi_{\e i})=0,\qquad\text{in $\R^d$},
\end{equation}
and we then set $\varphi_\e:=\{\varphi_{\e i}\}_{1\le i\le d}$.
We emphasize that the existence of a stationary solution $\varphi_{\e i}$ to the above equation is ensured by uniqueness thanks to the massive term.
We show that this approximate corrector~$\varphi_\e$ satisfies the following properties,
\begin{equation}\label{eq:lim-masscorr}
\lim_{\e\downarrow0}\Big(\e^2\E[|\varphi_\e|^2]+\E[\mathds1_{\R^d \setminus F} |\nabla \varphi_{\e}-\nabla \varphi|^2]+\e^2\E[\mathds1_{F}|\nabla\varphi_\e|^2]\Big)\,=\,0.
\end{equation}
We split the proof into two further substeps.

\medskip
\substep{1.1} Weak compactness: proof of
\begin{equation}\label{eq:weak-conv-phieps-phi}
\mathds1_{\R^d\setminus F}\nabla\varphi_{\e}\cvf\mathds1_{\R^d\setminus F}\nabla\varphi,\qquad\text{weakly in $\Ld^2(\Omega)$.}
\end{equation}
Let $1\le i\le d$ be fixed. First note the following a priori estimate for $\varphi_{\e i}$,
\begin{equation}\label{eq:apriori-phieps}
\e^2\E[|\varphi_{\e i}|^2]+\E[\mathds1_{\R^d\setminus F}|\nabla\varphi_{\e i}|^2]+\e^2\E[\mathds1_F|\nabla\varphi_{\e i}|^2]\,\lesssim\,1.
\end{equation}
By Assumption~\ref{ass1},
as $\mathds1_{\R^d\setminus F}\nabla\varphi_{\e i}$ is bounded in $\Ld^2(\Omega)$, we can construct a stationary field $\hat\varphi_{\e i}:=P\varphi_{\e i}\in W^{1,p}_\loc(\R^d;\Ld^p(\Omega))$ such that $\hat \varphi_{\e i}=\varphi_{\e i}$ in $\R^d\setminus F$ and such that $\nabla\hat\varphi_{\e i}$ is bounded in $\Ld^p(\Omega)$.
Hence, by weak compactness, there is a stationary potential random field $\Phi_i\in\Ld^p(\Omega)$ with $\E[\Phi_i]=0$ such that $\nabla\hat\varphi_{\e i}\cvf\Phi_i$ in $\Ld^p(\Omega)$. By the definition of the extension, this implies $\mathds1_{\R^d\setminus F}\nabla\varphi_{\e i}=\mathds1_{\R^d\setminus F}\nabla\hat\varphi_{\e i}\cvf\mathds1_{\R^d\setminus F}\Phi_i$ in $\Ld^p(\Omega)$. By the boundedness of $\mathds1_{\R^d\setminus F}\nabla\varphi_{\e i}$ in $\Ld^2(\Omega)$, this weak convergence must actually hold in~$\Ld^2(\Omega)$.
Testing the corrector equation~\eqref{eq:approx-cor-eps} for $\varphi_{\e i}$ with some stationary field $\psi \in \Ld^2(\Omega;H^1_\loc(\R^d))$, passing to the limit $\e \downarrow 0$, and using the a priori estimates~\eqref{eq:apriori-phieps}, we get
\[
\expec{\nabla \psi \cdot (\Phi_{i}+e_i) \mathds1_{\R^d \setminus F}}=0,
\]
and therefore, by density, for all stationary potential random fields $\Psi\in \Ld^2(\Omega)$ with $\E[\Psi]=0$,
\[
\expec{\Psi \cdot (\Phi_{i}+e_i) \mathds1_{\R^d \setminus F}}=0.
\]
As this is the weak formulation of equation~\eqref{eq:corr} in the probability space, we conclude by uniqueness that $\mathds1_{\R^d \setminus F}\Phi_i =\mathds1_{\R^d \setminus F}\nabla \varphi_i $, thus concluding the claimed weak convergence~\eqref{eq:weak-conv-phieps-phi}.

\medskip
\substep{1.2} Proof of~\eqref{eq:lim-masscorr}.\\
We appeal to an energy argument. From the equations for $\varphi_{\e i}$ and $\varphi_i$, we find respectively
\begin{eqnarray*}
\e^2\E[|\varphi_{\e i}|^2]+\e^2\E[\mathds1_F|\nabla\varphi_{\e i}|^2]+\E[\mathds1_{\R^d\setminus F}|\nabla\varphi_{\e i}|^2]&=&-e_i\cdot\E[\mathds1_{\R^d\setminus F}\nabla\varphi_{\e i}]-\e^2e_i\cdot\E[\mathds1_F\nabla\varphi_{\e i}],\\
\E[\mathds1_{\R^d\setminus F}|\nabla\varphi_i|^2]&=&-e_i\cdot\E[\mathds1_{\R^d\setminus F}\nabla\varphi_i],
\end{eqnarray*}
and thus,
\begin{multline*}
\e^2\E[|\varphi_{\e i}|^2]+\e^2\E[\mathds1_F|\nabla\varphi_{\e i}|^2]+\E[\mathds1_{\R^d\setminus F}|\nabla\varphi_{\e i}-\nabla\varphi_i|^2]\\
\,=\,-\E[\mathds1_{\R^d\setminus F}(2\nabla\varphi_i+e_i)\cdot(\nabla\varphi_{\e i}-\nabla\varphi_i)]-\e^2e_i\cdot\E[\mathds1_F\nabla\varphi_{\e i}].
\end{multline*}
From the weak convergence established in Step~1.1, we can pass to the limit in the right-hand side and the conclusion~\eqref{eq:lim-masscorr} follows.

\medskip
\step2 Approximate Dirichlet correctors.\\
In order to obtain an almost sure homogenization result, the almost sure sublinearity of correctors is well-known to be a key step. When using approximate correctors of Step~1, this would require to prove the strong convergence $\e\varphi_\e(\tfrac\cdot\e)\to0$ in $L^2_\loc(\R^d)$ almost surely as $\e\downarrow0$, and in order to remove the approximation we would further need $\nabla\varphi_\e(\frac\cdot\e)-\nabla\varphi(\frac\cdot\e)\to0$ in $L^2_\loc(\R^d)$ almost surely.
As shown in Step~1, these two convergences are known to hold in $\Ld^2(\Omega)$, hence in $L^2_\loc(\R^d)$ in probability. In order to obtain almost sure convergences, we would typically need to combine this with an application of the ergodic theorem for large-scale averages. However, this would require to consider a diagonal limit combining both large-scale averaging and the small-mass limit, which is a priori out of reach since we are deprived of any quantitative statement. To get around this problem, we shall rather make use of the approximate corrector with {\it Dirichlet} boundary conditions on $D/\e$.

More precisely, for any $\e>0$ and $1\le i\le d$, we define $\varphi_{\e i}^\circ\in H^1_0(\frac1\e D;L^2(\Omega))$ as the unique almost sure weak solution of the following approximate corrector equation with Dirichlet boundary conditions,\footnote{Although considering an equation in $D/\e$, note that we take coefficient $\mathds1_{\R^d\setminus F}+\e^2\mathds1_F$ instead of $1-\chi_\e+\e^2\chi_\e$, which will be convenient in Step~2.1 below.}
\begin{equation}\label{eq:approx-cor-eps0}
\e^2 \varphi_{\e i}^\circ-\nabla \cdot (\mathds1_{\R^d\setminus F}+\e^2 \mathds1_F) (e_i+\nabla \varphi_{\e i}^\circ)=0,\qquad\text{in $\tfrac1\e D$},
\end{equation}
and we then set $\varphi_\e^\circ:=\{\varphi_{\e i}^\circ\}_{1\le i\le d}$.
In these terms, we show the following almost sure version of the convergence properties of Step~1: for any ball $B\subset D$ we have almost surely,
\begin{equation}\label{e.000}
\lim_{\e \downarrow 0}
\Big(\int_D\e^2|\varphi_\e^\circ(\tfrac\cdot\e)|^2+\int_D \mathds1_{\R^d \setminus \e F} |\nabla \varphi_\e^\circ(\tfrac \cdot \e)-\nabla \varphi(\tfrac \cdot \e)|^2+\int_D\e^2\mathds1_{\e F}|\nabla\varphi_\e^\circ(\tfrac\cdot\e)|^2\Big)\,=\,0.
\end{equation}
Without loss of generality, let us focus on correctors in the direction $e_1$.

\medskip
\substep{2.1} Application of the subadditive ergodic theorem.\\
Consider the random set function $\calF$ given for all bounded open sets $O\subset\R^d$ by
\begin{eqnarray*}
\calF(O)\,:=\,\inf_{\psi \in H^1_0(O)} \int_O \Big(\tfrac{\gamma^2}{d_O^2} \psi^2+ (\mathds1_{\R^d\setminus F}+\tfrac{\gamma^2}{d_O^2}\mathds1_F) |e_1+\nabla \psi|^2\Big),
\end{eqnarray*}
where $d_O:=\diam (O)$ and $\gamma>0$ will be chosen later. This is a rather unusual form for an integral functional since the integrand depends itself on the set.   
Since $F$ is stationary, $\calF$  is also stationary.
So defined, the set function $\calF$ is subadditive in the sense that for all disjoint bounded open sets $O_1,\dots,O_N$ we have
\[
\calF(\cup_{i=1}^N O_i) \le \sum_{i=1}^N \calF(O_i),
\]
which indeed follows from the gluing properties of Dirichlet conditions and from $d_{\cup_{i=1}^N O_i} \ge \max_{1\le i \le N} d_{O_i}$.
In addition, note that
\[\calF(O) \le |O| (1 \vee \tfrac{\gamma^2}{d_O^2}).\]
In this setting, the subadditive ergodic theorem entails that there exists some $\bar \calF\ge 0$ such that almost surely we have for all bounded open sets $O\subset\R^d$,
\begin{equation}\label{e.suberg}
\lim_{R\uparrow \infty} |RO|^{-1} \calF(RO)
 \,=\, \lim_{R \uparrow \infty} \expec{|RO|^{-1} \calF(RO)} \,=\,\bar \calF.
\end{equation}

\medskip
\substep{2.2} Convergence of energies: proof that almost surely, 
\begin{equation}\label{eq:conv-cor-red30}
\lim_{\e\downarrow0}\fint_{D}\Big(\e^2\varphi_{\e1}^\circ(\tfrac\cdot\e)^2
+(\mathds1_{\R^d\setminus\e F}+\e^2\mathds1_{\e F})|e_1+\nabla\varphi_{\e 1}^\circ(\tfrac\cdot\e)|^2\Big)
\,=\,\E\Big[\mathds1_{\R^d\setminus F}|e_1+\nabla\varphi_1|^2\Big].
\end{equation}
By~\eqref{e.suberg} with $R=\frac1\e$, $O=D$, and $\gamma=\diam(D)$, as the Dirichlet corrector $\varphi_{\e1}^\circ$ is the minimizer for $\calF(D/\e)$, it suffices to show that the limit in~\eqref{e.suberg} coincides with
\begin{equation}\label{ea.equal}
\bar \calF\,=\,\E\Big[\mathds1_{\R^d\setminus F}|\nabla\varphi_1+e_1|^2\Big].
\end{equation}
We start with the proof of the upper bound.
The definition of $\calF$ reads
\[\calF(D/\e)\,=\,\inf_{\psi\in H^1_0(D/\e)}G_\e(\psi),\qquad G_\e(\psi)\,:=\,\int_{D/\e}\Big(\e^2\psi^2+(\mathds1_{\R^d\setminus F}+\e^2\mathds1_F)|e_1+\nabla\psi|^2\Big).\]
For all $\kappa>0$, choose a smooth cutoff function~$\zeta_{\kappa}$ supported in $D$ such that $\zeta_{\kappa}=1$ in $D_\kappa:=\{x\in D:\dist(x,\partial D)>\kappa\}$ and such that $0\le \zeta_{\kappa}\le1$ and $|\nabla \zeta_{\kappa}|\lesssim 1/\kappa$.
In these terms, by minimality of $\varphi_{\e1}^\circ$, we can bound
\begin{equation}\label{eq:upper-FD}
\calF(D/\e)\,=\,G_\e(\varphi_{\e1}^\circ)\,\le\, G_\e(\zeta_{\kappa}(\e\cdot)\varphi_{\e 1}),
\end{equation}
where we recall that $\varphi_{\e1}$ stands for the stationary corrector of Step~1.
Expanding the expression in the right-hand side, we find
\begin{multline*}
G_\e(\zeta_{\kappa}(\e\cdot)\varphi_{\e1})\,=\, \int_{D/\e}\Big( \e^2\zeta_{\kappa}(\e\cdot)^2 \varphi_{\e1}^2 +(\mathds1_{\R^d\setminus F}+\e^2\mathds1_F)|e_1+\nabla (\zeta_{\kappa}(\e\cdot) \varphi_{\e1})|^2\Big)\\
\,\le\,\int_{D/\e} \Big(\e^2\varphi_{\e1}^2 + (\mathds1_{\R^d\setminus F}+\e^2\mathds1_F)|e_1+\nabla \varphi_{\e1}|^2\Big)\\
+\int_{D/\e} (\mathds1_{\R^d\setminus F}+\e^2\mathds1_F) \Big(2e_1+(1+\zeta_{\kappa}(\e\cdot))\nabla \varphi_{\e1}+\e\varphi_{\e1} (\nabla \zeta_{\kappa})(\e\cdot)\Big)\cdot \Big((\zeta_{\kappa}(\e\cdot)-1)\nabla \varphi_{\e1}+\e\varphi_{\e1} (\nabla \zeta_{\kappa})(\e\cdot)\Big).
\end{multline*}
Taking the expectation, passing to the limit $\e\downarrow0$, using the stationarity of $\varphi_{\e1}$, and appealing to the result~\eqref{eq:lim-masscorr} of Step~1, we get 
\[
\limsup_{\e \downarrow 0} |D/\e|^{-1} \E[G_\e(\zeta_{\kappa}(\e\cdot)\varphi_{\e1})]\,\le\,\E\Big[\mathds1_{\R^d\setminus F}|e_1+\nabla\varphi_1|^2\Big]
+|D\setminus D_\kappa|\,\E\Big[1+\mathds1_{\R^d\setminus F}|\nabla\varphi_1|^2\Big].\]
Recalling~\eqref{eq:upper-FD}, this provides an upper bound on $\limsup_\e|D/ \e|^{-1}\E[\calF(D/\e)]$, while the latter is equal to~$\bar \calF$ by~\eqref{e.suberg}. Taking $\kappa$ arbitrarily small, we are thus led to
\begin{equation}\label{ea.ub}
\bar \calF \,\le\, \E\Big[\mathds1_{\R^d\setminus F}|e_1+\nabla\varphi_1|^2\Big].
\end{equation}
We turn to the proof of the corresponding lower bound.
Note that
\[
 G_\e(\psi)\,\ge \, G_0 (\psi)\,=\,\int_{D/\e} \mathds1_{\R^d\setminus F}|e_1+\nabla\psi|^2,
\]
and recall that homogenization for the soft-inclusion problem precisely entails
\[
\lim_{\e \downarrow 0} \E\bigg[\inf_{\psi \in H^1_0(D/\e)} |D/\e|^{-1} G_0 (\psi)\bigg]\,=\,\E\Big[\mathds1_{\R^d\setminus F}|e_1+\nabla\varphi_1|^2\Big].
\]
This yields the lower bound
\begin{equation}\label{ea.lb}
\bar \calF \,\ge\, \E\Big[\mathds1_{\R^d\setminus F}|e_1+\nabla\varphi_1|^2\Big].
\end{equation}
Combined with~\eqref{ea.ub}, this proves the claim~\eqref{ea.equal}, hence~\eqref{eq:conv-cor-red30}.

\medskip
\substep{2.3} Proof of~\eqref{e.000}.\\
Reorganizing of the square, we can write
\begin{multline*}
\fint_{D}\Big(\e^2\varphi_{\e 1}^\circ(\tfrac\cdot\e)^2+\mathds1_{\R^d\setminus \e F}|\nabla(\varphi_{\e 1}^\circ-\varphi_1)(\tfrac\cdot\e)|^2+\e^2\mathds1_{\e F}|e_1+\nabla\varphi_{\e 1}^\circ(\tfrac\cdot\e)|^2\Big)\\
\,=\,\fint_{D}\Big(\e^2\varphi_{\e 1}^\circ(\tfrac\cdot\e)^2
+(\mathds1_{\R^d\setminus\e F}+\e^2\mathds1_{\e F})|e_1+\nabla\varphi_{\e 1}^\circ(\tfrac\cdot\e)|^2\Big)
+\fint_{D}\mathds1_{\R^d\setminus\e F}|e_1+\nabla\varphi_1(\tfrac\cdot\e)|^2\\
-2\fint_{D}\mathds1_{\R^d\setminus\e F}(e_1+\nabla\varphi_{\e 1}^\circ(\tfrac\cdot\e))\cdot(e_1+\nabla\varphi_1(\tfrac\cdot\e)),
\end{multline*}
and thus, integrating by parts in the last term and using the corrector equation for $\varphi_1$,
\begin{multline*}
\fint_{D}\Big(\e^2\varphi_{\e 1}^\circ(\tfrac\cdot\e)^2+\mathds1_{\R^d\setminus \e F}|\nabla(\varphi_{\e 1}^\circ-\varphi_1)(\tfrac\cdot\e)|^2+\e^2\mathds1_{\e F}|e_1+\nabla\varphi_{\e1}^\circ(\tfrac\cdot\e)|^2\Big)\\
\,=\,\fint_{D}\Big(\e^2\varphi_{\e1}^\circ(\tfrac\cdot\e)^2
+(\mathds1_{\R^d\setminus\e F}+\e^2\mathds1_{\e F})|e_1+\nabla\varphi_{\e1}^\circ(\tfrac\cdot\e)|^2\Big)
+\fint_{D}\mathds1_{\R^d\setminus\e F}|e_1+\nabla\varphi_1(\tfrac\cdot\e)|^2\\
-2\fint_{D}\mathds1_{\R^d\setminus\e F}e_1\cdot(e_1+\nabla\varphi_1(\tfrac\cdot\e)).
\end{multline*}
Appealing to the result~\eqref{eq:conv-cor-red30} of Step~2.2 for the limit of the first right-hand side term, using the ergodic theorem to further pass to the limit in the last two terms, and using the corrector equation for~$\varphi_1$ in the probability space, we obtain almost surely
\begin{multline*}
\lim_{\e\downarrow0}\fint_{D}\Big(\e^2\varphi^\circ_{\e1}(\tfrac\cdot\e)^2+\mathds1_{\R^d\setminus \e F}|\nabla(\varphi^\circ_{\e1}-\varphi_1)(\tfrac\cdot\e)|^2+\e^2\mathds1_{\e F}|e_1+\nabla\varphi^\circ_{\e1}(\tfrac\cdot\e)|^2\Big)\\
\,=\,2\E\Big[\mathds1_{\R^d\setminus F}|e_1+\nabla\varphi_1|^2\Big]
-2\E\Big[\mathds1_{\R^d\setminus F}e_1\cdot(e_1+\nabla\varphi_1)\Big]
\,=\,0,
\end{multline*}
that is,~\eqref{e.000}.

\medskip
\step3 Oscillating test functions. \\
Given $\psi\in C^\infty_c(D)$, testing the equation for $u_\e$ with $\psi+\e\varphi_{\e i}^\circ(\tfrac\cdot\e)\nabla_i\psi$, we obtain
\begin{multline*}
\int_D u_\e(\psi+\e\varphi_{\e i}^\circ(\tfrac\cdot\e)\nabla_i\psi)+\int_D(\nabla_i\psi)(e_i+\nabla\varphi_{\e i}^\circ(\tfrac\cdot\e))\cdot(1-\chi_\e+\e^2\chi_\e)\nabla u_\e\\
+\int_D\e\varphi_{\e i}^\circ(\tfrac\cdot\e)(\nabla \nabla_i\psi)\cdot(1-\chi_\e+\e^2\chi_\e)\nabla u_\e = \int_D f(\psi+\e\varphi_{\e i}^\circ(\tfrac\cdot\e)\nabla_i\psi).
\end{multline*}
Integrating by parts in the second term, using the approximate corrector equation~\eqref{eq:approx-cor-eps0}, and noting that $\chi_\e=\mathds1_{\e F}$ on the support of $\psi$ for $\e$ small enough, this entails
\begin{multline*}
\int_D u_\e\psi
-\int_D u_\e (1-\chi_\e+\e^2 \chi_\e)(e_i+\nabla\varphi_{\e i}^\circ(\tfrac\cdot\e))\cdot(\nabla\nabla_i\psi)\\
+\int_D\e\varphi_{\e i}^\circ(\tfrac\cdot\e)(\nabla \nabla_i\psi)\cdot(1-\chi_\e+\e^2\chi_\e)\nabla u_\e = \int_D f(\psi+\e\varphi_{\e i}^\circ(\tfrac\cdot\e)\nabla_i\psi).
\end{multline*}
By the Cauchy--Schwarz inequality and by the convergence properties~\eqref{e.000} of the approximate corrector, this yields in the limit $\e\downarrow0$, almost surely,
\begin{equation}\label{eq:conv-tartar}
\lim_{\e\downarrow0}\Big(\int_D u_\e\psi
-\int_D u_\e (1-\chi_\e)(e_i+\nabla\varphi_{i}(\tfrac\cdot\e))\cdot(\nabla\nabla_i\psi)\Big)
 = \int_D f \psi.
\end{equation}
It remains to pass to the limit in both left-hand side terms.

\medskip
\step4 Conclusion. \\
By~\eqref{eq:compact-baru} in form of $\|u_\e - \bar u-\chi_\e v_\e\|_{L^p(D)} \to 0$ and by the ergodic theorem of Lemma~\ref{lem:control Poincare constants} for $v_\e$, we find for the first summand in~\eqref{eq:conv-tartar},
\begin{equation*}
\lim_{\e\downarrow0}\int_D u_\e\psi 
\,=\,\E[v]\int_D(f-\bar u)\psi+\int_D\bar u\psi.
\end{equation*}
Let us now turn to the second summand in~\eqref{eq:conv-tartar}.
We appeal to a truncation argument and set $u_{\e,N}:=(u_\e\vee(-N))\wedge N$ for $N\ge1$.
The strong convergence~\eqref{eq:compact-baru} in $\Ld^p(D)$ entails that, along a subsequence (not relabeled), there exists a Lebesgue-negligible set $Z \subset D$ such that
$$
| u_\e (x) - \bar u (x) -\chi_\e(x) v_\e(x)| \xrightarrow{\e \downarrow0}0, \qquad \text{for all } x \in D\setminus Z.
$$
Egorov's theorem entails in turn that, for fixed $N\geq 1$, there exists $Z \subset Z_N \subset D$ such that $\left\lvert Z_N \right\rvert \leq 2^{-N}$ and there exists $\e(N)>0$ such that 
$$
| u_\e (x) - \bar u (x) -\chi_\e(x) v_\e(x)| \,\leq\, \tfrac N2, \qquad \text{for all $x \in D \setminus Z_N$ and all $0<\e<\e(N)$.}
$$
Hence, by the triangle inequality, for all $x \in D \setminus Z_N $ and all $ 0<\e<\e(N)$,
\begin{equation}\label{eq:aux1}
( 1 - \chi_\e(x) )\mathds{1}_{|u_\e(x)| \geq N}\,\leq \,\mathds{1}_{|\bar u(x)| \geq \frac{N}{2}}.
\end{equation}
In particular, by the Cauchy--Schwarz inequality, by~\eqref{eq:aux1}, and by the a priori estimate for $u_\e$ in~$L^2(D)$,
\begin{multline*}
\bigg|\int_D u_\e  (1-\chi_\e)(e_i+\nabla\varphi_i(\tfrac\cdot\e))\cdot(\nabla\nabla_i\psi)-\int_D u_{\e,N}  (1-\chi_\e)(e_i+\nabla\varphi_i(\tfrac\cdot\e))\cdot(\nabla\nabla_i\psi)\bigg|\\
\,\lesssim\,\Big(\int_D
|u_\e|^2\Big)^\frac12\Big(\int_D(1 - \chi_\e)\mathds{1}_{| u_\e| \geq N}|e_i+\nabla\varphi_i(\tfrac\cdot\e)|^2\Big)^\frac12
\,\lesssim \,\Big(\int_D \big( \mathds{1}_{Z_N} + \mathds{1}_{|\bar u| \geq \frac N2} \big)  |e_i+\nabla\varphi_i(\tfrac\cdot\e)|^2\Big)^\frac12.
\end{multline*}
Now using the ergodic theorem to pass to the limit $\e \downarrow 0$ and then using the monotone convergence theorem to pass to the limit $N\uparrow \infty$, we find
\begin{multline*}
\limsup_{N\uparrow\infty}\limsup_{\e\downarrow0}\bigg|\int_D u_\e  (1-\chi_\e)(e_i+\nabla\varphi_i(\tfrac\cdot\e))\cdot(\nabla\nabla_i\psi)-\int_D u_{\e,N}  (1-\chi_\e)(e_i+\nabla\varphi_i(\tfrac\cdot\e))\cdot(\nabla\nabla_i\psi)\bigg|\\
\,\lesssim \, \E\big[ |e_i + \nabla \varphi_i |^2\big]^\frac12  \lim_{N\uparrow\infty}\Big(2^{-N} + \int_D \mathds{1}_{|\bar u| \geq \frac N2} \Big)^\frac12 \, = \,0.
\end{multline*}
Note that for fixed $N$ it follows from~\eqref{eq:compact-baru} that $(1-\chi_\e)(u_{\e,N}-\bar u_N)\to0$ in $\Ld^2(D)$ as $\e\downarrow0$, where we have similarly defined $\bar u_N:=(\bar u\vee(-N))\wedge N$. Further using the ergodic theorem and the definition of~$\bar a$, cf.~\eqref{eq:def-bara-re}, we may then infer
\begin{equation*}
\lim_{N\uparrow\infty}\limsup_{\e\downarrow0}\bigg|\int_D u_\e(1-\chi_\e)(e_i+\nabla\varphi_i(\tfrac\cdot\e))\cdot(\nabla\nabla_i\psi)
-\int_D \bar u_{N}e_i\cdot\bar a (\nabla\nabla_i\psi)\bigg|
\,=\,0.
\end{equation*}
Passing to the limit $N\uparrow\infty$ in $\bar u_N$ and integrating by parts, we conclude
\begin{equation*}
\lim_{\e\downarrow0}\int_D u_\e(1-\chi_\e)(e_i+\nabla\varphi_i(\tfrac\cdot\e))\cdot(\nabla\nabla_i\psi)
\,=\,-\int_D \nabla \bar u \cdot\bar a \nabla\psi.
\end{equation*}
Inserting these different limits into~\eqref{eq:conv-tartar}, we conclude
\begin{equation*}
(1-\E[v])\int_D\bar u\psi
+\int_D \nabla \bar u \cdot\bar a \nabla\psi = (1-\E[v])\int_D f \psi,
\end{equation*}
meaning that $\bar u$ satisfies the claimed homogenized problem~\eqref{eq:homog-prob}.
By~\cite[Lemma~8.8]{JKO94}, the extension property of Assumption~\ref{ass1} together with the assumption $\E[\mathds1_F]<1$ implies that $\bar a$ is positive definite, so that $\bar u$ is uniquely defined as the solution of the homogenized problem and belongs to~$H^1_0(D)$. This allows to get rid of the extractions, which concludes the proof.
\qed

\subsection{Proof of Theorem~\ref{th:corr-qual}}\label{sec:thcorrqual}
We consider a modulated energy between $u_\e$ and its two-scale expansion,
where we use a fixed cut-off function $\rho\in C^\infty_c(D)$ to truncate the corrector in a neighborhood of the boundary of $D$,
\[\int_D(u_\e-\bar u_\e-v_\e)^2+\int_D(1-\chi_\e+\e^2\chi_\e)\big|\nabla(u_\e-\bar u_\e-v_\e-\e\varphi^\circ_{\e i}(\tfrac\cdot\e)\rho\nabla_i\bar u_\e)\big|^2,\]
where we recall that $\varphi^\circ_\e$ is the approximate Dirichlet corrector defined in Step~2 of the proof of Theorem~\ref{th:main}
and that $\bar u_\e$ is the approximation of $\bar u$ chosen in the statement.
Expanding the squares and recalling that $v_\e\in H^1_0(F_\e(D))$, we find
\begin{multline*}
\int_D(u_\e-\bar u_\e-v_\e)^2+\int_D(1-\chi_\e+\e^2\chi_\e)\big|\nabla(u_\e-\bar u_\e-v_\e-\e\varphi_{\e i}^\circ(\tfrac\cdot\e)\rho\nabla_i\bar u_\e)\big|^2\\
\,=\,\int_D\Big(u_\e^2
+(1-\chi_\e+\e^2\chi_\e)|\nabla u_\e|^2\Big)
+\int_D\Big(v_\e^2
+\e^2\chi_\e|\nabla v_\e|^2\Big)
+\int_D\bar u_\e(\bar u_\e+2v_\e)-2\int_D u_\e(\bar u_\e+v_\e)\\
+\int_D(1-\chi_\e+\e^2\chi_\e)\Big|\nabla\bar u_\e+(\nabla\varphi_{\e i}^\circ)(\tfrac\cdot\e)\rho\nabla_i\bar u_\e+\e\varphi_{\e i}^\circ(\tfrac\cdot\e)\nabla(\rho\nabla_i\bar u_\e)\Big|^2\\
-2\int_D(1-\chi_\e+\e^2\chi_\e)\e\varphi_{\e i}^\circ(\tfrac\cdot\e)\nabla(\rho\nabla_i\bar u_\e)\cdot\nabla u_\e
+2\int_D\e^2\chi_\e\e\varphi_{\e i}^\circ(\tfrac\cdot\e)\nabla(\rho\nabla_i\bar u_\e)\cdot\nabla v_\e\\
+2\int_D\e^2\chi_\e\nabla v_\e\cdot\big(\nabla\bar u_\e+(\nabla\varphi_{\e i}^\circ)(\tfrac\cdot\e)\rho\nabla_i\bar u_\e \big)
-2\int_D(1-\chi_\e+\e^2\chi_\e)\nabla u_\e\cdot\big(\nabla\bar u_\e+(\nabla\varphi_{\e i}^\circ)(\tfrac\cdot\e)\rho\nabla_i\bar u_\e\big)\\
-2\int_D(1-\chi_\e+\e^2\chi_\e)\nabla u_\e\cdot\nabla v_\e.
\end{multline*}
Using the equations for $u_\e$ and $v_\e$ to rewrite both the first two right-hand side terms and the last two, and reorganizing the terms, we find
\begin{multline*}
\int_D(u_\e-\bar u_\e-v_\e)^2+\int_D(1-\chi_\e+\e^2\chi_\e)\big|\nabla(u_\e-\bar u_\e-v_\e-\e\varphi_{\e i}^\circ(\tfrac\cdot\e)\rho\nabla_i\bar u_\e)\big|^2\\
\,=\,
\int_Df(u_\e-\bar u_\e-v_\e)-\int_D\bar u_\e(f-\bar u_\e-v_\e)\\
+\int_D(1-\chi_\e+\e^2\chi_\e)\Big|\nabla\bar u_\e+(\nabla\varphi_{\e i}^\circ)(\tfrac\cdot\e)\rho\nabla_i\bar u_\e+\e\varphi_{\e i}^\circ(\tfrac\cdot\e)\nabla(\rho\nabla_i\bar u_\e)\Big|^2\\
+2\int_D\e^2\chi_\e\e\varphi_{\e i}^\circ(\tfrac\cdot\e)\nabla(\rho\nabla_i\bar u_\e)\cdot\nabla v_\e
+2\int_D\e^2\chi_\e\nabla v_\e\cdot\big(\nabla\bar u_\e+(\nabla\varphi_{\e i}^\circ)(\tfrac\cdot\e)\rho\nabla_i\bar u_\e\big)\\
-2\int_D(f-u_\e) \e \varphi_{\e i}^\circ(\tfrac\cdot\e)\rho\nabla_i\bar u_\e.
\end{multline*}
Now we aim to pass to the limit in the right-hand side.
As $\rho$ is compactly supported in~$D$, we note that $\chi_\e=\mathds1_{\e F}$ in the support of $\rho$ for $\e$ small enough.
Using the almost sure weak convergence $u_\e\cvf\bar u+\E[v](f-\bar u)$ in $L^2(D)$ as obtained in Theorem~\ref{th:main}, recalling that the ergodic theorem for $v_\e$ yields $v_\e\cvf\E[v](f-\bar u)$ in $\Ld^2(D)$ almost surely, using the energy estimates for~$u_\e$ and~$v_\e$ (and in particular $\e^2 \int_D \chi_\e |\nabla v_\e|^2 \lesssim 1$),
using the result~\eqref{e.000} of Step~2 of the proof of Theorem~\ref{th:main}, and recalling that the approximation $\bar u_\e$ of $\bar u$ is chosen such that $\bar u_\e \to \bar u$ in $\Ld^2(D)$ and
$\sup_\e (\|\nabla \bar u_\e\|_{\Ld^\infty(D)}+\e \|\nabla^2 \bar u_\e\|_{\Ld^\infty(D)})<\infty$, we find almost surely
\begin{multline}\label{eq:lim-modenergy}
\lim_{\e\downarrow0}\bigg(\int_D(u_\e-\bar u_\e-v_\e)^2+\int_D(1-\chi_\e+\e^2\chi_\e)\big|\nabla(u_\e-\bar u_\e-v_\e-\e\varphi_{\e i}^\circ(\tfrac\cdot\e)\rho\nabla_i\bar u_\e)\big|^2\\
-\int_D\mathds1_{\R^d\setminus\e F}\big|\nabla\bar u_\e+(\nabla\varphi_{i})(\tfrac\cdot\e)\rho\nabla_i\bar u_\e\big|^2\bigg)
\,=\,
-(1-\E[v])\int_D\bar u(f-\bar u).
\end{multline}
In order to pass to the limit in the last left-hand side term,
we decompose
\[\nabla\bar u_\e+(\nabla\varphi_{i})(\tfrac\cdot\e)\rho\nabla_i\bar u_\e\,=\,(e_i+\nabla\varphi_{i})(\tfrac\cdot\e)\rho\nabla_i\bar u_\e+(1-\rho)\nabla\bar u_\e,\]
we use that $\nabla \bar u_\e \to \nabla \bar u$ in $\Ld^2(D)$, appeal to the ergodic theorem, and we recognize the definition of $\bar a$, cf.~\eqref{eq:def-bara-re}, to the effect that almost surely
\begin{multline*}
\lim_{\e\downarrow0}\int_D\mathds1_{\R^d\setminus\e F}\big|\nabla\bar u_\e+(\nabla\varphi_{i})(\tfrac\cdot\e)\rho\nabla_i\bar u_\e\big|^2\\
\,=\,
\int_D\nabla\bar u\cdot\bar a\nabla\bar u
-\int_D(1-\rho)^2\nabla\bar u\cdot\bar a\nabla\bar u
+\E[\mathds1_{\R^d\setminus F}]\int_D(1-\rho)^2|\nabla\bar u|^2.
\end{multline*}
Inserting this into~\eqref{eq:lim-modenergy} and noting that the energy identity for the equation~\eqref{eq:homog-prob} for $\bar u$ reads
\[\int_D\nabla\bar u\cdot\bar a\nabla\bar u\,=\,(1-\E[v])\int_D\bar u(f-\bar u),\]
we deduce almost surely
\begin{multline*}
\lim_{\e\downarrow0}\bigg(\int_D(u_\e-\bar u_\e-v_\e)^2+\int_D(1-\chi_\e+\e^2\chi_\e)\big|\nabla(u_\e-\bar u_\e-v_\e-\e\varphi_{\e i}^\circ(\tfrac\cdot\e)\rho\nabla_i\bar u_\e)\big|^2\bigg)\\
\,=\,
-\int_D(1-\rho)^2\nabla\bar u\cdot\bar a\nabla\bar u
+\E[\mathds1_{\R^d\setminus F}]\int_D(1-\rho)^2|\nabla\bar u|^2\,\lesssim\,\int_D(1-\rho)^2|\nabla\bar u|^2.
\end{multline*}
Let us now remove the cut-off function $\rho$ in the left-hand side. Using the triangle inequality and arguing similarly as above to estimate the error, we are led to
\begin{equation*}
\lim_{\e\downarrow0}\bigg(\int_D(u_\e-\bar u_\e-v_\e)^2+\int_D(1-\chi_\e+\e^2\chi_\e)\big|\nabla(u_\e-\bar u_\e-v_\e-\e\varphi_{\e i}^\circ(\tfrac\cdot\e)\nabla_i\bar u_\e)\big|^2\bigg)\\
\,\lesssim\,\int_D(1-\rho)^2|\nabla\bar u|^2.
\end{equation*}
By the arbitrariness of $\rho$, the right-hand side can be made arbitrarily small, and we conclude almost surely
\begin{equation*}
\lim_{\e\downarrow0}\bigg(\int_D(u_\e-\bar u_\e-v_\e)^2+\int_D(1-\chi_\e+\e^2\chi_\e)\big|\nabla(u_\e-\bar u_\e-v_\e-\e\varphi_{\e i}^\circ(\tfrac\cdot\e)\nabla_i\bar u_\e)\big|^2\bigg)
\,=\,0.
\end{equation*}
As $v_\e$ vanishes outside $F_\e(D)$, this yields in particular
\begin{equation*}
\lim_{\e\downarrow0}\bigg(\int_D(u_\e-\bar u_\e-v_\e)^2+\int_D(1-\chi_\e)\big|\nabla(u_\e-\bar u_\e-\e\varphi_{\e i}^\circ(\tfrac\cdot\e)\nabla_i\bar u_\e)\big|^2\bigg)
\,=\,0.
\end{equation*}
Expanding the gradient in the second integral, using~\eqref{e.000} as above to neglect the terms involving $\e\varphi_{\e i}^\circ$ and to replace $\nabla\varphi_{\e i}^\circ$ by $\nabla\varphi_i$, we conclude almost surely
\begin{equation}\label{eq:res-conv-reg}
u_\e-\bar u_\e-v_\e\,\to\,0,\qquad(1-\chi_\e)\big(\nabla u_\e-(e_i+\nabla\varphi_i)(\tfrac\cdot\e)\nabla_i\bar u_\e\big)\,\to\,0,\qquad\text{in $L^2(D)$}.
\end{equation}
Combined with the result~\eqref{eq:conv-veps-lem} of Lemma~\ref{lem:control Poincare constants}
and with the convergence $\bar u_\e \to \bar u$ in $\Ld^2(D)$, the first convergence actually becomes
\begin{equation}\label{eq:res-conv-reg-re}
u_\e-\bar u-v(\tfrac\cdot\e)(f-\bar u)\,\to\,0,\qquad\text{in $L^2(D)$}.
\end{equation}
\qed

\section{Proof of quantitative error estimates}\label{sec:quant}
This section is devoted to the proof of Theorem~\ref{th:quant}, which we split into two main steps.

\medskip
\step1 Estimates inside inclusions: we prove the following post-processing of Lemma~\ref{prop:sol-inside},
\begin{eqnarray}
\|u_\e-\bar u-v_\e\|_{\Ld^2(F_\e(D))}&\lesssim&\|u_\e-\bar u\|_{\Ld^2(D\setminus F_\e(D))}+\e\|f\|_{\Ld^2(D)},\label{eq:L2est-err-ins}\\
\|\mathds1_{F_\e(D)}(u_\e-\bar u-v_\e)\|_{\dot H^{-1}(D)}&\lesssim&\e\|f\|_{\Ld^2(D)}, \label{eq:H1est-err-ins}
\end{eqnarray}
where we recall that $v_\e$ stands for the solution of the auxiliary problem~\eqref{eq:def-veps}.

We start with the proof of the $\Ld^2$ estimate~\eqref{eq:L2est-err-ins}.
By Assumption~\ref{ass3}, the extension condition~\ref{ass1} holds with $p=2$: given a ball~$B$ that contains the domain~$D$, extending $u_\e-\bar u$ by $0$ on $B\setminus D$, we may thus define
\[w_\e\,:=\,P_{B,\e} ((u_\e-\bar u) \mathds 1_{D \setminus F_\e(D)})|_D~\in~ H^{1}_0(D),\]
which satisfies $w_\e=u_\e-\bar u$ in $D\setminus F_\e(D)$ and
\[\|w_\e\|_{L^{2}(D)} \lesssim \|u_\e-\bar u\|_{L^2(D \setminus F_\e(D))},\qquad
\| \nabla w_\e\|_{L^{2}(D)} \lesssim \|\nabla (u_\e-\bar u)\|_{L^2(D \setminus F_\e(D))}.\]
In terms of the solution $v_\e$ of~\eqref{eq:def-veps}, note that we have
\begin{equation}\label{eq:ueps-veps}
u_\e-\bar u-v_\e-\e^2\triangle (u_\e-v_\e)=0,\qquad\text{in $F_\e(D)$}.
\end{equation}
Now subtracting $w_\e$, we find that $u_\e-\bar u-w_\e-v_\e$ belongs to $H^1_0(F_\e(D))$ and satisfies
\begin{equation*}
(u_\e-\bar u-w_\e-v_\e)-\e^2\triangle(u_\e-\bar u-w_\e-v_\e)=-w_\e+\e^2\triangle(\bar u+w_\e),\qquad\text{in $F_\e(D)$}.
\end{equation*}
By Lemma~\ref{lem:regp} with $p=2$, we deduce after rescaling
\begin{equation*}
\|u_\e-\bar u-w_\e-v_\e\|_{\Ld^2(F_\e(D))}
\,\lesssim\,\|w_\e\|_{\Ld^2(F_\e(D))}+\e\|\nabla w_\e\|_{\Ld^2(F_\e(D))}+\e\|\nabla\bar u\|_{\Ld^2(F_\e(D))}.
\end{equation*}
Hence, by the triangle inequality, by the properties of $w_\e$, and by the a priori estimates for $u_\e$ and $\bar u$, we get
\begin{eqnarray*}
\|u_\e-\bar u-v_\e\|_{\Ld^2(F_\e(D))}
&\le&\|w_\e\|_{\Ld^2(F_\e(D))}+\e\|\nabla w_\e\|_{\Ld^2(F_\e(D))}+\e\|\nabla\bar u\|_{\Ld^2(F_\e(D))}\\
&\lesssim& \|u_\e-\bar u\|_{\Ld^2(D\setminus F_\e(D))}+\e\|\nabla u_\e\|_{\Ld^2(D\setminus F_\e(D))} +\e\|\nabla \bar u\|_{\Ld^2(D)}\\
&\lesssim&\|u_\e-\bar u\|_{\Ld^2(D\setminus F_\e(D))}+\e\|f\|_{\Ld^2(D)},
\end{eqnarray*}
that is, \eqref{eq:L2est-err-ins}.

We turn to the proof of the $H^{-1}$ estimate~\eqref{eq:H1est-err-ins}.
Consider the solution $r_\e\in H^1(F_\e(D))$ of the auxiliary problem
\begin{equation}\label{e.aux}
\left\{\begin{array}{ll}
-\triangle r_\e=0,&\text{in $F_\e(D)$},\\
\partial_\nu r_\e=\partial_\nu(u_\e-v_\e),&\text{on $\partial F_\e(D)$},\\
\fint_{\e I_n}r_\e=0,&\forall n.
\end{array}\right.
\end{equation}
Using~\eqref{eq:ueps-veps}, the energy identity for~\eqref{e.aux} yields
\begin{eqnarray*}
\e^2\int_{F_\e(D)}|\nabla r_\e|^2
&=&\e^2\int_{\partial F_\e(D)}r_\e\partial_\nu(u_\e-v_\e)\\
&=&\int_{F_\e(D)}r_\e (u_\e-\bar u-v_\e)+\e^2\int_{F_\e(D)}\nabla r_\e\cdot \nabla(u_\e-v_\e).
\end{eqnarray*}
By the Poincar\'e--Wirtinger inequality for $r_\e$ in each inclusion $\e I_n\subset F_\e(D)$, recalling $\sup_n\calP_2(I_n)<\infty$ by Assumption~\ref{ass3}, we deduce
\begin{equation*}
\e\|\nabla r_\e\|_{\Ld^2(F_\e(D))}
\,\lesssim\,\|u_\e-\bar u-v_\e\|_{\Ld^2(F_\e(D))}+\e\|\nabla(u_\e-v_\e)\|_{\Ld^2(F_\e(D))}.
\end{equation*}
Given $h\in C^\infty_c(D)$,
using the equation~\eqref{e.aux} for $r_\e$ together with~\eqref{eq:ueps-veps}, we get
\begin{eqnarray*}
\int_{F_\e(D)}h(u_\e-\bar u-v_\e)
&=&\e^2\int_{\partial F_\e(D)}h\partial_\nu(u_\e-v_\e)-\e^2\int_{F_\e(D)}\nabla h\cdot \nabla(u_\e-v_\e)\\
&=&\e^2\int_{F_\e(D)}\nabla h\cdot \nabla r_\e-\e^2\int_{F_\e(D)}\nabla h\cdot \nabla(u_\e-v_\e),
\end{eqnarray*}
and thus, by the above a priori bound on $\nabla r_\e$, combined with the a priori estimates on $u_\e$ and $v_\e$,
\begin{eqnarray*}
\Big|\int_{F_\e(D)}h(u_\e-\bar u-v_\e)\Big|
&\lesssim&\|\nabla h\|_{\Ld^2(F_\e(D))}\Big(\e\|u_\e-\bar u-v_\e\|_{\Ld^2(F_\e(D))}+\e^2\|\nabla(u_\e-v_\e)\|_{\Ld^2(F_\e(D))}\Big)\\
&\lesssim&\|\nabla h\|_{\Ld^2(F_\e(D))}\Big(\e\|u_\e-\bar u-v_\e\|_{\Ld^2(F_\e(D))}+\e\|f\|_{\Ld^2(D)}\Big).
\end{eqnarray*}
Combined with the $\Ld^2$ estimate~\eqref{eq:L2est-err-ins} on $u_\e-\bar u-v_\e$, this proves~\eqref{eq:H1est-err-ins}.

\medskip
\step2 Conclusion: two-scale expansion error.\\
Given a cutoff parameter $\eta \in[\e,1]$ to be chosen later on, let $\rho:= \rho_{\eta}\in C^\infty_c(D)$ be a boundary cutoff with $\rho\equiv 1$ in $D_\eta := \{ x \in D \, : \, \dist(x,\partial D)>\eta \}$ and $|\nabla\rho|\lesssim\eta^{-1}$. 
Using the extension $P_{B,\e}$ on a ball~$B$ that contains the domain $D$ and extending $u_\e$ and $\bar u$ by $0$ on~$B\setminus D$, we can consider the modified two-scale expansion error
\[w_{\eta,\e}\,:=\, P_{B,\e}\Big(\big(u_\e-\bar u-\e\rho\varphi_i(\tfrac\cdot\e) \nabla_i\bar u\big)\mathds1_{D\setminus F_\e(D)}\Big)\Big|_D \quad\in~ H^1_0(D),\]
where we recall that $\varphi_i$ is the corrector for the soft-inclusion problem, cf.~\eqref{eq:corr}, and where $\bar u$ is the solution of the homogenized problem~\eqref{eq:homog-prob}. 
First, using the double-porosity equation~\eqref{eq:double-por-00}, we find that~$w_{\eta,\e}$ satisfies 
\begin{multline*}
(1-\chi_\e)w_{\eta,\e}-\nabla\cdot (1-\chi_\e)\nabla w_{\eta,\e} 
\,=\, f  - \chi_\e u_\e + \e^2 \nabla \cdot \chi_\e \nabla u_\e\\
- (1-\chi_\e)(\bar u + \e \rho \varphi_i(\tfrac\cdot\e) \nabla_i\bar u)  +\nabla \cdot (1-\chi_\e ) \nabla( \bar u + \e \rho \varphi_i(\tfrac\cdot\e) \nabla_i \bar u).
\end{multline*}
Next, noting that the corrector equation and the definition of the flux corrector in Lemma~\ref{lem:corr} yield
\begin{eqnarray*}
\nabla\cdot \Big( (1 - \chi_\e)( e_i+\nabla \varphi_i (\tfrac\cdot\e) )  \nabla_i \bar u\Big)
-\nabla\cdot\bar a \nabla\bar u
&=&\Big( (1 - \chi_\e)(e_i+ \nabla \varphi_i (\tfrac\cdot\e) ) - \bar a e_i \Big) \cdot \nabla \nabla_i \bar u\\
&=&(\nabla_k\sigma_{ijk})(\tfrac{\cdot}{\e}) \nabla^2_{ij} \bar u\\
&=&-\nabla \cdot \big(\e \sigma_i (\tfrac{\cdot}{\e}) \nabla \nabla_i \bar u \big)
\end{eqnarray*}
and further using the homogenized equation~\eqref{eq:homog-prob} for $\bar u$, we infer 
\begin{multline*}
\nabla \cdot (1-\chi_\e ) \nabla( \bar u + \e \rho \varphi_i(\tfrac\cdot\e) \nabla_i \bar u)
\,=\,-(1-\E[v])(f-\bar u)
-\nabla \cdot \Big((1-\chi_\e)(1-\rho) \nabla \varphi_i(\tfrac\cdot\e) \nabla_i \bar u\Big)\\
+ \e \nabla \cdot \Big((1-\chi_\e)\varphi_i(\tfrac\cdot\e)\nabla (\rho\nabla_i \bar u) -\sigma_i(\tfrac\cdot\e)\nabla \nabla_i \bar u \Big).
\end{multline*}
Inserting this into the above, and smuggling in $v_\e$, we deduce that $w_{\eta,\e}$ satisfies
\begin{multline*}
(1-\chi_\e)w_{\eta,\e}-\nabla\cdot(1-\chi_\e)\nabla w_{\eta,\e}
\,=\,\big(\E[v](f-\bar u) - v_\e\big)
-\chi_\e (u_\e-\bar u - v_\e)\\
 -\nabla \cdot \Big((1-\chi_\e)(1-\rho) \nabla \varphi_i(\tfrac\cdot\e) \nabla_i \bar u\Big)
+\e \nabla \cdot \Big(
(1-\chi_\e)\varphi_i(\tfrac\cdot\e)\nabla(\rho\nabla_i\bar u)
-\sigma_i(\tfrac\cdot\e)\nabla \nabla_i \bar u
\Big)\\
-\e (1-\chi_\e) \rho \varphi_i(\tfrac\cdot\e) \nabla_i\bar u + \e^2  \nabla \cdot \chi_\e \nabla u_\e.
\end{multline*}
We now appeal to the energy estimate for this equation.
In order to absorb the right-hand term involving the flux corrector $\sigma$ (which appears without a factor $1-\chi_\e$), we use Young's inequality together with the fact that
\begin{equation}\label{eq:good w}
\| \nabla w_{\eta,\e} \|_{\Ld^2(D)} \,\lesssim\, \| \nabla w_{\eta,\e} \|_{\Ld^2(D \setminus F_\e(D))},
\end{equation}
and we are then led to
\begin{multline}\label{eq:energy-wdiff}
\int_D(1-\chi_\e)(|w_{\eta,\e}|^2+|\nabla w_{\eta,\e}|^2)
\,\lesssim\,
\Big|\int_D \big(\E[v](f-\bar u) - v_\e\big) \,w_{\eta,\e}\Big|
+\Big|\int_{F_\e(D)} (u_\e-\bar u - v_\e)w_{\eta,\e}\Big|\\
+\int_{D\setminus D_\eta} \Big( |\nabla\varphi(\tfrac\cdot\e)|^2 + (\tfrac{\e}{\eta})^2 |\varphi(\tfrac\cdot\e)|^2 \Big)| \nabla \bar u|^2 
+\e^2 \int_D \big(|\varphi(\tfrac\cdot\e)|^2+|\sigma(\tfrac\cdot\e)|^2\big) |\nabla^2\bar u|^2
+\e^2 \int_D  \lvert f \rvert^2,
\end{multline}
where we also used the energy estimates on $u_\e$ for the last right-hand term. Note that the third right-hand term corresponds to the boundary layer, for which one will only get a bound $O(\e)$ after optimizing the choice of $\eta$. It remains to estimate the first two right-hand side terms.

We start by examining the first right-hand side term in~\eqref{eq:energy-wdiff}.
On the one hand, in terms of the inclusion corrector in Lemma~\ref{lem:corr}, we can write
\[v(\tfrac\cdot\e) (f-\bar u) - \E\left[v\right] (f-\bar u) = \nabla \cdot \big( \e \theta(\tfrac\cdot\e) (f-\bar u) \big) - \e \theta(\tfrac\cdot\e) \cdot \nabla (f-\bar u).\]
On the other hand, by definition of $v$, for all $n$, we find that $\bar v_{\e,n} :=v(\tfrac\cdot\e) \fint_{\e I_n} (f - \bar u) \in H^1_0(\e I_n)$ satisfies
\[
v_\e - \bar v_{\e,n} - \e^2 \triangle (v_\e - \bar v_{\e,n}) = f-\bar u - \fint_{\e I_n} (f-\bar u), \qquad \text{in $\e I_n$},
\]
hence, by the Poincar\'e--Wirtinger inequality, recalling $\sup_n\calP_2(I_n)<\infty$ by Assumption~\ref{ass3},
\[\| v_\e - \bar v_{\e,n} \|_{L^2(\e I_n)} \,\lesssim\, \e \| \nabla (f-\bar u)\|_{L^2(\e I_n)}.\]
Combining these two observations together with the Cauchy--Schwarz inequality, we deduce
\begin{multline*}
\int_D \big(\E[v](f-\bar u) - v_\e\big)\,w_{\eta,\e}
\,=\,\e \int_D \Big( (f-\bar u) \,\theta(\tfrac\cdot\e) \cdot \nabla w_{\eta,\e} + w_{\eta,\e} \theta(\tfrac\cdot\e) \cdot \nabla(f-\bar u) \Big) \\
+ \sum_n \int_{\e I_n} \bigg( (\bar v_{\e,n} -v_\e ) w_{\eta,\e} +v(\tfrac\cdot\e) w_{\eta,\e} \Big( (f-\bar u) - \fint_{\e I_n} (f-\bar u) \Big) \bigg)\\
\, \lesssim \,  \e \int_D |\theta(\tfrac\cdot\e)|\big( |f-\bar u|+ |\nabla (f - \bar u) | \big) \big( | \nabla w_{\eta,\e} | + | w_{\eta,\e} | \big)
+\e\sum_n \| \nabla (f - \bar u) \|_{L^2(\e I_n)} \| w_{\eta,\e}\|_{L^2(\e I_n)}
\end{multline*}
where we also used the fact that $0 \le v \le 1$ almost surely. Recalling~\eqref{eq:good w} and further using the Poincar\'e inequality for $w_{\eta,\e}$ in $D$, we conclude
\begin{multline}
\Big|\int_D \big(\E[v](f-\bar u) - v_\e\big)\,w_{\eta,\e}\Big|\\
\, \lesssim \, 
\e \Big(\|\theta(\tfrac\cdot\e)(f-\bar u)\|_{L^2(D)}+\|\theta(\tfrac\cdot\e)\nabla (f - \bar u)\|_{L^2(D)}+\| \nabla(f-\bar u)\|_{L^2(D)} \Big)\|\nabla w_{\eta,\e} \|_{L^2(D\setminus F_\e(D))}.\label{eq:control T2}
\end{multline}
It remains to examine the second right-hand side term in~\eqref{eq:energy-wdiff}.
Using the result~\eqref{eq:H1est-err-ins} of Step~1, together with~\eqref{eq:good w} once again, we find
\begin{equation}\label{eq:control T3}
\Big|\int_{F_\e(D)} (u_\e - \bar u - v_\e) w_{\eta,\e}\Big|
\,\lesssim \,  \e \|f\|_{\Ld^2(D)} \| \nabla w_{\eta,\e}\|_{\Ld^2(D)}
\,\lesssim \,\e \|f\|_{\Ld^2(D)} \| \nabla w_{\eta,\e}\|_{\Ld^2(D\setminus F_\e(D))} .
\end{equation}
Finally, inserting~\eqref{eq:control T2} and~\eqref{eq:control T3} into~\eqref{eq:energy-wdiff}, and appealing to Young's inequality, we are led to
\begin{multline*}
\int_D(1-\chi_\e)(|w_{\eta,\e}|^2+|\nabla w_{\eta,\e}|^2)
\,\lesssim\,\int_{D\setminus D_\eta} \Big( | \nabla \varphi(\tfrac\cdot\e)|^2 + (\tfrac{\e}{\eta})^2 |\varphi(\tfrac\cdot\e) |^2 \Big)  |\nabla \bar u|^2 \\
+ \e^2 \int_D \Big( 1 + | \theta(\tfrac\cdot\e) |^2 + | \varphi(\tfrac\cdot\e) |^2 + | \sigma(\tfrac\cdot\e)|^2 \Big)\Big( | f |^2+ | \bar u |^2 + | \nabla f |^2 + | \nabla \bar u |^2 + | \nabla^2 \bar u |^2 \Big).
\end{multline*}
Taking the expectation, recalling assumption~\eqref{eq:hyp bar u} and the bounds on correctors~\eqref{eq:moment bounds}, and using the energy estimate for $\bar u$, we infer
\begin{equation*}
\E \bigg[\int_D(1-\chi_\e)(|w_{\eta,\e}|^2+|\nabla w_{\eta,\e}|^2 )\bigg]
\, \lesssim \,
\e^2 \Big(\| f \|^2_{H^1(D)}+\|\nabla^2\bar u\|_{L^2(D)}^2\Big)+(1+\tfrac{\e}{\eta})^2|D \setminus D_\eta|\|\nabla\bar u\|_{L^\infty(D)}^2.
\end{equation*}
Noting that $|D\setminus D_\eta|\lesssim\eta$, the choice $\eta=\e$ allows to bound the right-hand side by $O(\e)$. Recalling the definition of $w_{\eta,\e}$, this means
\[\E \Big[\| u_\e - \bar u - \e \rho \varphi_i(\tfrac\cdot\e)\nabla_i \bar u \|^2_{H^1(D\setminus F_\e(D))}\Big]
\,\ \lesssim \,\ 
\e \Big(\| f \|^2_{H^1(D)}+\|\nabla^2\bar u\|_{L^2(D)}^2+\|\nabla\bar u\|_{L^\infty(D)}^2\Big).\]
Finally, noting similarly that
\[\E\Big[\|\e (1-\rho) \varphi_i(\tfrac\cdot\e)\nabla_i \bar u\|_{H^1(D\setminus F_\e(D)}^2\Big]\,\lesssim\,\e,\]
and further recalling the result~\eqref{eq:L2est-err-ins} of Step~1, the conclusion follows.
\qed

\appendix
\section{Discussion of the assumptions}\label{sec:app}
This appendix is devoted to the proofs of Lemmas~\ref{lem:ass1} and~\ref{lem:ass2} on the validity of our two main assumptions~\ref{ass1} and~\ref{ass2}.

\subsection{Proof of Lemma~\ref{lem:ass1}}\label{sec:app1}
We split the proof into five subsections, and consider items~(i)--(v) separately. In each case, the proof proceeds by first constructing the extension operator locally on suitable neighborhoods of the inclusions.

\subsubsection{Proof of~(i)}
This is a consequence of the classical extension procedure of~\cite[Chapter~VI]{Stein-70}. As the formulation of~\ref{ass1} is not quite standard, we include details for the reader's convenience.
For all~$n$, consider the neighborhood of inclusion $I_n$ given by
\begin{equation}\label{eq:def-Tn0}
T_n\,:=\,\Big\{x\in I_n+\tfrac1{2C_0}\diam(I_n)B^\circ\,:\,\dist\big(x,\partial(I_n+\tfrac1{2C_0}\diam(I_n)B^\circ)\big)>\tfrac1{4C_0}\diam(I_n)\Big\},
\end{equation}
where we recall the short-hand notation $B^\circ=B(0,1)$ for the unit ball at the origin.
By definition,~$T_n$ has a $C^2$ boundary and satisfies
\[I_n~\subset~T_n~\subset~I_n+\tfrac1{2C_0}\diam(I_n)B^\circ,\]
hence, by assumption,
\begin{equation}\label{eq:disj-Tn0}
T_n\cap T_m=\varnothing,\qquad\text{for all $n\ne m$.}
\end{equation}
From here, we split the proof into three steps: First, we appeal to a result of Stein~\cite{Stein-70} for the construction of an extension $ H^1(T_n\setminus I_n)\to H^1(T_n)$ around each inclusion. Next, the two parts of Assumption~\ref{ass1} are deduced by applying this extension locally around each inclusion.

\medskip
\step1 Construction of local extension operators $ H^1(T_n\setminus I_n)\to H^1(T_n)$.\\
Recall that for all $n$ the rescaled inclusion $I_n':=\diam(I_n)^{-1}I_n$ is assumed to satisfy the uniform Lipschitz condition in the statement, and note that $\diam(I_n')=1$.
Similarly, the rescaled neighborhood $T_n':=\diam(I_n)^{-1}T_n$ and the `annulus' $T_n'\setminus I_n'$ both satisfy the same condition (up to possibly increasing the constant~$C_0$).
Then appealing to~\cite[Chapter~VI, Theorem~5]{Stein-70}, there is a linear extension operator $P_n':H^1(T_n'\setminus I_n')\to H^1(T_n')$ such that $P_n'u=u$ in~$T_n'\setminus I_n'$ and
\[\|P_n'u\|_{H^1(T_n')}\,\lesssim\,\|u\|_{H^1(T_n'\setminus I_n')},\]
for some multiplicative constant only depending on $d,C_0$ (see indeed~\cite[Theorem~3.8]{CCV-21} for the dependence of the constant). 
Now defining
\[P_n''u:=P_n'(u-\textstyle\fint_{T_n'\setminus I_n'}u)+\textstyle\fint_{T_n'\setminus I_n'}u\quad\in~H^1(T_n'),\]
we find $P_n''u=u$ in $T_n'\setminus I_n'$ and, using the Poincar\'e--Wirtinger inequality in $T_n'\setminus I_n'$,
\[\|\nabla P_n''u\|_{L^2(T_n')}\,\lesssim\,\|u-\textstyle \fint_{T_n'\setminus I_n'}u\|_{H^1(T_n'\setminus I_n')}\,\lesssim\,\|\nabla u\|_{L^2(T_n'\setminus I_n')},\]
where the multiplicative constants still only depend on $d,C_0$.
Next, by homogeneity, we can rescale this estimate by $\diam(I_n)$: for $u\in H^1(T_n\setminus I_n)$, we define
\[P_nu\,:=\,P_n''\big(u(\diam(I_n)\cdot)\big)(\diam(I_n)^{-1}\cdot)~\in ~H^1(T_n),\]
which then satisfies $P_nu=u$ in $T_n\setminus I_n$ and
\begin{equation}\label{eq:prop-Pnloc}
\|\nabla P_nu\|_{L^2(T_n)}\,\lesssim\,\|\nabla u\|_{L^2(T_n\setminus I_n)},
\end{equation}
where the multiplicative constant only depends on $d,C_0$.

\medskip
\step2 Conclusion --- part~1.\\
We show the validity of the first part~\eqref{eq:ass1-1} of~\ref{ass1} with $p=2$ (hence, with any $1\le p\le2$).
Let $B\subset\R^d$ be a ball and let $0<\e<\diam(B)$ be fixed. In view of the assumptions on the inclusions $\{I_n\}_n$, we can construct modified inclusions $\{\tilde I_n\}_n$ that satisfy the same assumptions (up to possibly increasing $C_0$), such that $(\bigcup_n\tilde I_n)\cap\frac1\e B=(\bigcup_n I_n)\cap\frac1\e B$ and such that the corresponding neighborhoods $\{\tilde T_n\}_n$ constructed in~\eqref{eq:def-Tn0} are all included in $\frac2\e B$.
By Step~1, for all $n$, we can construct an extension operator $\tilde P_n:H^1(\tilde T_n\setminus \tilde I_n)\to H^1(\tilde T_n)$ such that $\tilde P_nu=u$ in $\tilde T_n\setminus \tilde I_n$ and
\[\|\nabla\tilde P_n u\|_{L^2(\tilde T_n)}\,\lesssim\,\|\nabla u\|_{L^2(\tilde T_n\setminus\tilde I_n)}.\]
Given $u\in H^1_0(B)$, extending it by $0$ on $2B\setminus B$, and recalling that the neighborhoods $\{\tilde T_n\}_n$ are pairwise disjoint, cf.~\eqref{eq:disj-Tn0}, we may then define
\[P_{B,\e} u\,:=\,u\mathds1_{B\setminus \e\cup_n\tilde T_n}+\sum_n\tilde P_n(u(\e\cdot))(\tfrac1\e\cdot)\mathds1_{\e\tilde T_n}~\in~H^1_0(2B).\]
By definition, it satisfies $P_{B,\e} u=u$ in $B\setminus\e(\cup_n \tilde I_n)=B\setminus\e (\cup_nI_n)$ and, by homogeneity,
\[\|\nabla P_{B,\e} u\|_{L^2(2B)}\,\lesssim\,\|\nabla u\|_{L^2(B\setminus\e(\cup_nI_n))},\]
where the multiplicative constant only depends on $d,C_0$.

\medskip
\step3 Conclusion --- part~2.\\
We turn to the validity of the second part~\eqref{eq:ass1-2} of~\ref{ass1} with $p=2$.
In terms of the local extension operators $\{P_n\}_n$ constructed in Step~1, we define $P:H^1_\loc(\R^d)\to H^{1}_\loc(\R^d)$ as follows,
\[Pu\,:=\,u\mathds1_{\R^d\setminus \cup_nT_n}+\sum_nP_n(u|_{T_n\setminus I_n})\mathds1_{T_n}~\in~H^1_\loc(\R^d).\]
By definition, $Pu=u$ in $\R^d\setminus\cup_nI_n$.
Up to using an arbitrary criterion to ensure uniqueness of local extensions (e.g.\@ using a minimality argument), we find that $Pu$ is a stationary random field whenever the pair $(u,\{I_n\}_n)$ is jointly stationary.
It remains to check that it satisfies the desired estimate~\eqref{eq:ass1-2} with~$p=2$.
For that purpose, in the spirit of~\cite[Lemma~2.5]{DG-23}, as a consequence of the ergodic theorem together with a simple approximation argument, we first note that expectations can be expanded as follows: if~$\zeta$ is a nonnegative random field such that the pair $(\zeta,\{I_n\}_n)$ is jointly stationary, then
\begin{equation}\label{eq:expand-expect}
\E[\zeta\mathds1_{\cup_nT_n}]\,=\,\E\bigg[\sum_n\frac{\mathds1_{0\in I_n}}{|I_n|}\int_{T_n}\zeta\bigg].
\end{equation}
In particular, given $u\in L^2(\Omega;H^1_\loc(\R^d))$ such that $(u,\{I_n\}_n)$ is jointly stationary, we can decompose
\[\E[|\nabla Pu|^2]\,=\,\E[|\nabla Pu|^2\mathds1_{\R^d\setminus\cup_nT_n}]+\E\bigg[\sum_n\frac{\mathds1_{0\in I_n}}{|I_n|}\int_{T_n}|\nabla Pu|^2\bigg].\]
By definition of $Pu$ together with the properties of the local extensions $\{P_n\}_n$, cf.~\eqref{eq:prop-Pnloc}, we get
\[\E[|\nabla Pu|^2]\,\lesssim\,\E[|\nabla u|^2\mathds1_{\R^d\setminus\cup_nT_n}]+\E\bigg[\sum_n\frac{\mathds1_{0\in I_n}}{|I_n|}\int_{T_n}|\nabla u|^2\bigg].\]
Appealing again to~\eqref{eq:expand-expect}, this means
\[\E[|\nabla Pu|^2]\,\lesssim\,\E[|\nabla u|^2],\]
which is the desired estimate~\eqref{eq:ass1-2} with $p=2$.\qed

\subsubsection{Proof of~(ii)}
This is the generalization of a result due to Zhikov~\cite[Lemma 8]{Zhikov-86} for spherical structures (see also~\cite[Section~8.4]{JKO94}). We briefly sketch the needed adaptations. The main question is to determine how the norm of the local extension operators constructed in the proof of~(i) depends on particle separation, that is, determine the best scaling in~\eqref{eq:prop-Pnloc} with respect to the separation distance.
For abbreviation, let $B_r := B(0,r)$ stand here for the ball of radius $r>0$ centered at~$0$. First,  as in the proof of~(i), we consider a Stein linear extension operator $P:H^1(B_2\setminus B_1) \to H^1(B_2)$ such that
$Pw=w$ on~$B_2\setminus B_1$ and
\[\|\nabla Pw\|_{L^2(B_2)} \lesssim \| \nabla w \|_{\Ld^2(B_2 \setminus B_1)}.\]
For later purposes, note that we can also ensure the $L^2$-control
\[\|Pw\|_{L^2(B_2)} \lesssim \|w \|_{\Ld^2(B_2 \setminus B_1)}.\]
Let $\{\phi_n'\}_n$ be the sequence of homeomorphisms $\phi_n':B_2\to\R^d$ defined in the statement, and consider the rescaled maps $\phi_n:=\diam(I_n)\phi_n'$ such that $\phi_n(B_1)=\diam(I_n)\,\phi_n'(B_1)=I_n$.
Also consider the Lipschitz homeomorphisms $\theta_n:\R^d\to\R^d$ given by
\[\theta_n(y)\,:=\,\left\{\begin{array}{lll}
y&:&|y|<1,\\
\frac{y}{|y|} \big( 1 +\frac{1}{C_0^2}\nu_n(|y| - 1) \big)&:&|y|>1,
\end{array}\right.\]
where we recall that the constant $C_0$ is such that $C_0^{-1}\le\|\nabla\phi_n'\|_{L^\infty}\le C_0$ for all $n$.
In these terms, we consider the neighborhood of inclusion $I_n$ given by
\[T_n\,:=\,\phi_n(B_{1+\frac{\nu_n}{C_0^2}})\,=\,(\phi_n\circ\theta_n)(B_2)~~\supset~~I_n=\phi_n(B_1)=(\phi_n\circ\theta_n)(B_1).\]
and we define, for $u\in H^1(T_n\setminus I_n)$,
\[P_nu\,:=\,(P(u\circ\phi_n\circ\theta_n|_{B_2\setminus B_1}))\circ(\phi_n\circ\theta_n)^{-1}\quad\in~H^1(T_n).\]
By definition, this satisfies $P_nu=u$ in $T_n\setminus I_n$ and
\begin{equation}\label{eq:extension-munu}
\int_{T_n}|P_nu|^2\,\lesssim\,\nu_n^{-1}\int_{T_n\setminus I_n}|u|^2,\qquad
\int_{T_n}|\nabla P_nu|^2\,\lesssim\,\nu_n^{-1}\int_{T_n\setminus I_n}|\nabla u|^2,
\end{equation}
where the multiplicative constant only depends on $d,C_0$.
In addition, note that
\[T_n\,=\,\phi_n(B_{1+\frac{\nu_n}{C_0^2}})~\subset~ I_n+\|\nabla\phi_n\|_{L^\infty}\tfrac{\nu_n}{C_0^2}B_1~\subset~ I_n+\tfrac12\rho_nB_1,\]
where we have used $\nu_n=\rho_n/D_n$ and $C_0^2\ge\|\nabla\phi_n'\|_{L^\infty}\|\nabla(\phi_n')^{-1}\|_{L^\infty}=\|\nabla\phi_n\|_{L^\infty}\|\nabla(\phi_n)^{-1}\|_{L^\infty}$  in the last inclusion, and observed that $D_n=\diam(I_n)\ge2\|\nabla(\phi_n)^{-1}\|_{L^\infty}^{-1}$.
By definition of the $\rho_n$'s, this entails
\[T_n\cap T_m=\varnothing\quad\text{for all $n\ne m$}.\]
Therefore, we can apply this extension procedure locally around each inclusion and, combining this with the moment condition on the $\nu_n$'s as in~\cite[Lemma 8]{Zhikov-86}, the conclusion follows similarly as in the proof of item~(i); we leave the remaining details to the reader.\qed

\subsubsection{Proof of~(iii)}
This result for anisotropic inclusions is completely new to our knowledge.
As before, we start by constructing local extension operators.
For fixed $n$, in a suitable orthonormal frame, we can assume $I_n= B'_{\delta_n} \times (-L_n,L_n)$ for some $\delta_n,L_n>0$, and we then consider the following neighborhood of $I_n$,
\[T_n\,:=\,T_n' \times (-L_n-\tfrac12\delta_n\mu_n,L_n+\tfrac12\delta_n\mu_n),\qquad T_n'\,:=\,B'_{\delta_n(1+\frac12\mu_n)},\]
where we recall the notation $\mu_n=1\wedge\frac{\rho_n}{\delta_n}$ and $\rho_n=\min_{m:m\ne n}\dist(I_n,I_m)$. Note that by definition we have
\[T_n\cap T_m=\varnothing\quad\text{for all $n\ne m$}.\]
We shall construct an extension operator $P_n:H^1(T_n\setminus I_n)\to H^1(T_n)$ such that $P_nu=u$ in $T_n\setminus I_n$ and
\begin{equation}\label{eq:extension-cyl-todo}
\|\nabla P_nu\|_{L^2(T_n)}\,\lesssim\,\mu_n^{-1}\|\nabla u\|_{L^2(T_n\setminus I_n)},
\end{equation}
thus improving on the construction of the previous section (replacing indeed the factor $\nu_n^{-1}$ by $\mu_n^{-1}$).
Next, applying this extension locally around each inclusion and combining with the moment condition on the $\mu_n$'s as in~\cite[Lemma 8]{Zhikov-86}, the conclusion follows similarly as in the proof of item~(i); we leave the details to the reader.

\begin{figure}[hbtp]
\begin{tikzpicture}[line cap=round,line join=round,>=triangle 45,x=0.3cm,y=0.3cm]
\clip(-8,-8) rectangle (19,9.5);
\draw [rotate around={0.:(0.,6.)},line width=0.5pt,dotted] (0.,6.) ellipse (2. and 0.5);
\draw [rotate around={0.:(0.,-6.)},line width=0.5pt,dotted] (0.,-6.) ellipse (2. and 0.5);
\draw [line width=0.5pt,dotted] (-2.,6.)-- (-2.,-6.);
\draw [line width=0.5pt,dotted] (2.,6.)-- (2.,-6.);
\draw [rotate around={0.:(0.,6.5)},line width=0.5pt] (0.,6.5) ellipse (2.449489742783178 and 0.6123724356957945);
\draw [rotate around={0.:(0.,-6.5)},line width=0.5pt] (0.,-6.5) ellipse (2.449489742783178 and 0.6123724356957945);
\draw [rotate around={0.:(0.,4.5)},line width=0.5pt] (0.,4.5) ellipse (2.449489742783178 and 0.6123724356957945);
\draw [rotate around={0.:(0.,3.)},line width=0.5pt] (0.,3.) ellipse (2.449489742783178 and 0.6123724356957945);
\draw [rotate around={0.:(0.,-4.5)},line width=0.5pt] (0.,-4.5) ellipse (2.449489742783178 and 0.6123724356957945);
\draw [rotate around={0.:(0.,-3.)},line width=0.5pt] (0.,-3.) ellipse (2.449489742783178 and 0.6123724356957945);
\draw [line width=0.5pt] (-2.445264123169735,6.4640457165808565)-- (-2.4493272378405795,-6.492946268840901);
\draw [line width=0.5pt] (2.44458629397462,6.461271900292365)-- (2.4444358369569223,-6.53931717102672);
\draw [line width=0.5pt] (17,6.5)-- (17,6.);
\draw [line width=0.5pt] (17.,6.)-- (17.1,6.1);
\draw [line width=0.5pt] (17.,6.)-- (16.9,6.1);
\draw [line width=0.5pt] (17.,6.5)-- (16.9,6.4);
\draw [line width=0.5pt] (17.,6.5)-- (17.1,6.4);
\draw (17,7.3) node[anchor=north west] {$\frac{\mu_n}2$};
\draw [line width=0.5pt] (15.9,7.15)-- (16.4,7.15);
\draw [line width=0.5pt] (16.4,7.15)-- (16.3,7.25);
\draw [line width=0.5pt] (16.4,7.15)-- (16.3,7.05);
\draw [line width=0.5pt] (15.9,7.15)-- (16,7.25);
\draw [line width=0.5pt] (15.9,7.15)-- (16,7.05);
\draw (15.3,9.3) node[anchor=north west] {$\frac{\mu_n}2$};
\draw [line width=0.5pt] (-5.,6.)-- (-5.,-6.);
\draw [line width=0.5pt] (-5.,-6.)-- (-4.8,-5.8);
\draw [line width=0.5pt] (-5.,-6.)-- (-5.2,-5.8);
\draw [line width=0.5pt] (-5.,6.)-- (-5.2,5.8);
\draw [line width=0.5pt] (-5.,6.)-- (-4.8,5.8);
\draw (-7.5,1.) node[anchor=north west] {$L$};
\draw [line width=0.5pt] (-3.5,6.)-- (-3.5,3.);
\draw [line width=0.5pt] (-3.5,3.)-- (-3.3,3.2);
\draw [line width=0.5pt] (-3.5,3.)-- (-3.7,3.2);
\draw [line width=0.5pt] (-3.5,6.)-- (-3.3,5.8);
\draw [line width=0.5pt] (-3.5,6.)-- (-3.7,5.8);
\draw (-5,5.5) node[anchor=north west] {$1$};
\draw (2.4,-6.2) node[anchor=north west] {$T_n$};
\draw (-1,0.7461704453720828) node[anchor=north west] {$I_n$};
\draw [rotate around={0.:(7.,4.5)},line width=0.5pt] (7.,4.5) ellipse (2.449489742783178 and 0.6123724356957945);
\draw [rotate around={0.:(7.,-4.5)},line width=0.5pt] (7.,-4.5) ellipse (2.449489742783178 and 0.6123724356957945);
\draw [line width=0.5pt] (4.560837878595681,4.443826081726938)-- (4.55073429124636,-4.5082820888198825);
\draw [line width=0.5pt] (9.421926770979649,4.408392578080399)-- (9.432416497150065,-4.572175993408565);
\draw [rotate around={0.:(7.,4.5)},dotted,line width=0.5pt] (7.,4.5) ellipse (2. and 0.5);
\draw [rotate around={0.:(7.,-4.5)},dotted,line width=0.5pt] (7.,-4.5) ellipse (2. and 0.5);
\draw [line width=0.5pt,dotted] (5.008842105594642,4.45303575858484)-- (5.000347390839616,-4.509318807186508);
\draw [line width=0.5pt,dotted] (8.98451028911472,4.437891771681497)-- (8.975190294642031,-4.5785108670630725);

\draw[line width=0.5pt,dashed] (2.4444358369569223,4.5)-- (4.55073429124636,4.5);
\draw[line width=0.5pt,dashed] (2.4444358369569223,-4.5)-- (4.55073429124636,-4.5);
\draw[line width=0.5pt,dashed] (2.4444358369569223,3)-- (4.55073429124636,3);
\draw[line width=0.5pt,dashed] (9.421926770979649,3)-- (11.550572523443755,3);
\draw[line width=0.5pt,dashed] (2.4444358369569223,6.5)-- (11.550572523443755,6.5);

\draw (6.,1.2241967693328335) node[anchor=north west] {};
\draw (9.2,-3.7) node[anchor=north west] {$T_{n,\circ}$};

\draw [rotate around={0.:(14.,3.)},line width=0.5pt] (14.,3.) ellipse (2.449489742783178 and 0.6123724356957945);
\draw [rotate around={0.:(14.,3.)},line width=0.5pt,dotted] (14.,3.) ellipse (2. and 0.5);
\draw [rotate around={0.:(14.,6.)},line width=0.5pt,dotted] (14.,6.) ellipse (2. and 0.5);
\draw [rotate around={0.:(14.,6.5)},line width=0.5pt] (14.,6.5) ellipse (2.449489742783178 and 0.6123724356957945);
\draw [line width=0.5pt,dotted] (12.,6.)-- (12.,3.);
\draw [line width=0.5pt,dotted] (16.,6.)-- (16.,3.);
\draw [line width=0.5pt] (16.445023427922244,6.4630368114611345)-- (16.444193010609403,2.959750367324543);
\draw [line width=0.5pt] (11.550572523443755,6.49563367222873)-- (11.55053360067129,3.002673464040804);
\draw (13.2,5.5) node[anchor=north west] {};
\draw (16.3,3.5) node[anchor=north west] {$T_{n,+}$};
\end{tikzpicture}
\caption{}
\label{fig:cylinders}
\end{figure}

We now focus on the construction of the local extension operator $P_n$ around $I_n$.
If $\delta_n\ge L_n$, then $\diam(I_n)\simeq \delta_n$ and thus $\mu_n\simeq\nu_n$, so that the construction of the desired extension operator already follows from~\eqref{eq:extension-munu} in the proof of~(ii). Henceforth, we may thus assume
\[\delta_n\le L_n.\]
In addition, by homogeneity, up to a dilation, we can assume without loss of generality
\[\delta_n=1.\]
Let us decompose the neighborhood $T_n$ of $I_n$ into three overlapping pieces, $T_n= T_{n,+}\cup  T_{n,-}\cup T_{n,\circ}$,
\begin{equation*}
T_{n,+}\,:=\,T_n'\times(L_n-1,L_n),\quad T_{n,-}\,:=\,T_n'\times(-L_n,-L_n+1),\quad T_{n,\circ}\,:=\,T_n'\times(-L_n+\tfrac12,L_n-\tfrac12),
\end{equation*}
where we recall $T_n'=B'_{1+\mu_n/2}$; see Figure~\ref{fig:cylinders}.
We start by constructing an extension operator on the middle piece in the decomposition and argue by dimension reduction.
By the construction~\eqref{eq:extension-munu} in the proof of~(ii) (now applied to $B_1'\subset T_n'\subset\R^{d-1}$), we get a linear extension operator $P_{n,\circ}':H^1(T_n'\setminus B_1')\to H^1(T_n')$ such that $P'_{n,\circ} u=u$ in~$T_n'\setminus B_1'$ and
\[
\int_{T_n'}|P_{n,\circ}'u|^2\,\lesssim\,\mu_n^{-1}\int_{T_n'\setminus B_1'}|u|^2,\qquad
\int_{T_n'}|\nabla P_{n,\circ}'u|^2\,\lesssim\,\mu_n^{-1}\int_{T_n'\setminus B_1'}|\nabla u|^2.\]
For $u\in H^1(T_{n,\circ}\setminus I_n)$, we then define
\[(P_{n,\circ}u)(x',x_d)\,:=\,P_{n,\circ}'(u(\cdot,x_d))(x')~\in~H^1(T_{n,\circ}),\]
By the properties of $P_{n,\circ}'$ (including the $L^2$-control), it satisfies $P_{n,\circ}u=u$ in $T_{n,\circ}\setminus I_n$ and
\[\int_{T_{n,\circ}}|\nabla P_{n,\circ} u|^2\,\lesssim\,\mu_n^{-1}\int_{T_{n,\circ}\setminus I_n}|\nabla u|^2.\]
Next, we turn to extension operators around both ends of the cylinder.
By a straightforward adaptation of the same construction in the proof of~(ii), we can find extension operators $P_{n,\pm}:H^1(T_{n,\pm}\setminus I_n)\to H^1(T_{n,\pm})$ such that $P_{n,\pm}u=u$ in $T_{n,\pm}\setminus I_n$ and
\[\int_{T_{n,\pm}}|\nabla P_{n,\pm}u|^2\,\lesssim\,\mu_n^{-1}\int_{T_{n,\pm}\setminus I_n}|\nabla u|^2;\]
we skip the details for shortness.
It remains to glue together these three extensions. For that purpose, we choose a partition of unity $\chi_{n,\pm},\chi_{n,\circ} \in C^\infty_b(T_n;[0,1])$ such that $\chi_{n,+} + \chi_{n,-} + \chi_{n,\circ}=1$ in $T_n$, $\chi_{n,\pm}=0$ outside $T_{n,\pm}$, $\chi_{n,\circ}=0$ outside $T_{n,\circ}$, and $\| \nabla \chi_{n,\pm}\|_{L^\infty},\| \nabla \chi_{n,\circ}\|_{L^\infty}\lesssim1$ (see Figure~\ref{fig:cylinders}). We then consider the extension operator $P_n : H^1(T_n \setminus I_n) \to H^1(T_n)$ given by
\[P_nu \,:=\, \chi_{n,+} P_{n,+}(u|_{T_{n,+}\setminus I_n}) + \chi_{n,-} P_{n,-}(u|_{T_{n,-}\setminus I_n}) +\chi_{n,\circ} P_{n,\circ}(u|_{T_{n,\circ}\setminus I_n}),\]
which satisfies by construction $P_nu = u$ in $T_n \setminus I_n$.
In addition, by the properties of the partition of unity, we can estimate
\begin{multline*}
\int_{T_n}|\nabla P_nu|^2
\,\lesssim\, \int_{T_{n,+}} |\nabla P_{n,+}u|^2+\int_{T_{n,-}} |\nabla P_{n,-}u|^2+\int_{T_{n,\circ}} |\nabla P_{n,\circ}u|^2\\
+\int_{T_n'\times(L-1,L-\frac12)} |P_{n,+} u - P_{n,\circ} u|^2
+ \int_{T_n'\times(-L+\frac12,-L+1)} |P_{n,-} u - P_{n,\circ} u|^2.
\end{multline*}
Recalling that $P_{n,+} u = u = P_{n,\circ} u $ on $(\partial B'_1) \times (L-1,L-\frac12)$, and similarly $P_{n,-} u =u= P_{n,\circ} u$ on $(\partial B'_1)\times(-L+\tfrac12, - L+1)$, Poincar\'e's inequality allows to control the last two terms by gradients. Together with the properties of $P_{n,\pm}$ and $P_{n,\circ}$, this leads us to
\begin{equation*}
\int_{T_n}|\nabla P_nu|^2
\,\lesssim\, \int_{T_{n,+}} |\nabla P_{n,+}u|^2+\int_{T_{n,-}} |\nabla P_{n,-}u|^2+\int_{T_{n,\circ}} |\nabla P_{n,\circ}u|^2
\,\lesssim\,\mu_n^{-1}\int_{T_n\setminus I_n} |\nabla u|^2,
\end{equation*}
that is, \eqref{eq:extension-cyl-todo}.\qed

\subsubsection{Proof of~(iv)}
This is the generalization of a result due to Zhikov~\cite{Zhikov-90b} for dense cubic packings of unit balls (see also~\cite[Lemma 3.14]{JKO94}).
Some work is needed to adapt it the more general setting of~(iv).
We split the proof into four steps.

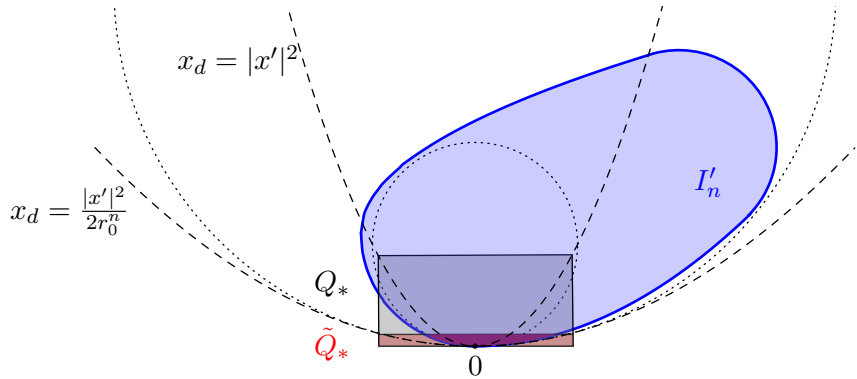
\begin{figure}[hbtp]
\begin{tikzpicture}[line cap=round,line join=round,>=triangle 45,x=2.5cm,y=2.5cm]
\clip(-2.5,-0.3) rectangle (2.5,1.8);
\fill[line width=1.pt,color=red,fill=red,fill opacity=0.4] (-0.5062638905320481,0.) -- (-0.5066168304292887,0.06207276319529555) -- (0.5159344902053536,0.06758346820240614) -- (0.5166962905850788,0.) -- cycle;
\draw [line width=0.5pt,dotted] (0.,0.54) circle (0.54);

\draw[name path=G,line width=1.pt,color=blue,smooth,samples=100,domain=-0.6:0.94] plot(\x,{((((\x)+0.6)/1.7)^(0.5)+0.6)});
\draw[name path=H,line width=1.pt,color=blue,smooth,samples=100,domain=-0.6:-0.] plot(\x,{(0.6-(0.36-(\x)^(2))^(0.5))});
\draw[name path=A,line width=1.pt,color=blue,smooth,samples=100,domain=-0.24:0.] plot(\x,{(0.6-(0.36-(\x)^(2))^(0.5))});
\draw [name path=E,line width=0.5pt] (-0.5089917843483605,0.4797641417298664)-- (0.5112405029524542,0.4840126886279227);
\draw [name path=F,line width=0.5pt] (0.5166962905850788,0.)-- (-0.5062638905320481,0.);
\draw [line width=0.5pt] (-0.5066168304292887,0.06207276319529555)-- (-0.24,0.06207276319529555);
\draw [name path=B,line width=0.5pt] (-0.24,0.06207276319529555)-- (0.,0.06207276319529555);
\draw[name path=c,line width=1pt,color=blue,smooth,samples=100,domain=0.:1.45] plot(\x,{(\x)^(2)/3});
\draw[name path=C,line width=1pt,color=blue,smooth,samples=100,domain=0:0.43] plot(\x,{(\x)^(2)/3});
\draw [name path=D,line width=0.5pt] (0,0.06207276319529555)-- (0.43,0.06207276319529555);
\draw [line width=0.5pt] (0.43,0.06207276319529555)-- (0.5159344902053536,0.06207276319529555);
\draw [name path=I,line width=1.pt,color=blue] (0.9391,1.552) arc (105:-45:0.515);

\tikzfillbetween[of=A and B]{purple, opacity=1};
\tikzfillbetween[of=C and D]{purple, opacity=2};
\tikzfillbetween[of=E and F]{gray, opacity=0.4};

\tikzfillbetween[of=G and H]{blue, opacity=0.2};
\tikzfillbetween[of=c and I]{blue, opacity=0.2};

\draw [line width=0.5pt,dotted] (0.,1.9) circle (1.9);
\draw[line width=0.5pt,dashed,smooth,samples=100,domain=-2:2.] plot(\x,{(\x)^(2)/(2*1.9)});
\draw[line width=0.5pt,dashed,smooth,samples=100,domain=-1.2:1.2] plot(\x,{(\x)^(2.0)/0.54});

\draw (-1.62,1.65) node[anchor=north west] {$x_d=|x'|^2$};
\draw (-2.5,0.9) node[anchor=north west] {$x_d=\frac{\lvert x' \rvert^2}{2r_0^n}$};

\draw [line width=0.5pt] (0.5112405029524542,0.4840126886279227)-- (0.5166962905850788,0.);

\draw [line width=0.5pt] (-0.5062638905320481,0.)-- (-0.5089917843483605,0.4797641417298664);

\draw (-0.09,0) node[anchor=north west] {$0$};
\draw (1.1,1.) node[anchor=north west] {\textcolor{blue}{$I_n'$}};
\draw (-0.9,0.45) node[anchor=north west] {$Q_*$};
\draw (-0.9,0.15) node[anchor=north west] {\textcolor{red}{$\tilde Q_*$}};

\fill (0.,0) circle (1pt);
\end{tikzpicture}
\caption{This illustrates the construction of local neighborhoods at a boundary point, say $x_*=0\in\partial I_n'$.}
\label{fig:contact}
\end{figure}

\medskip
\step1 Construction of polygonal neighborhoods.\\
For all $n$, consider the rescaled inclusion $I_n':=(r_1^n)^{-1}I_n$ and set $r_0^n:={r_2^n}/{r_1^n}\in[1,C_0]$.
As displayed in Figure~\ref{fig:contact}, the geometric assumptions on the inclusions imply that
for all $n$ and any boundary point $x_* \in \partial I_n$, there is an orthonormal system of coordinates $(x',x_d)\in\R^{d-1}\times\R$ and a strictly convex $C^2$ function
\[\phi_n :  B'(0,\tfrac12):= \{ x' \in \R^{d-1}:|x'|< \tfrac12 \} \to \R\]
such that
\begin{gather}
r_0^n - \sqrt{(r_0^n)^2 - \lvert x' \rvert^2} \,\le\, \phi_n(x') \,\le\, 1 - \sqrt{1-\lvert x' \rvert^2},\qquad \text{for all } x' \in B'(0,\tfrac12),\nonumber\\[2mm]
I_n' \cap \big(x_*+B'(0,\tfrac12)\times(-\tfrac12,\tfrac12)\big) \,=\, x_* + \big\{ (x',x_d)  \, : \, \phi_n(x') < x_d < \tfrac12, \, \lvert x'\rvert < \tfrac12\big\},\label{eq:phi_n}\\[2mm]
\partial I_n' \cap  \big(x_*+B'(0,\tfrac12)\times(-\tfrac12,\tfrac12)\big) \,=\, x_* + \big\{ (x', \phi_n(x')) \, : \,  \lvert x'\rvert < \tfrac12\big\}.\nonumber
\end{gather}
Note that $\phi_n(0)=0$, $\nabla \phi_n(0) = 0$, and 
\begin{equation}\label{eq:quadratic bounds phi_n}
\tfrac{1}{2 r_0^n} \lvert x' \rvert^2 \leq \phi_n(x') \leq \lvert x'\rvert^2,\qquad\text{for all $x' \in B'(0,\tfrac12)$}.
\end{equation}
In addition, note that assumption~\eqref{eq:closesmallincl} entails
\[\sup_n \sharp \big\{ m \, : \, \dist ( I_n,I_m) \leq \tfrac1{C_0}r_1^n \big\} \,\lesssim_{C_0}\,1.\]
Considering the set $S_{n}'\subset \partial I_n'$ of ``almost contact points'' of the inclusion $I'_n$,
which we define as
\[S_n':=(r_1^n)^{-1}S_n,\qquad S_{n} \,:=\, \big\{ x \in \partial I_n \, : \, \exists m \neq n , \, \dist(I_n,I_m)=\dist(x,I_m) \leq \tfrac1{C_0}r_1^n \big\},\]
we have
\[\sup_n \sharp S_{n}' \,\lesssim_{C_0}\,1,
\qquad\inf_n \inf_{x,y\in S_n'\atop x\ne y}|x-y|\,\gtrsim_{C_0}\,1.\]
For all $x_*\in S_{n}'$, consider the open half-plane $\Pi^{+}_n(x_*) \subset \R^d$ tangent to~$I_n'$ at $x_*$ and containing $I_n'$. By construction and convexity, the set
\[T_n'\,:=\, \bigcap_{x_* \in S'_{n}} \Pi^+_n(x_*)\]
is a convex polytope with $\sharp S'_n$ facets such that
\[I_n'\subset T_n'\qquad\text{and}\qquad\partial T_n' \cap \partial I_n' = S'_{n}.\]
Note that these neighborhoods $\{T_n'\}_n$ have a priori no reason to be pairwise disjoint and may have sharp angles (thus threatening their uniform Lipschitz regularity).
However, adding a finite number of additional boundary points to the set $S_{n}'$ (only depending on $d,C_0$) allows us to define polytopes $T_n$ with the following properties:
\begin{enumerate}[---]
\item $|T_n|\lesssim_{C_0} |I_n|$;
\item for all $n\ne m$ we have $T_n\cap T_m=\varnothing$;
\smallskip\item for all $n$, angles between neighboring facets of $T_n$ are $\ge\frac\pi2$ (say);
\smallskip\item $r_*:=\frac15\inf_n\inf_{x,y\in S_n',x\ne y}|x-y|$ satisfies\footnote{For the upper bound, it is enough to add enough points to $S_n'$.} $1 \lesssim_{C_0} r_* \le 1$.
\end{enumerate}
In terms of the length $r_*$ defined in this last item, let us now introduce suitable neighborhoods of almost contact points: for all $x_*\in S'_n$, we set
\begin{eqnarray*}
\tilde Q_*(x_*)\,:=\,x_*+B'(0,r_*)\times(0,\tfrac{r_*^2}{2r_0^n})
~~\subset~~
Q_*(x_*)\,:=\,x_*+B'(0,r_*)\times(0,r_*^2),
\end{eqnarray*}
and we then define the reduced inclusions
\[\tilde I_n'\,:=\, I_n'\setminus\bigcup_{x_*\in S'_{n}}\tilde Q_*(x_*).\]
For all $x_*,y_*\in S'_{n}$ with $x_*\ne y_*$, the condition $|x_*-y_*|\ge5r_*$ entails $\dist(Q_*(x_*),Q_*(y_*))\ge r_*$ (recall that $r_*\le 1$).
By construction, reduced inclusions $\{\tilde I_n'\}_n$ are thus uniformly separated
and satisfy the different assumptions of item~(i) for some constant only depending on $d,C_0$.

\medskip
\step2 Local extensions around rescaled inclusions.\\
In order to construct a linear extension operator $H^1(T_n'\setminus I_n')\to W^{1,1}(T_n')$ around each rescaled inclusion,
we start by constructing corresponding extension operators $H^1(Q_*(x_*)\setminus I_n')\to W^{1,1}(\tilde Q_*(x_*))$ around each almost contact point $x_*\in S'_{n}$.
Let $x_*\in S'_{n}$ be fixed, and
recall~\eqref{eq:phi_n} and the orthonormal coordinates $(x',x_d)\in\R^{d-1}\times\R$ associated with $\phi_n$.
Up to a translation, we can assume $x_* = 0$ and then set $\tilde Q_*:= \tilde Q_*(0)$, $Q_*:=Q_*(0)$.
In this setting, consider the $C^2$ diffeomorphism
$\psi_n : B'(0,r_*)\times\R \to \R^{d}$ (onto its image) given by
\[\psi_n(x',x_d) := \Big( \tfrac{1}{r_*} x' , \tfrac{2r_0^n}{r_*^2} \big( x_d - \phi_n(x') + \tfrac{1}{2 r_0^n}|x'|^2\big) \Big).\]
In terms of the `model' sets
\begin{equation*}
K~:=~\{(y',y_d):|y'|^2<y_d<1\}\quad\subset\quad \{(y',y_d):|y'|<1,\,0<y_d<1\}~=:~Q,
\end{equation*}
the properties in~\eqref{eq:phi_n} and~\eqref{eq:quadratic bounds phi_n}, together with the definition of $\tilde Q_*,Q_*$, entail
\begin{equation*}
\psi_n(I_n'\cap \tilde Q_*)\,\subset\,K\,\subset\,\psi_n(I_n'\cap Q_*),
\qquad Q\setminus K\,\subset\,\psi_n(Q_*\setminus I_n').
\end{equation*}
By a trivial generalization of~\cite[Lemma 3.14]{JKO94} to the vectorial setting $d \geq 2$, one can construct a linear extension operator $P_0:H^1(Q\setminus K)\to W^{1,1}(Q)$ such that $P_0u=u$ in $Q\setminus K$ and for all $1 \leq p < 2\tfrac{d+1}{d+3}$,
\[\| \nabla P_0 u\|_{L^p(K)} \,\lesssim_p\, \| \nabla u \|_{L^2(Q \setminus K)}.\]
Hence, for $u \in H^1(Q_*\setminus I_n')$, we have $u\circ\psi_n^{-1}|_{Q\setminus K}\in H^1(Q\setminus K)$ and we can define
\[P_{n,0}u\,:=\,P_0(u\circ\psi_n^{-1}|_{Q\setminus K})\circ\psi_n|_{I_n'\cap \tilde Q_*}\,\in\, W^{1,1}(I_n'\cap \tilde Q_*).\]
As by definition $P_{n,0}u=u$ on $\tilde Q_*\cap\partial I_n'$, we can extend $P_{n,0}u$ by $u$ on $\tilde Q_*\setminus I_n'$, thus defining an element~$P_{n,0}u\in W^{1,1}(\tilde Q_*)$.
This defines a linear extension operator $P_{n,0}:H^1(Q_*\setminus I_n')\to W^{1,1}(\tilde Q_*)$ such that $P_{n,0}u=u$ in $\tilde Q_*\setminus I_n'$ and for all $1\le p<2\frac{d+1}{d+3}$,
\[\|\nabla P_{n,0} w \|_{\Ld^p(I_n'\cap \tilde Q_*)}\,\lesssim_{C_0}\,\|\nabla P_0(w\circ\psi_n^{-1}|_{Q\setminus K})\|_{\Ld^p(K)}\,\lesssim_p\, \|w\circ\psi_n^{-1}\|_{\Ld^2(Q\setminus K)}
\,\lesssim_{C_0}\, \|w\|_{\Ld^2(Q_*\setminus I_n')},\]
where we note that the Lipschitz norms of $\psi_n$ and $\psi_n^{-1}$ are bounded only depending on $d,C_0$.
Combining this construction around every almost contact point $x_*\in S'_{n}$, we are led to a linear extension operator $P_n':H^1(T_n'\setminus I_n')\to W^{1,1}(T_n'\setminus\tilde I_n')$ such that $P_n'u=u$ in $T_n'\setminus I_n'$ and for all $1\le p<2\frac{d+1}{d+3}$,
\begin{equation}\label{eq:prop-Pn}
\|\nabla P_n'u\|_{L^p(T_n'\setminus\tilde I_n')}\,\lesssim_{C_0,p}\,\|\nabla u\|_{L^2(T_n'\setminus I_n')}.
\end{equation}
Next, given that the reduced inclusions $\{\tilde I_n'\}_n$ satisfy the uniform separation and regularity requirements of item~(i) for some constant only depending only on $d,C_0$, we can apply the result of Step~1 of the proof of item~(i), which actually also holds on $L^p$ instead of $L^2$ for any $1\le p\le2$ (see~\cite[Chapter~VI, Theorem~5]{Stein-70}): there exists a linear extension operator $P_n'':W^{1,1}(T_n'\setminus\tilde I_n')\to W^{1,1}(T_n')$ such that $P_n''u=u$ in $T_n'\setminus\tilde I_n'$ and for all $1\le p\le2$,
\[\|\nabla P_n''u\|_{L^p(T_n')}\,\lesssim_{C_0}\,\|\nabla u\|_{L^p(T_n'\setminus\tilde I_n')}.\]
Hence, combining this with~\eqref{eq:prop-Pn}, the composition $P_n:=P_n''\circ P_n'$ defines a linear extension operator $H^1(T_n'\setminus I_n')\to W^{1,1}(T_n')$ such that $P_nu=u$ in $T_n'\setminus I_n'$ and, for all $1\le p<2\frac{d+1}{d+3}$,
\[\|\nabla P_nu\|_{L^p(T_n')}\,\lesssim_{C_0}\,\|\nabla u\|_{L^2(T_n'\setminus I_n')}.\]

\medskip
\step3 Conclusion --- part~1.\\
We show the validity of the first part~\eqref{eq:ass1-1} of Assumption~\ref{ass1} for all $1\le p<2\frac{d+1}{d+3}$.
For all $n$, consider the rescaled operator $P_n^\e:H^1(\e(T_n\setminus I_n))\to W^{1,1}(\e T_n):u\mapsto P_n(u(\e r_1^n\cdot))(\frac\cdot{\e r_1^n})$, which satisfies $P_n^\e u=u$ in $\e(T_n\setminus I_n)$ and, by homogeneity,
\begin{equation*}
\|\nabla P_n^\e u\|_{L^p(\e T_n)}
\,\lesssim_{C_0,p}\,(\e r_1^n)^{d(\frac1p-\frac12)}\|\nabla u\|_{L^2(\e(T_n\setminus I_n))}\,\lesssim_{C_0}\,|\e I_n|^{\frac1p-\frac12}\|\nabla u\|_{L^2(\e(T_n\setminus I_n))}.
\end{equation*}
Given some ball $B\subset\R^d$, we apply this local extension around each inclusion~$\e I_n$ intersecting~$B$.
Provided $0<\e<\diam(B)$, arguing similarly as in Step~2 of the proof of item~(i), we can pretend that for each inclusion~$\e I_n$ intersecting $B$ the neighborhood $\e T_n$ is included in $2B$.
Given $u\in H^1_0(B)$, extending it by $0$ on $2B\setminus B$, recalling that the neighborhoods $\{T_n\}_n$ are pairwise disjoint, and setting for abbreviation 
\[ F_\e (B)~:=~ \bigcup_{n:\e I_n \cap B \neq\varnothing} \e  I_n~~\subset~~\bigcup_{n:\e I_n \cap B \neq\varnothing} \e  T_n~=:~T_\e (B)~~\subset~~2B,\]
we may then define
\[P_{B,\e}u\,:=\,u\mathds1_{B\setminus T_\e(B)}+\sum_{n:\e I_n\cap B\neq\varnothing}(P_n^\e u)\mathds1_{\e T_n}~\in~W^{1,1}_0(2B).\]
By definition, this satisfies $P_{B,\e}u=u$ in $B\setminus F_\e(B)$ and for all $1\le p<2\frac{d+1}{d+3}$,
\begin{equation*}
\begin{array}{lll}
\|\nabla P_{B,\e} u\|_{L^p(  2B)}^p
&\le&\displaystyle\|\nabla u\|_{L^p(B\setminus T_\e(B))}^p+\sum_{n:\e I_n\cap B\ne\varnothing}\|\nabla P_n^\e u\|_{L^p(\e T_n)}^p\\[6mm]
&\lesssim_{C_0,p}&\displaystyle\|\nabla u\|_{L^p(B\setminus T_\e(B))}^p+\sum_{n:\e I_n\cap B\ne\varnothing}|\e I_n|^{1-\frac p2}\|\nabla u\|_{L^2(\e (T_n\setminus I_n))}^p\\[6mm]
&\lesssim&\displaystyle|B|^{1-\frac{p}2}\|\nabla u\|_{L^2(B\setminus F_\e(B))}^p,
\end{array}
\end{equation*}
which is the desired estimate~\eqref{eq:ass1-1}.

\medskip
\step4 Conclusion --- part~2.\\
We turn to the validity of the second part~\eqref{eq:ass1-2} of Assumption~\ref{ass1} for all $1\le p<2\frac{d+1}{d+3}$.
In terms of the extension operators $\{P_n\}_n$ constructed in Step~2, we define $P:H^1_\loc(\R^d)\to W^{1,1}_\loc(\R^d)$ as follows,
\[Pu\,:=\,u\mathds1_{\R^d\setminus\cup_nT_n}+\sum_nP_n\big(u(r_1^n\cdot)|_{T_n'\setminus I_n'}\big)(\tfrac\cdot{r_1^n})\,\mathds1_{T_n}~\in~W^{1,1}_\loc(\R^d).\]
By definition, $Pu=u$ in $\R^d\setminus\cup_nI_n$. As in the proof of item~(i), we can ensure that $Pu$ is a stationary field whenever the pair $(u,\{I_n\}_n)$ is jointly stationary, and it remains to check that it satisfies the desired estimate~\eqref{eq:ass1-2}.
Given $u\in L^2(\Omega;H^1_\loc(\R^d))$ such that $(u,\{I_n\}_n)$ is jointly stationary,
using~\eqref{eq:expand-expect}, we can decompose
\[\E[|\nabla Pu|^p]\,=\,\E[|\nabla Pu|^p\mathds1_{\R^d\setminus\cup_nT_n}]+\E\bigg[\sum_n\frac{\mathds1_{0\in I_n}}{|I_n|}\int_{T_n}|\nabla Pu|^p\bigg].\]
By definition of $Pu$ together with the properties of the local extensions $\{P_n\}_n$, we get for $1\le p<2\frac{d+1}{d+3}$,
\[\E[|\nabla Pu|^p]\,\lesssim_{C_0,p}\,\E[|\nabla u|^p\mathds1_{\R^d\setminus\cup_nT_n}]+\E\bigg[\sum_n\frac{\mathds1_{0\in I_n}}{|I_n|}|T_n|^{1-\frac p2}\Big(\int_{T_n\setminus I_n}|\nabla u|^2\Big)^\frac p2\bigg],\]
hence, by Jensen's, H\"older's inequalities, and the property $|T_n|\lesssim |I_n|$,
\begin{eqnarray*}
\E[|\nabla Pu|^p]&\lesssim_{C_0,p}&\E[|\nabla u|^2\mathds1_{\R^d\setminus\cup_nT_n}]^\frac p2+\E\bigg[\sum_n\frac{\mathds1_{0\in I_n}}{|I_n|}|T_n|\bigg]^{1-\frac p2}\E\bigg[\sum_n\frac{\mathds1_{0\in I_n}}{|I_n|}\int_{T_n\setminus I_n}|\nabla u|^2\bigg]^\frac p2
\\
&\lesssim_{C_0,p}&\E[|\nabla u|^2\mathds1_{\R^d\setminus\cup_nT_n}]^\frac p2+ \E\bigg[\sum_n\frac{\mathds1_{0\in I_n}}{|I_n|}\int_{T_n\setminus I_n}|\nabla u|^2\bigg]^\frac p2
\end{eqnarray*}
Using~\eqref{eq:expand-expect} again, this yields for all $1\le p<2\frac{d+1}{d+3}$,
\[\|\nabla Pu\|_{L^p(\Omega)}\,\lesssim\,\|\mathds1_{\R^d\setminus\cup_nI_n}\nabla u\|_{L^2(\Omega)},\]
which is the desired estimate~\eqref{eq:ass1-2}.
\qed

\subsubsection{Proof of~(v)}
The validity of the first part of~\ref{ass1} is due to Zhikov~\cite[Appendix]{Zhikov-90}, and it can be combined with similar arguments as above to deduce the second part of~\ref{ass1}. We leave the details to the reader.\qed

\subsection{Proof of Lemma~\ref{lem:ass2}}
We split the discussion into two parts, considering items~(i) and~(ii) separately. While $L^p$ regularity questions are classical, some care is needed here to check the precise dependence on the regularity of the inclusions, so as to ensure a uniform statement on the whole collection $\{I_n\}_n$ of inclusions.

\subsubsection{Proof of~(i)}
We split the proof into two steps, and consider the cases $\diam(I_n)\ge1$ and $\diam(I_n)\le1$ separately.

\medskip
\step1 Case $\diam(I_n)\ge1$.\\
Given some unit ball~$B$ with $I_n\cap B\ne\varnothing$ and given $g\in L^\infty(I_n\cap 2B)^d$, consider some $u\in H^1_0(I_n)$ satisfying
\begin{equation}\label{eq:ug}
u-\triangle u=\Div(g),\qquad\text{in $I_n\cap2B$}.
\end{equation}
As $\diam(I_n)\ge1$, we note that the uniform $C^1$ condition for $I_n'=\diam(I_n)^{-1}I_n$ necessarily holds a fortiori for $I_n$ itself (with the same constant $C_0$ and modulus of continuity $\omega$).
Let then $\{D_i^n\}_i$ be a collection of balls covering $\partial I_n$ satisfying the properties of the uniform $C^1$ condition in the statement, and note that without loss of generality we may further assume $\diam(D_i^n)\le\frac23$ (say).
Also note that the geometric assumptions ensure that for all $i$ we have
\begin{equation}\label{eq:geom}
|I_n\cap D_i^n|\,\gtrsim_{C_0}\,|D_i^n|.
\end{equation}
For some constant $C_1\ge C_0$ only depending on $d,C_0$, we can complement this collection with a collection of balls $\{B_i^n\}_i$ covering~$(I_n\cap B)\setminus\cup_iD_i^n$ such that, for all $i$,
\begin{equation*}
\tfrac32B_i^n\,\subset\, I_n,
\qquad
\diam(B_i^n)\,\in\,[C_1^{-1},\tfrac23],\qquad
\sup_i\sharp\{j:B_i^n\cap B_j^n\ne\varnothing\}\,\le\, C_1.
\end{equation*}
Note that the upper bound on diameters ensures that for any $D_i^n$ (resp.~$B_i^n$) with $D_i^n\cap B\ne\varnothing$ (resp.~$B_i^n\cap B\ne\varnothing$) we have $\frac32D_i^n\subset2B$ (resp.~$\frac32B_i^n\subset2B$).
By our geometric assumptions, applying standard~$L^p$ regularity theory to the solution $u$ of~\eqref{eq:ug}, we get the following estimates (see e.g.~\cite{Gilbarg-Trudinger-01}): for all $2\le q<\infty$, we have on each `boundary' ball~$D_i^n$ with $D_i^n\cap B\ne\varnothing$,
\begin{equation}\label{eq:bndaryCZ}
\|\nabla u\|_{L^q(I_n\cap D_i^n)}\,\lesssim_{q,C_0}\,\|g\|_{L^q(I_n\cap\frac32D_i^n)}+|D_i^n|^{\frac1q-\frac12}\|\nabla u\|_{L^2(I_n\cap \frac32D_i^n)},
\end{equation}
and similarly, on each `interior' ball $B_i^n$ with $B_i^n\cap B\ne\varnothing$,
\begin{equation}\label{eq:interiorCZ}
\|\nabla u\|_{L^q(I_n\cap B_i^n)}\,\lesssim_{q,C_0}\,\|g\|_{L^q(I_n\cap\frac32B_i^n)}+|B_i^n|^{\frac1q-\frac12}\|\nabla u\|_{L^2(I_n\cap \frac32B_i^n)},
\end{equation}
where the multiplicative constants only depend on $d,q,C_0,\omega$.
Summing the above estimates, recalling that the balls $\{D_i^n\}_i$ and $\{B_i^n\}_i$ have diameters $\ge C_1^{-1}$,
and using the $\ell^q-\ell^2$ inequality, we can deduce
\begin{equation}\label{eq:ug-estimreg}
\|\nabla u\|_{L^q(I_n\cap B)}\,\lesssim_{q,C_0}\,\|g\|_{L^q(I_n\cap 2B)}+\|\nabla u\|_{L^2(I_n\cap 2B)},
\end{equation}
thus proving the validity of Assumption~\ref{ass2} for all $2\le q<\infty$.

\medskip
\step2 Case $\diam(I_n)\le1$.\\
In this case, we note that for a unit ball $B$ with $I_n\cap B\ne \varnothing$ we necessarily have $I_n\subset2B$.
Hence, Assumption~\ref{ass2} amounts to the following: given $g\in L^q(I_n)^d$, if $u\in H^1_0(I_n)$ is the weak solution of
\begin{equation}\label{eq:ug-re}
u-\triangle u=\Div(g),\qquad\text{in $I_n$},
\end{equation}
then we have
\begin{equation}\label{eq:ug-concl}
\|\nabla u\|_{L^q(I_n)}\,\lesssim_q\,\|g\|_{L^q(I_n)}.
\end{equation}
For that purpose, let us first consider the rescaled inclusion $I_n':=D_n^{-1}I_n$ with $D_n:=\diam(I_n)\le1$. Given $g\in L^\infty(I_n)^d$, let $g':=g(D_n\cdot)\in L^\infty(I_n')^d$ and consider the weak solution $u'\in H^1_0(I_n')$ of the rescaled equation
\begin{equation}\label{eq:ug-re-resc}
D_n^2u'-\triangle u'=\Div(g'),\qquad\text{in $I_n'$}.
\end{equation}
Repeating the argument of Step~1 for this rescaled equation, and noting that it holds independently of the size of the zeroth-order term, we find for all $2\le q<\infty$, for any unit ball $B$ with $I_n'\cap B\ne\varnothing$,
\[\|\nabla u'\|_{L^q(I_n'\cap B)}\,\lesssim_q\,\|g'\|_{L^q(I_n'\cap 2B)}+\|\nabla u'\|_{L^2(I_n'\cap2B)}.\]
As the rescaled inclusion satisfies $\diam(I_n')=1$, the condition $I_n'\cap B\ne\varnothing$ implies $I_n'\subset2B$. Hence, the above yields for all $2\le q<\infty$,
\[\|\nabla u'\|_{L^q(I_n')}\,\lesssim_q\,\|g'\|_{L^q(I_n')}+\|\nabla u'\|_{L^2(I_n')}.\]
Now note that the energy estimate for~\eqref{eq:ug-re-resc}, together with Jensen's inequality, yields for any $q\ge2$,
\[\|\nabla u'\|_{L^2(I_n')}\,\le\,\|g'\|_{L^2(I_n')}\,\le\,|I_n'|^{\frac12-\frac1q}\|g'\|_{L^q(I_n')}\,\lesssim\,\|g'\|_{L^q(I_n')}.\]
Combined with the above, this entails for all $2\le q<\infty$,
\[\|\nabla u'\|_{L^q(I_n')}\,\lesssim\,\|g'\|_{L^q(I_n')},\]
and the conclusion~\eqref{eq:ug-concl} then follows by scaling.

\subsubsection{Proof of~(ii)}
If the inclusions $I_n$'s only have Lipschitz boundary, it is well known that Calder\'on--Zygmund estimates might in general fail to hold in $L^q$ for some $2\le q<\infty$. Yet, it has been shown by Jerison and Kenig~\cite[Theorem~1.1]{Jerison-Kenig-95} that there exists some $q_0>3$ (or $q_0>4$ if $d=2$), only depending on the Lipschitz constant of the domain, such that Calder\'on--Zygmund estimates hold in $L^q$ for all $2\le q\le q_0$. Although only stated in a global form in~\cite{Jerison-Kenig-95}, such estimates can be checked to hold in the localized form that we need in~\eqref{eq:bndaryCZ} \&~\eqref{eq:interiorCZ}, cf.~\cite{Jerison-perso}.
In the special case of a convex polygonal domain (or of $C^1$ deformations thereof), the estimates are actually known to hold in~$L^q$ for all $2\le q<\infty$ (see e.g.~\cite{MR911445} and~\cite[Section~4.3.1]{MR2641539}).
Then proceeding as for item~(i), the conclusion follows. 
\qed

\section*{Data availability statement}
There is no associated data with this article.

\section*{Acknowledgements}
The authors are grateful to Igor Vel\v{c}i\'c and Valery Smyshlyaev for their suggestions which helped substantially improve the results of this manuscript.
EB and AG acknowledge financial support from the European Research Council (ERC) under the European Union's Horizon 2020 research and innovation programme (Grant Agreement n$^\circ$864066).
MD acknowledges financial support from the European Union (ERC, PASTIS, Grant Agreement n$^\circ$101075879).\footnote{Views and opinions expressed are however those of the authors only and do not necessarily reflect those of the European Union or the European Research Council Executive Agency. Neither the European Union nor the granting authority can be held responsible for them.}

\bibliographystyle{abbrv}
\bibliography{biblio}

\def\cprime{$'$} \def\cprime{$'$} \def\cprime{$'$}
\begin{thebibliography}{10}

\bibitem{Allaire-92}
G.~Allaire.
\newblock Homogenization and two-scale convergence.
\newblock {\em SIAM J. Math. Anal.}, 23(6):1482--1518, 1992.

\bibitem{Arbogast}
T.~Arbogast, J.~Douglas, Jr., and U.~Hornung.
\newblock Derivation of the double porosity model of single phase flow via
  homogenization theory.
\newblock {\em SIAM J. Math. Anal.}, 21(4):823--836, 1990.

\bibitem{AKM-book}
S.~Armstrong, T.~Kuusi, and J.-C. Mourrat.
\newblock {\em Quantitative stochastic homogenization and large-scale
  regularity}, volume 352 of {\em Grundlehren der Mathematischen
  Wissenschaften}.
\newblock Springer, Cham, 2019.

\bibitem{bella2025}
P.~Bella, M.~Capoferri, M.~Cherdantsev, and I.~Vel\v{c}i\'c.
\newblock Quantitative estimates for high-contrast random media.
\newblock Preprint, arXiv:2502.09493.

\bibitem{bernou2023homogenization}
A.~Bernou, M.~Duerinckx, and A.~Gloria.
\newblock Homogenization of active suspensions and reduction of effective
  viscosity.
\newblock Preprint, arXiv:2301.00166.

\bibitem{Bouchitte-15}
G.~Bouchitt\'{e}, C.~Bourel, and L.~Manca.
\newblock Resonant effects in random dielectric structures.
\newblock {\em ESAIM Control Optim. Calc. Var.}, 21(1):217--246, 2015.

\bibitem{MR1993376}
A.~Bourgeat, A.~Mikeli\'{c}, and A.~Piatnitski.
\newblock On the double porosity model of a single phase flow in random media.
\newblock {\em Asymptot. Anal.}, 34(3-4):311--332, 2003.

\bibitem{Braides-Piatnitski-04}
A.~Braides, V.~Chiad\`o~Piat, and A.~Piatnitski.
\newblock A variational approach to double-porosity problems.
\newblock {\em Asymptot. Anal.}, 39(3-4):281--308, 2004.

\bibitem{CCV-23}
M.~Capoferri, M.~Cherdantsev, and I.~Vel\v{c}i\'{c}.
\newblock Eigenfunctions localised on a defect in high-contrast random media.
\newblock {\em SIAM J. Math. Anal.}, 55(6):7449--7489, 2023.

\bibitem{CCV-21}
M.~Cherdantsev, K.~Cherednichenko, and I.~Vel\v{c}i\'{c}.
\newblock High-contrast random composites: homogenisation framework and new
  spectral phenomena.
\newblock Preprint, arXiv:2110.00395.

\bibitem{MR3902123}
M.~Cherdantsev, K.~Cherednichenko, and I.~Vel\v{c}i\'{c}.
\newblock Stochastic homogenisation of high-contrast media.
\newblock {\em Appl. Anal.}, 98(1-2):91--117, 2019.

\bibitem{Cherednichenko-Cooper-16}
K.~D. Cherednichenko and S.~Cooper.
\newblock Resolvent estimates for high-contrast elliptic problems with periodic
  coefficients.
\newblock {\em Arch. Ration. Mech. Anal.}, 219(3):1061--1086, 2016.

\bibitem{Cherednichenko-Ershova-Kiselev-20}
K.~D. Cherednichenko, Y.~Y. Ershova, and A.~V. Kiselev.
\newblock Effective behaviour of critical-contrast {PDE}s: micro-resonances,
  frequency conversion, and time dispersive properties. {I}.
\newblock {\em Comm. Math. Phys.}, 375(3):1833--1884, 2020.

\bibitem{Smyshlyaev}
S.~Cooper, I.~Kamotski, and V.~P. Smyshlyaev.
\newblock Quantitative multiscale operator-type approximations for
  asymptotically degenerating spectral problems.
\newblock Preprint, arXiv:2307.13151.

\bibitem{DG-22pol}
M.~Duerinckx and A.~Gloria.
\newblock Quantitative homogenization theory for random suspensions in steady
  {S}tokes flow.
\newblock {\em J. \'{E}c. polytech. Math.}, 9:1183--1244, 2022.

\bibitem{DG-23}
M.~Duerinckx and A.~Gloria.
\newblock {\em On {E}instein's effective viscosity formula}, volume~7 of {\em
  Memoirs of the European Mathematical Society}.
\newblock EMS Press, Berlin, 2023.

\bibitem{DCRT-20}
H.~Duminil-Copin, A.~Raoufi, and V.~Tassion.
\newblock Subcritical phase of {$d$}-dimensional {P}oisson-{B}oolean
  percolation and its vacant set.
\newblock {\em Ann. H. Lebesgue}, 3:677--700, 2020.

\bibitem{Gilbarg-Trudinger-01}
D.~Gilbarg and N.~S. Trudinger.
\newblock {\em Elliptic partial differential equations of second order}.
\newblock Classics in Mathematics. Springer-Verlag, Berlin, 2001.
\newblock Reprint of the 1998 edition.

\bibitem{MR4103433}
A.~Gloria, S.~Neukamm, and F.~Otto.
\newblock A regularity theory for random elliptic operators.
\newblock {\em Milan J. Math.}, 88(1):99--170, 2020.

\bibitem{GNO-quant}
A.~Gloria, S.~Neukamm, and F.~Otto.
\newblock Quantitative estimates in stochastic homogenization for correlated
  fields.
\newblock {\em Anal. PDE}, 14(8):2497--2537, 2021.

\bibitem{GO15}
A.~Gloria and F.~Otto.
\newblock The corrector in stochastic homogenization: optimal rates, stochastic
  integrability, and fluctuations.
\newblock Preprint, arXiv:1510.08290.

\bibitem{MR911445}
P.~Grisvard.
\newblock Le probl\`eme de {D}irichlet dans l'espace {$W^1_p$}.
\newblock {\em Portugal. Math.}, 43(4):393--398, 1985/86.

\bibitem{Heida-23}
M.~Heida.
\newblock Stochastic homogenization on perforated domains? - {E}xtension
  operators.
\newblock {\em Networks and Heterogeneous Media}, 18:1820--1897, 01 2023.

\bibitem{Hornung-97}
U.~Hornung, editor.
\newblock {\em Homogenization and porous media}, volume~6 of {\em
  Interdisciplinary Applied Mathematics}.
\newblock Springer-Verlag, New York, 1997.

\bibitem{Jerison-perso}
D.~Jerison.
\newblock Personal communication, 2024.

\bibitem{Jerison-Kenig-95}
D.~Jerison and C.~E. Kenig.
\newblock The inhomogeneous {D}irichlet problem in {L}ipschitz domains.
\newblock {\em J. Funct. Anal.}, 130(1):161--219, 1995.

\bibitem{JKO94}
V.~V. Jikov, S.~M. Kozlov, and O.~A. Ole{\u\i}nik.
\newblock {\em Homogenization of differential operators and integral
  functionals}.
\newblock Springer-Verlag, Berlin, 1994.
\newblock Traduit du russe par G. A. Iosif{\cprime}yan.

\bibitem{Smyshlyaev1}
I.~V. Kamotski and V.~P. Smyshlyaev.
\newblock Localized modes due to defects in high contrast periodic media via
  two-scale homogenization.
\newblock {\em J. Math. Sci. (N.Y.)}, 232(3):349--377, 2018.

\bibitem{Smyshlyaev0}
I.~V. Kamotski and V.~P. Smyshlyaev.
\newblock Two-scale homogenization for a general class of high contrast {PDE}
  systems with periodic coefficients.
\newblock {\em Appl. Anal.}, 98(1-2):64--90, 2019.

\bibitem{MR2641539}
V.~Maz'ya and J.~Rossmann.
\newblock {\em Elliptic equations in polyhedral domains}, volume 162 of {\em
  Mathematical Surveys and Monographs}.
\newblock American Mathematical Society, Providence, RI, 2010.

\bibitem{Stein-70}
E.~M. Stein.
\newblock {\em Singular integrals and differentiability properties of
  functions}.
\newblock Princeton Mathematical Series, No. 30. Princeton University Press,
  Princeton, N.J., 1970.

\bibitem{Zhikov-86}
V.~V. Zhikov.
\newblock Averaging of functionals of the calculus of variations and elasticity
  theory.
\newblock {\em Izv. Akad. Nauk SSSR Ser. Mat.}, 50(4):675--710, 877, 1986.

\bibitem{Zhikov-90}
V.~V. Zhikov.
\newblock Asymptotic problems connected with the heat equation in perforated
  domains.
\newblock {\em Mat. Sb.}, 181(10):1283--1305, 1990.

\bibitem{Zhikov-90b}
V.~V. Zhikov.
\newblock Problems of the continuation of functions in connection with
  averaging theory.
\newblock {\em Differentsial\cprime nye Uravneniya}, 26(1):39--50, 181, 1990.

\bibitem{Zhikov-00}
V.~V. Zhikov.
\newblock On an extension and an application of the two-scale convergence
  method.
\newblock {\em Mat. Sb.}, 191(7):31--72, 2000.

\bibitem{Zhikov-0}
V.~V. Zhikov.
\newblock Gaps in the spectrum of some elliptic operators in divergent form
  with periodic coefficients.
\newblock {\em Algebra i Analiz}, 16(5):34--58, 2004.

\end{thebibliography}

\end{document}